\documentclass[11pt,reqno]{amsart}
\usepackage{graphicx,color}
\usepackage{amssymb,amscd,latexsym, mathrsfs, epsfig}
\usepackage[all]{xy}
\usepackage{todonotes}

\title[Stability, connections and curves]{Stability data, irregular connections\\ and tropical curves}
\date{}

\author[S. A. Filippini]{Sara A. Filippini}
 \address{Institut de Math\'ematiques de Marseille, Technop\^ole Ch\^ateau-Gombert 39, rue Fr\'ed\'eric Joliot-Curie, 13453 Marseille Cedex 13, France}
\email{sara.filippini@univ-amu.fr}
\author[M. Garcia-Fernandez]{Mario Garcia-Fernandez}
\address{Instituto de Ciencias Matem\'aticas (CSIC-UAM-UC3M-UCM), Nicol\'as Cabrera 13-15, Cantoblanco, 28049 Madrid, Spain}
  \email{mario.garcia@icmat.es}
\author[J. Stoppa]{Jacopo Stoppa}
  \address{SISSA, Via Bonomea 265, 34136 Trieste, Italia}
  \email{jstoppa@sissa.it}
  
\thanks{}

\usepackage{hyperref,url}
\usepackage{amssymb,latexsym}
\usepackage{amsthm,amssymb,amscd}
\usepackage[latin1]{inputenc}
\usepackage{enumerate}
\usepackage{todonotes}
\usepackage[all]{xy} \CompileMatrices
\SelectTips{cm}{12} 

\usepackage{verbatim} 

\theoremstyle{plain}
\newtheorem{theorem}{Theorem}[section]
\newtheorem{lem}[theorem]{Lemma}
\newtheorem{corollary}[theorem]{Corollary}
\newtheorem{proposition}[theorem]{Proposition}

\theoremstyle{definition}
\newtheorem{definition}[theorem]{Definition}
\newtheorem{definition-theorem}[theorem]{Definition-Theorem}
\newtheorem{example}[theorem]{Example}

\theoremstyle{remark}
\newtheorem{remark}[theorem]{Remark}

\numberwithin{equation}{section} 

\newcommand\PP{\mathbb P}
\newcommand\C{\mathbb C}
\newcommand\Q{\mathbb Q}
\newcommand\R{\mathbb R}
\newcommand\Z{\mathbb Z}
\newcommand\N{\mathbb N}

\newcommand{\z}{z}
\newcommand{\eps}{\varepsilon}

\newcommand\Hom{\operatorname{Hom}}
\newcommand\Aut{\operatorname{Aut}}
\renewcommand\Im{\operatorname{Im}}
\renewcommand\Re{\operatorname{Re}}

\newcommand{\bra}{\langle}
\newcommand{\ket}{\rangle}
\newcommand{\wed}{\wedge}
\newcommand{\del}{\partial}

\newcommand{\A}{\mathcal{A}}

\newcommand{\G}{G}

\newcommand{\T}{T}
\newcommand{\W}{W}

\renewcommand{\H}{H}

\newcommand{\dt}{\operatorname{DT}}

\newcommand{\Stab}{\operatorname{Stab}}

\newcommand{\g}{\mathfrak{g}}

\makeatletter \@addtoreset{equation}{section} \makeatother

\newcommand{\CC}{{\mathbb C}}

\newcommand{\RR}{{\mathbb R}}

\newcommand{\cA}{\mathcal{A}}

\newcommand{\End}{\operatorname{End}}
\newcommand{\ad}{\operatorname{ad}}
\newcommand{\Ad}{\operatorname{Ad}}

\begin{document}

\begin{abstract}  We study a class of meromorphic connections $\nabla(Z)$ on $\PP^1$, para\-metrised by the central charge $Z$ of a stability condition, with values in a Lie algebra of formal vector fields on a torus. Their definition is motivated by the work of Gaiotto, Moore and Neitzke on wall-crossing and three-dimensional field theories. Our main results concern two limits of the families $\nabla(Z)$ as we rescale the central charge $Z \mapsto RZ$. In the $R \to 0$ ``conformal limit'' we recover a version of the connections introduced by Bridgeland and Toledano Laredo (and so the Joyce holomorphic generating functions for enumerative invariants), although with a different construction yielding new explicit formulae. In the $R \to \infty$ ``large complex structure" limit the connections $\nabla(Z)$ make contact with the Gross-Pandharipande-Siebert approach to wall-crossing based on tropical geometry. Their flat sections display tropical behaviour, and also encode certain tropical/relative Gromov-Witten invariants. 
 
\end{abstract} 

\maketitle 
\setcounter{tocdepth}{1}
 
\tableofcontents

\section{Introduction}\label{sec:intro}

Bridgeland-Toledano Laredo connections and Joyce holomorphic generating functions are important geometric objects in the theory of stability conditions. Given a finite length abelian category $\mathcal{C}$ the paper \cite{bt_stab} introduces a holomorphic family of meromorphic connections on $\PP^1$, parametrized by stability conditions on $\mathcal{C}$, of the form
\begin{equation}\label{btlConn}
\nabla^{BTL}(Z) = d - \left(\frac{Z}{t^2} + \frac{f(Z)}{t}\right)dt
\end{equation}  
where $Z$ is the central charge defining the stability condition and $t$ denotes a local coordinate on $\PP^1$. The connections $\nabla^{BTL}(Z)$ take values in an extension of the Ringel-Hall (Lie) algebra of $\mathcal{C}$ by central charges. The residue $f(Z)$ is the Joyce holomorphic generating function for invariants ``counting" stable objects in $\mathcal{C}$, introduced previously in \cite{joyHolo}. As we move the central charge $Z$ these invariants change discontinuously (\emph{wall-crossing}), but $\nabla^{BTL}(Z)$ varies holomorphically and its generalised monodromy remains constant (see Theorem \ref{BTLthm} for a precise statement). This approach builds on a large body of previous work, including \cite{Boalch2, ks, rein}.

Very similar ideas appeared independently in the physical literature, in the work of Gaiotto, Moore and Neitzke \cite{gmn}. The corresponding connections in physics are attached to a class of $\mathcal{N} = 2$ four-dimensonal gauge theories, and take the form
\begin{equation}\label{gmnConn}
\nabla^{GMN}(Z) = d - \left(\frac{\A^{(-1)}(Z)}{z^2} + \frac{\A^{(0)}(Z)}{z} + \A^{(1)}(Z)\right)dz
\end{equation}  
where $Z$ is the central charge of the theory and $z$ is a local coordinate on $\PP^1$. These connections take values in the Lie algebra of complex vector fields on a compact torus $(S^1)^{2 r}$, and depend on a physical counterpart of the counting invariants, the spectrum of BPS states of the theory. As we move the parameters of the theory, and so the central charge, the  spectrum of BPS states changes discontinuously, but $\nabla^{GMN}(Z)$ varies real-analytically and the generalised monodromy is constant. This approach is formally analogous to the classical work of Cecotti and Vafa \cite{vafa1} and Dubrovin \cite{dubrovin}, but is largely conjectural at the present time.

The main idea of the present paper is to use the point of view and techniques introduced in \cite{gmn} to get new insights on Bridgeland-Toledano Laredo connections, their deformations, and most importantly their relation to another aspect of wall-crossing, namely tropical geometry and the method introduced by Gross, Pandharipande and Siebert \cite{gps}. We do not assume familiarity with \cite{gmn} or with physical ideas. Indeed even if our treatment is motivated by \cite{gmn}, the basic idea of studying families of connections like \eqref{btlConn} by trying to embed them in families like \eqref{gmnConn} is well known and forms part of the theory of Frobenius type and CV-structures on arbitrary bundles due to C. Hertling \cite{hert}. 

We will work in an abstract setting, that of \emph{continuous families of stability data} with values in the (infinite-dimensional) Kontsevich-Soibelman Poisson Lie algebra $\g$. These were introduced in \cite{ks}. The algebra $\g$ is attached to a lattice $\Gamma$ with a skew-symmetric bilinear form and can be seen as a Poisson algebra of functions on a complex affine torus. The lattice $\Gamma$ is known as the \emph{charge lattice}, and the central charge is an element of $\Hom(\Gamma, \C)$. The connections considered in this paper take values in the Lie algebra $D^*(\widehat{\g})$ of derivations of a completion of $\g$ (a Lie algebra of formal vector fields on a torus).

In more concrete terms we work simply with a collection of rational numbers $\Omega(\gamma, Z)$ for $\gamma \in \Gamma$ which are constant in the central charge $Z$ in strata of $\Hom(\Gamma, \C)$, and satisfy the \emph{Kontsevich-Soibelman wall-crossing formula} across different strata (a collection of \emph{``Donaldson-Thomas invariants"}, or a \emph{BPS spectrum}). Note that, at least for suitable categories $\mathcal{C}$, the Bridgeland-Toledano Laredo connections $\nabla^{BTL}(Z)$ of \cite{bt_stab} specialize to $D^*(\widehat{\g})$-valued connections through a Lie algebra morphism, known as an integration map. On the other hand, from the point of view of \cite{gmn}, we are replacing complex vector fields with their formal analogue $D^*(\widehat{\g})$. Note also that a direct construction of a $D^*(\widehat{\g})$-valued Bridgeland-Toledano Laredo connection starting from a spectrum $\Omega(\gamma, Z)$, without passing through an integration map, is implicit in \cite{bt_stab, bt_stokes} and in \cite{ks} Section 2.8. Our  results may be summarized as follows.\\

\noindent\textbf{Preliminary existence result (Theorem \ref{prop:GMNconstruction}).} Given a spectrum $\Omega(\gamma, Z)$ as above, under a positivity assumption (Definition \ref{positivity}), we construct a corresponding family of connections $\nabla(Z)$ of the form \eqref{gmnConn}, parametrised by the central charge $Z$, with values in the Lie algebra of derivations $D^*(\widehat{\g})$. The coefficients $\A^{(i)}(Z)$ are real analytic functions of $Z$ and the Kontsevich-Soibelman wall-crossing formula becomes the statement that the resulting $\nabla(Z)$ have constant generalized monodromy on $\PP^1$, just as in \cite{bt_stab} and \cite{gmn}. Moreover, similarly to \cite{gmn}, the coefficients $\A^{(-1)}$ and $\A^{(1)}$ determine each other through a suitable symmetry (see \eqref{conjugation}). Our methods are different from those in \cite{bt_stab} and are very much inspired by the physical proposal of \cite{gmn}. The proof uses the positivity assumption on the spectrum, which implies that the infinite-dimensional Lie algebra $D^*(\widehat{\g})$ is pro-solvable (cf. \cite{bt_stab}). We discuss an extension beyond the positive case below. 

Theorem \ref{prop:GMNconstruction} is only a very special case in the vast literature on isomonodromy and it is probably quite standard for experts. What really matters for us is that the particular method of proof we follow allows us to obtain our main results, Theorems \ref{thm:tropical1}, \ref{thm:tropical2} and \ref{thm:limit}, which interpolate between Bridgeland-Toledano Laredo connections (and so Joyce functions) and tropical geometry (the results of Gross-Pandharipande-Siebert).\\

\noindent \textbf{Conformal limit (Theorem \ref{thm:limit}).} Considerations in \cite{gmn} suggest the study of a particular scaling limit of the connections $\nabla(Z)$. Namely one rescales the central charge by $Z \mapsto R Z$ while also rescaling the $\PP^1$ variable as $z = R t$ and letting $R \to 0$. This is called a \emph{conformal limit} in \cite{gaiotto}\footnote{In the physical setup the norm $| Z(\alpha) |$ is closely related to the energy of a state of charge $\alpha$. So invariance under $Z \mapsto R Z$ is related to scale invariance.}. The connections $\nabla(Z)$ are endowed with a natural, nontrivial, explicit \emph{frame} $g(Z) \in \Aut^*(\widehat{\g})$ (such that $g(Z)^{-1}\cdot \cA^{(-1)} = -Z$). Theorem \ref{thm:limit} shows that the rescaled connections $\nabla(R Z)_{z = R t}$ converge as $R \to 0$ to the a corresponding Bridge\-land-Toledano Laredo connection (of the form \eqref{btlConn}), in the specific gauge provided by $g(RZ)$. The frame is an interesting object in itself and encodes certain tropical enumerative invariants, as we explain below.

Theorem \ref{thm:limit} has some interesting implications. On the one hand we rediscover from this point of view the invariance property of Joyce functions $f(\lambda Z) = f(Z)$ (\emph{conformal invariance}) and provide new explicit formulae for the flat sections of the Bridgeland--Toledano Laredo connection. On the other hand the proof provides an alternative derivation of the inverse of the monodromy map in \cite{Boalch2} for suitable solvable groups (a first explicit description of this inverse follows from \cite{bt_stab,bt_stokes}).\\   

\noindent\textbf{New explicit formulae.} Theorems \ref{prop:GMNconstruction} and \ref{thm:limit} recover in particular the $D^*(\widehat{\g})$-valued Bridge\-land-Toledano Laredo connections with a different construction, as a scaling limit of $\nabla(Z)$. Although it is restricted to the Lie algebra of formal vector fields $D^*(\widehat{\g})$ (rather than some Ringel-Hall algebra specializing to it), our alternative construction applies to the more general type of connections \eqref{gmnConn} (with order $2$ poles both at $0$ and $\infty$), and yields new explicit formulae for the connections and their flat sections. The connection $\nabla(Z)$ is uniquely determined by a canonical system of local flat sections $X(z, Z)$, with values in $\Aut^*(\widehat{\g})$, given explicitly by a sum over graphs, 
\begin{equation}\label{intro:X}
X(z, Z) e_{\alpha} = e_{\alpha} \exp_*\left(z^{-1} Z(\alpha) + z \bar{Z}(\alpha) - \bra \alpha , \sum_{\T} W_{\T}(Z) G_{\T}(z, Z)\ket\right).
\end{equation}   
Here $\exp_*$ is defined via the standard series for the exponential map and the commutative product on $\g$. The $G_{\T}(z, Z)$ are a collection of $z$-holomorphic functions with branch-cuts, with values in $ \g $, and are graph integrals attached to a set of trees $\T$ decorated by elements of the lattice $\Gamma$. They should be compared to the multilogarithms appearing in \cite{bt_stab, bt_stokes}. The symbols $W_{\T}(Z) \in \Gamma \otimes \Q$ denote suitable combinatorial weights. We will give explicit formulae for $W_{\T}$ and $G_{\T}$ in \textbf{Section \ref{sec:GMN:sub:fixed}}. We will also find an explicit expression for the connection $\nabla(Z)$ (\textbf{Corollary \ref{cor:explicitA}}) and the frame $g(Z)$ (\textbf{Lemma \ref{lem:RHsolution}}). These results specialize to explicit formulae for $\nabla^{BTL}(Z)$ and its flat sections under Theorem \ref{thm:limit}, see \textbf{sections \ref{sec:BTLfixed}} and \textbf{\ref{sec:limit:sub:finish}}.\\

\noindent \textbf{Tropical limit (Theorems \ref{thm:tropical1} and \ref{thm:tropical2}).} The approach of \cite{joyHolo} and \cite{bt_stab} relates wall-crossing for counting invariants to meromorphic connections and their monodromy. There is an important alternative approach due to Gross, Pandharipande and Siebert based on tropical geometry \cite{gps}. It seems natural to ask if the two methods can be directly related. We will show that the connections $\nabla(Z)$ interpolate between the two. 

Indeed ideas from \cite{gmn} suggest to study another scaling limit, namely to look at $\nabla(R Z)$ as $R \to \infty$. This is not interesting for BTL connections due to the conformal invariance $f(R Z) = f(Z)$. On the contrary we show that in the case of their deformations $\nabla(Z)$ this limit is closely related to the geometry of rational tropical curves immersed in $\R^2$. This should be expected in the light of the physical construction \cite{gmn} and considerations from mirror symmetry: roughly speaking in the original physical and differential-geometric setup of \cite{gmn} the limit $R \to \infty$ should be mirror to a large complex structure limit, whose relation to tropical geometry is much studied (see e.g. \cite{gross}). 

Our results apply to the special local flat sections $X(z, R Z)$ of $\nabla(R Z)$. For simplicity we will only examine the model case when the charge lattice is $\Gamma \cong \Z^2$, generated by $\gamma, \eta$ with $\bra \gamma, \eta\ket = \kappa > 0$, with the simplest nontrivial stability data in a chamber $U^+$ in the parameter space $\mathcal{U}$. In other words we are looking at the $\kappa$-Kronecker quiver case and $U^+$ is the chamber is which only the simple objects are stable. Notice that according to \cite{ks} Section 1.4 the $\kappa$-Kronecker quiver case gives a sort of universal local model for wall-crossing formulae. The connections $\nabla(R Z)$ are far from trivial even for $\kappa = 1$ (the $A_2$ quiver case). Recall that the flat sections $X(z, R Z)$ are constructed using the special functions $G_{\T}(z, R Z)$ (the analogues of multilogarithms in the work of Bridgeland-Toledano Laredo). At a generic point $z^* \in \C^*$ the graph integrals $G_{\T}(z^*, R Z)$ have discontinuous jumping behaviour when $Z$ crosses the boundary $\del U^+$, while the flat section $X(z, R Z)$ is continuous across $\del U^+$. This will enable us to compute how the sum over graphs expansion \eqref{intro:X} changes across the critical locus $\del U^+$, and to relate this behaviour to tropical curves and invariants. We prove two main facts, stated as \textbf{Theorems \ref{thm:tropical1}} and \textbf{\ref{thm:tropical2}}:  
\begin{enumerate}
\item the graph integrals $G_{\T}(z^*, R Z)$ appearing in the expansion \eqref{intro:X} for flat sections of $\nabla(R Z )$ have a natural tropical interpretation for $R \to \infty$;
\item the sum of leading order $R \to \infty$ contributions over decorated graphs $T$ of the same tropical degree yields a tropical enumerative invariant.
\end{enumerate}
We provide here concise versions of the full statements given in Section \ref{sec:tropical} (where we also introduce the few notions from tropical geometry we actually need).
\begin{theorem}\label{thm:tropical1} As $Z$ crosses the boundary $\del U^+ \subset \mathcal{U}$ from the interior, a special function $G_{\T}(z^*, R Z)$ attached to a tree $T$ appearing in the sum over graphs expansion \eqref{intro:X} for the flat section $X(z^*, R Z)$ is replaced by a linear combination of the form 
\begin{equation} 
\sum_{T'} \pm G_{T'}(z^*, R Z ),
\end{equation}
where we sum over a finite set of trees $T'$ (not necessarily distinct). The terms corresponding to a single-vertex tree in the sum above are uniquely characterised by their asymptotic behaviour as $R \to \infty$. These leading order terms are in bijection with a finite set of weighted \emph{trivalent} graphs $C_i(T)$, which have a natural structure of \emph{combinatorial types of rational tropical curves} immersed in $\R^2$. 
\end{theorem}
\noindent The fully detailed statement is given in Section 5.1. Explicit examples are given in Section \ref{sec:tropical}, Figures 1, 2 and 3.

Each (type of a) tropical curve $C_i$ appears with a sign $\eps(C_i) = \pm 1$ and it turns out that this sign is determined by the residue theorem. Each tree $T$ appearing in the expansion for $X(z^*, R Z)$ in the chamber $U^+$ defines a pair of unordered partitions $\deg(\T)$, whose parts are positive integral multiples of the generators $\gamma$, $\eta$: these are simply the decorations of vertices. In the light of Theorem \ref{thm:tropical1} it is natural to identify the pair of partitions $\deg(T)$ with a tropical degree ${\bf w}$. The sum of all parts of ${\bf w}$ can be regarded naturally as an element $\beta({\bf w}) \in \Gamma$. 
\begin{theorem}\label{thm:tropical2}  The sum of contributions $\eps(C_i(T)) = \pm 1$ over tropical types $C_i$, weighted by the coefficients $W_T$ in the expansion \eqref{intro:X} for flat sections in $U^+$,
\begin{equation*}
\sum_{\deg(T) = {\bf w}} W_T \sum_i \eps(C_i(T)) 
\end{equation*}
equals a tropical invariant $N^{\rm trop}({\bf w})$ enumerating plane rational tropical curves, times a simple combinatorial (``multi-cover") factor $c_{\bf w} \in \Gamma \otimes \Q$.   
\end{theorem}
\noindent Tropical invariants are briefly recalled in Section 5.6. The fully detailed statement (including the expression of $c_{\bf w}$) is given in Section 5.7. For an explicit example see Section \ref{sec:tropical}, Figure 4. Theorem \ref{thm:tropical2} uses the  wall-crossing theory developed by Gross, Pandharipande and Siebert in \cite{gps} and is in fact equivalent to part of it. The enumerative invariants $N^{\rm trop}({\bf w})$ are in turn equivalent to certain relative Gromov-Witten invariants (as proved in \cite{gps}). 

Theorems \ref{thm:tropical1}, \ref{thm:tropical2} and their proofs have as an immediate consequence the following much weaker result.
\begin{corollary} Fix a generic $z^* \in \C^*$. As $Z \to \del U^+$ from the interior and $R \to \infty$ there are expansions for flat sections
\begin{align*}
&e^{- \frac{1}{z^*} R Z(\alpha) - z^* R \bar{Z}(\alpha)} X(z^*, R Z) e_{\alpha} \\
&\sim e_{\alpha} \exp_*\left< - \alpha , \sum_{\bf w} c_{\bf w} N^{\rm trop}({\bf w}) e_{\beta(\bf w)} \frac{e^{- |Z(\beta({\bf w}))|R}}{2|Z(\beta({\bf w}))|R} \right> 
\end{align*}
as well as for the gauge transformations (connection frames)
\begin{equation*}
g(R Z) e_{\alpha} \sim e_{\alpha} \exp_*\left< - \alpha , \sum_{\bf w} c_{\bf w} N^{\rm trop}({\bf w}) e_{\beta(\bf w)} \frac{e^{- |Z(\beta({\bf w}))|R}}{2|Z(\beta({\bf w}))|R} \right> 
\end{equation*}
(notice that this last statement is independent of $z^*$).
\end{corollary}

Our results give a precise meaning to the intuition that Joyce's generating functions and the Gross-Pandharipande-Siebert wall-crossing theory, involving tropical counts and Gro\-mov-Witten theory, appear when one expands a single geometric object (in our case, the connections $\nabla$) at different points in parameter space ($R\to 0$, respectively $R\to\infty$). When taking the large $R$ limit we use the specific form of flat sections of $\nabla(RZ)$, and we have not been able to find a similar tropical structure underlying the special functions (multilogarithms) used in \cite{bt_stab}.\\
 
\noindent \textbf{Symmetric spectrum}. The positivity assumption (Definition \ref{positivity}) says simply that $\Omega(\gamma, Z)$ vanishes unless $\gamma$ belongs to some strictly convex (``positive") cone. This is analogous to the restriction to abelian categories in \cite{bt_stab} (there is also a similar assumption for the scattering diagrams in \cite{gps}). Removing this positivity assumption is an important open problem first pointed out in \cite{joyHolo}, and is related to the problem of extending the results of \cite{bt_stab} to triangulated categories. It is especially important to allow a \emph{symmetric spectrum} for which $\Omega(\gamma, Z) = \Omega(-\gamma, Z)$ for all $\gamma \in \Gamma$. In the categorical case this symmetry is induced by the shift functor $[1]$. This symmetry is also crucial from the physical point of view of \cite{gmn}, reflecting a basic symmetry of the physical theory (``CPT conjugation").

We show that our existence and specialization results Theorems\ref{prop:GMNconstruction} and \ref{thm:limit} extend to the symmetric case whenever there is a suitable lift of the continuous family with values in $\g$ to the ring of formal power series $\g[[\tau]]$, a point of view inspired by \cite{gps} (see Definition \ref{lift}). In interesting examples (e.g. for sufficiently simple triangulated categories) such lifts always exist locally in $Z$. The $\Gamma$-graded components of $\A^{(i)}(Z)$ and of the Joyce function $f(Z)$ become formal power series in $\tau$ (in the positive case these would be just polynomials in $\tau$, and the lift is unnecessary). Their value at $\tau = 1$ is independent of the lift, modulo convergence. 

This convergence problem for Joyce functions in the symmetric case was first pointed out in \cite{joyHolo}. It is especially hard because of conformal invariance. In \cite{us} our results are used to solve the convergence problem for the matrix elements of the connection $1$-form of $\nabla(R Z)$ for sufficiently large $R$ in a special but nontrivial class of examples, thus showing that $\nabla(R Z)$ give a regularisation of the Brigdeland-Toledano Laredo connections.\\

\noindent \textbf{Plan of the paper}. Section \ref{sec:stabdata} contains some background material on continuous families of stability data, BPS spectra, and the Kontsevich-Soibelman (KS) wall-crossing formula. In Section \ref{sec:gmn}, after recalling the basic properties of the Bridgeland-Toledano Laredo connections in the formal context of this paper, we prove the preliminary existence result Theorem \ref{prop:GMNconstruction} for the connections $\nabla(Z)$. Section \ref{sec:limit} is devoted to a proof of the specialization result for $R \to 0$, Theorem \ref{thm:limit}. Finally Section \ref{sec:tropical} contains the precise statements and proofs of our results on the opposite scaling limit of $\nabla(Z)$ as $R \to \infty$ and its relation to tropical curves and the theory developed by Gross, Pandharipande and Siebert in \cite{gps} (Theorems \ref{thm:tropical1} and \ref{thm:tropical2}). We will recall the few basic notions about generalised monodromy and tropical curves and invariants which we actually need along
  the way in sections \ref{sec:gmn} and \ref{sec:tropical}. 
 
\subsection*{Acknowledgements} We thank Anna Barbieri, Gilberto Bini, Tom Bridgeland, Ben Davison, Tamas Hausel, Daniel Huybrechts, Alessia Mandini, Luca Migliorini, Vivek Shende, Ivan Smith, Tom Sutherland, Szilard Szabo and Michael Wong for helpful comments and discussions. We are grateful to an anonymous Referee for important corrections and suggestions. This research was partially supported by the \'Ecole Polytechnique F\'ed\'eral de Lausanne and the Hausdorff Institute for Mathematics, Bonn (through the Junior Hausdorff Research Trimester ``Mathematical Physics"). The research leading to these results has received funding from the European Research Council under the European Union's Seventh Framework Programme (FP7/2007-2013)/ERC Grant agreement no. 307119.

\section{Continuous families of stability data}\label{sec:stabdata}
The basic object underlying the constructions in this paper, a \emph{(numerical) BPS spectrum} (or a ``collection of numerical Donaldson-Thomas invariants"), is essentially elementary and combinatorial: it is a collection of rational numbers $\Omega(\gamma, Z)$ which are functions of a point $\gamma$ of a lattice $\Gamma$ and of the central charge $Z \in \Hom(\Gamma, \C)$, satisfying an interesting Lie-theoretic identity, the \emph{Kontsevich-Soibelman wall-crossing formula}. In this Section we recall this concept by placing it in the necessary context of stability data on graded Lie algebras. 
\subsection{Spaces of stability data} Fix a rank $n$ lattice $\Gamma$, and let $\g = \bigoplus_{\gamma \in \Gamma} \g_{\gamma}$ denote a general $\Gamma$-graded Lie algebra over $\Q$ (we will soon specialize to the Lie algebra which is relevant to describe a BPS spectrum $\Omega$). The \emph{space of stability data} $\Stab(\g)$ on $\g$ is a complex $n$-dimensional manifold introduced by M. Kontsevich and Y. Soibelman in \cite{ks}. Its points are given by pairs $(Z, a)$ were $Z\!: \Gamma \to \C$ is a group homomorphism and $a\!: \Gamma\setminus\{0\} \to \g$ is a map of sets which preserves the grading (that is, such that $a(\gamma) \in \g_{\gamma}$). 
Additionally one requires that $(Z, a)$ satisfy the \emph{support property}
\begin{equation}\label{eq:support}
|| \gamma || \leq C |Z(\gamma)|
\end{equation}
whenever $\gamma \in \textrm{Supp}(a)$, that is, $a(\gamma)\neq 0$, for some arbitrary norm on $\Gamma\otimes_{\Z}\C$ and some constant $C > 0$ (the condition does not depend on the choice of norm). In particular the set $\{Z(\gamma): \gamma\in \textrm{Supp}(a)\} \subset \C$ is discrete. In algebro-geometric applications $Z$ is the central charge for a stability condition on a category $\mathcal{C}$ with Grothendieck group $\Gamma$, and the element $a(\gamma) \in \g_{\gamma}$ corresponds to a ``count" of $Z$-semistable objects of class $\gamma$. The Lie algebra $\g$ is typically infinite-dimensional and the support property is a quantitative analogue of the fact that the central charge of a semistable object should not vanish. In general $\Stab(\g)$ has an uncountable number of connected components.
 
The topology on the space of stability data \cite[Section 2.3]{ks} is essentially characterised by two properties. The first is that the projection $(Z, a) \mapsto Z$ on $\Hom(\Gamma, \C)$ is a local homeomorphism (this is the reason for working over $\Q$). So \emph{locally} we can think of stability data as parametrised by an open subset $U \subset \Hom(\Gamma, \C)$, i.e. the algebra elements $a(\gamma)$ are (possibly multi-valued) functions of $Z$. The second property is that the \emph{Kontsevich-Soibelman wall-crossing formula} should hold. To formulate this we notice that for each strictly convex cone $V \subset \C^*$ and stability data $(Z, a)$ we can define a complete Lie algebra $\widehat{\g}_{V, (Z, a)}$ topologically generated by elements $a(\gamma)$ such that $Z(\gamma) \in V$. As long as $Z(\textrm{Supp}(a))$ does not cross $\del V$ these completions $\widehat{\g}_{V, Z}$ are all subalgebras of a single complete Lie algebra $\widehat{\g}_V$. Given stability data $(Z, a)$ and a ray $\ell$ with $-\ell \subset V$, we define a group element 
\begin{equation}\label{eq:stokesdata}
S_\ell(Z) = \exp\left(\sum_{\gamma \in \Gamma, Z(\gamma) \in -\ell} a(\gamma) \right) \in \exp(\widehat{\g}_{V, (Z, a)})
\end{equation}
(the odd looking minus sign here is chosen to fit with our sign conventions in Section \ref{sec:gmn}). If a family $(Z_t, a_t)$ of stability data on $\g$ parametrised by $[0, 1]$ is continuous at $t_0$, then the support condition \eqref{eq:support} holds uniformly in a neighborhood of $t_0$ and for any strictly convex cone $V \subset \C^*$ such that $Z(\textrm{Supp}(a_{t_0})) \cap \partial V = \emptyset$, the group element $S_V$ given by 
\begin{equation}\label{eq:wall-crossing}
S_V = \prod^{\longrightarrow}_{\ell \subset V} S_\ell(Z_{t}) \in \widehat{\g}_V 
\end{equation}
is constant in this neighborhood. Here $\prod^{\to}$ denotes the slope-ordered product, writing the operators from left to right according to the clockwise order of rays $\ell$. This condition is known as the \emph{Kontsevich-Soibelman wall-crossing formula} and plays a central role in the theory of Donaldson-Thomas invariants.

\subsection{The Kontsevich-Soibelman Poisson Lie algebra and the Lie algebra of derivations}\label{sec:gdef} The graded Lie algebras which are most relevant for the present paper are infinite-dimensional ones introduced by Kontsevich and Soibelman in \cite[Section 2.5]{ks} and \cite[Section 10.1]{ks2}. They may be thought of, respectively, as a Poisson algebra of functions on a complex affine algebraic torus and a Lie algebra of complex vector fields (not necessarily preserving the Poisson structure). 

Let $\Gamma$ denote a lattice of finite rank $n$ endowed with an integral, skew-symmetric bilinear form $\bra - , - \ket$. In the rest of the paper we write $\g$ for the infinite-dimensional complex Lie algebra generated by symbols $e_{\gamma}$ for $\gamma \in \Gamma$, with bracket 
\begin{equation*}
[e_{\gamma}, e_{\eta}] = \bra \gamma, \eta\ket e_{\gamma + \eta}.  
\end{equation*}
We can also define a \emph{commutative} product $*$ on $\g$ simply by 
$e_{\gamma} * e_{\eta} = e_{\gamma + \eta}$ (i.e. $(\g, *)$ is the group algebra $\C[\Gamma]$). The product $*$ turns $\g$ into a Poisson algebra, i.e. the Lie bracket acts as a derivation.

\begin{remark} A different version of $\g$ which is often used has bracket
\begin{equation*}
[e_{\gamma}, e_{\eta}] = (-1)^{\bra \gamma, \eta\ket} \bra \gamma, \eta\ket e_{\gamma + \eta}.  
\end{equation*}
These two versions are isomorphic (non-canonically). Also notice that $\g$ is really defined over $\Z$, so in particular integral and rational elements of $\g$ are well defined.
\end{remark}

The second graded Lie algebra relevant for the present paper is the $\g$-module of derivations of $\g$ as a commutative, associative algebra, and will be denoted by $D^*(\g)$. Explicitly the Lie algebra of derivations $D^*(\g)$ is spanned by elements $e_\gamma \partial_\mu$ where $\gamma \in \Gamma$, $\mu \in \Gamma^\vee$ satisfying the
linear relations
\begin{equation*}
e_\gamma \partial_{\mu_1} + e_\gamma \partial_{\mu_2} = e_\gamma \partial_{\mu_1 + \mu_2}.
\end{equation*}
The commutator rule is given by
\begin{equation*}
[e_{\gamma_1} \partial_{\mu_1},e_{\gamma_2} \partial_{\mu_2}] = e_{\gamma_1+\gamma_2}(\mu_1(\gamma_2)\partial_{\mu_1} - \mu_2(\gamma_1) \partial_{\mu_2}).
\end{equation*}
We note that the definition of this bracket does not require the bilinear form on $\Gamma$. The role of $\bra - , - \ket$ is to define a Lie algebra morphism
\begin{equation}\label{eq:Liemorp}
\g \to D^*(\g) \colon e_\gamma \mapsto \ad(e_\gamma) = e_\gamma \partial_{\bra \gamma , - \ket}.
\end{equation}
The image of \eqref{eq:Liemorp} singles out a Lie subalgebra of \emph{divergence-free derivations}, spanned by elements $e_\gamma \partial_\mu$ with $\bra \mu,\gamma\ket = 0$. The dimension of each $\Gamma$-graded component this Lie subalgebra is $\operatorname{rank}(\Gamma) - 1$. The morphism \eqref{eq:Liemorp} extends naturally to the semi-direct product $\Hom(\Gamma, \C) \ltimes \g$ with crossed relations $[Z,e_\gamma] = Z(\gamma)e_\gamma$.

The three Lie algebras considered are used at different stages of our construction. The generalised monodromy (Stokes data) for all the connections in this work will be defined in terms of (a completion of) $\g$, while the construction of the connections $\nabla(Z)$ requires the use of $D^*(\g)$ (due to our particular method of proof, inspired by \cite{gmn}). The Lie algebra $\Hom(\Gamma, \C) \ltimes \g$ is preferred by the methods of \cite{bt_stab}, and will be implicitly used for the uniqueness of the conformal limit in Section \ref{sec:gmn:sub:finish}.

\subsection{Positive stability data}\label{sec:positivestab} In this paper we will work most of the time with a fixed amenable subalgebra $\mathfrak{g}_{\geqslant 0} \subset \g$ and its completion $\widehat{\g}$. The algebra $\mathfrak{g}_{\geqslant 0}$ is the analogue in our setup of the bialgebra in \cite[4.3]{bt_stab}. Fix a strictly convex cone $\Gamma_{\geqslant 0} \subset \Gamma$. Let $\g_{\geqslant 0} \subset \g$ be the Poisson Lie subalgebra 
\begin{equation*}
\g_{\geqslant 0} = \langle e_\gamma : \gamma \in \Gamma_{\geqslant 0} \rangle \subset \g
\end{equation*}
and note that $\g_{\geqslant 0}$ is graded by the semi-group $\Gamma_{\geqslant 0}$. For each $k \geqslant 1$, we denote by $\Gamma_{>k} \subset \Gamma_{\geqslant 0}$ the cone generated by elements $\gamma_1 + \ldots + \gamma_m$ for $m > k$ and $\gamma_i \in \Gamma_{\geqslant 0}$, $\gamma_j \neq 0$ for all $j = 1, \ldots, m$. The subspace $\g_{> k} \subset \g_{\geqslant 0}$ induced by $\Gamma_{>k}$ is an ideal. Consider the finite-dimensional nilpotent Lie algebra 
\begin{equation*}
\g_{\leqslant k} = \g_{\geqslant 0} / \g_{>k}.
\end{equation*}
The $\g_{\leqslant k}$ form an inverse system, and we define
\begin{equation*}
\widehat \g = \lim_{\longleftarrow} \g_{\leqslant k}.   
\end{equation*}
Similarly we define a pro-Lie group $\widehat G$ with Lie algebra $\widehat \g$ by
\begin{equation*}
\widehat G = \lim_{\longleftarrow} G_{\leqslant k}
\end{equation*}
where $G_{\leqslant k} = \exp(\g_{\leqslant k})$ is the Lie group with Lie algebra $\g_{\leqslant k}$. 
Note that $\widehat \g$ inherits a natural structure of Poisson Lie algebra, and it is possible to define in a standard way the commutative algebra $\exp_*$ and $\log_*$ maps which are each other's inverse.

Throughout the paper we will write $\Aut^*(\widehat{\g})$ for the group of automorphisms of $\widehat{\g}$ as a commutative algebra (we do not require that these preserve the Lie bracket). Similarly we always write $D^*(\widehat{\g})$ for the $\widehat{\g}$-module of derivations of $\widehat{\g}$ as a commutative, associative algebra. We will often forget the notation $*$ and simply write $e_{\gamma} e_{\eta}$ for the product, but at some points it will be important to have a special symbol for it to avoid confusion (especially to distinguish between the commutative algebra exponential $\exp_*$ and the Lie algebra exponential $\exp$).

We will mostly assume that our stability data on $\g$ are supported in a cone $\Gamma_{\geqslant 0}$. This is analogous to similar positivity assumptions in \cite{bt_stab} and \cite{gps}. 
\begin{definition}\label{positivity}
Fix a cone $\Gamma_{\geqslant 0}$. We say that $(Z,a) \in \Stab(\g)$ is \emph{positive} if $a(\gamma) \neq 0$ implies $\gamma \in \Gamma_{\geqslant 0}$.
\end{definition}

Note that for positive stability data all the rays $\ell_{\gamma}(Z) = - \R_{> 0} Z(\gamma)$ for $a(\gamma) \neq 0$ are contained in a fixed half-space $\mathbb{H}' \subset \CC$. For positive stability data all wall-crossing formulae take place in the fixed completion $\widehat{\g}$. In this setting, it is standard to rewrite \eqref{eq:wall-crossing} in a different way using the Poisson structure. Firstly notice that for any $\gamma \in \Gamma$ we may rewrite by M\"obius inversion
\begin{equation*}
a(\gamma) = -\sum_{n \geq 1, n | \gamma}\frac{\Omega(\gamma/n)}{n^2} e_{\gamma}.
\end{equation*}
\begin{definition} We call the collection of rational numbers $\Omega(\gamma)$ the \emph{BPS spectrum} (or simply the \emph{spectrum}) of the given family of stability data. It is naturally a function of the central charge $Z$. 
\end{definition}
\noindent For $\gamma \in \Gamma^{\rm prim}$, that is, for primitive $\gamma$ we have
\begin{equation*}
a(k \gamma) = - \sum_{p, n \geq 1, p n = k} \Omega(p \gamma) \frac{e_{p n \gamma}}{n^2}. 
\end{equation*}
Summing over all $k \geq 1$ and using standard dilogarithm notation we find
\begin{equation*}
\sum_{k \geq 1} a(k \gamma) = - \sum_{p \geq 1} \Omega(p \gamma) \sum_{n \geq 1} \frac{e_{p n \gamma}}{n^2} = - \sum_{p \geq 1} \Omega(p \gamma) \operatorname{Li}_2(e_{p\gamma}). 
\end{equation*} 
Since $\g$ is Poisson, for all $\gamma \in \Gamma_{\geqslant 0}$   
(not necessarily primitive) $[\operatorname{Li}_2(e_{\alpha}), -]$ acts as a commutative algebra derivation on the completion $\widehat{\g}$.   
Therefore $\exp\left(-[\operatorname{Li}_2(e_{\gamma}), -]\right)$ (the exponential of a derivation) acts as an algebra automorphism $T_{\gamma}$ of $\widehat{\g}$, preserving the Lie bracket (a Poisson automorphism). Direct computation shows that this action is especially nice:
\begin{equation*}
T_{\gamma}(e_{\eta}) = e_{\eta} (1 - e_{\gamma})^{\bra \gamma, \eta\ket}.
\end{equation*}
 Recall that we denote by $\ell_\gamma(Z)$ the ray $- \RR_{>0}Z(\gamma) \subset \CC^*$, for any given $\gamma \in \Gamma$ and $Z \in \Hom(\Gamma,\CC)$.

\begin{definition}\label{def:generic}
We say that $(Z,a) \in \Stab(\g)$ is \emph{generic} if when $a(\gamma), a(\eta) \neq 0$ and the rays $\ell_{\gamma}(Z)$, $\ell_{\eta}(Z)$ coincide, then $\bra \gamma, \eta \ket = 0$. We say that $(Z,a)$ is \emph{strongly generic} if $\ell_\gamma(Z) =  \ell_\eta(Z)$ with $a(\gamma) \neq 0 \neq a(\eta)$ imply that $\gamma$ and $\eta$ are linearly dependent.

Note that these conditions define open dense subsets of $\Stab(\g)$. We call the locus where the strongly generic condition does not hold the \emph{wall of marginal stability}.
\end{definition}
Given generic stability data $(Z,a)$ and a ray $-\ell \subset \Gamma_{\geqslant 0}$, the image of the group element $S_\ell$ via the adjoint representation on $\widehat{\g}$ admits the following expression
\begin{equation}\label{eq:stokesdata2}
\Ad S_\ell = \prod_{\gamma \in \Gamma^{\rm prim}, Z(\gamma) \in - \ell} \;\; \prod_{p \geq 1} T^{\Omega(p\gamma, Z)}_{p\gamma}.
\end{equation}
Here of course we write $T^{\Omega}_{\gamma}$ for the automorphism $\exp(-\Omega[\operatorname{Li}_2(e_{\gamma}), -])$. The point is that we do not need to specify an order for the previous product as the genericity condition implies that all the $T_{p\gamma}$ with $Z(\gamma) \in -\ell$ commute. 

Finally we can write \eqref{eq:wall-crossing} as an equivalent identity of Poisson automorphisms of $\widehat{\g}_V$, 
\begin{equation}\label{poisson1}
\prod^{\longrightarrow_{Z_{t_0 - \eps}}}_{\gamma \in \Gamma^{\rm prim}, \ell_{\gamma}(Z_{t_0 - \eps}) \subset V} \prod_{p \geq 1} T^{\Omega(p\gamma, Z_{t_0 - \eps})}_{p\gamma} = \prod^{\longrightarrow_{Z_{t_0 + \eps}}}_{\gamma \in \Gamma^{\rm prim}, \ell_{\gamma}(Z_{t_0 + \eps}) \subset V} \prod_{p \geq 1} T^{\Omega(p\gamma, Z_{t_0 + \eps})}_{p\gamma}.
\end{equation}
For this formula we assume that there is a single $t_0 \in [0, 1]$ for which the stability data is non-generic.

\noindent{\textbf{Example.}} A special case of \eqref{poisson1} appears in the case of the sublattice $\Gamma_0$ generated by elements $\gamma, \eta$ with $\bra \gamma, \eta \ket = 1$. Then, \eqref{poisson1} is equivalent to the ``pentagon identity''
\begin{equation*}
T_{\gamma} T_{\eta} = T_{\eta} T_{\gamma + \eta} T_{\gamma}.
\end{equation*}

\subsection{Symmetric stability data} The stability data on $\g$ coming from triangulated categories and physics are \emph{not} positive, bur rather satisfy the following condition.

\begin{definition}\label{symmetry} We say that $(Z,a) \in \Stab(\g)$ is \emph{symmetric} if $a(\gamma) = a(-\gamma)$ for all $\gamma \in \Gamma$.  
\end{definition}

The same condition must of course hold for the spectrum, $\Omega(\gamma) = \Omega(-\gamma)$. Under suitable conditions we can still give a meaning to the constructions in the present paper in the case of a symmetric spectrum $\Omega$ as formal power series in an auxiliary parameter $\tau$, by choosing a lift of the given symmetric stability data to the Lie algebra $\g[[\tau]]$.   
\begin{definition} For each $\alpha \in \Gamma$, $\Omega \in \Q$ and $p \in \N_{> 0}$ we define an element $T^{\Omega}_{\alpha, p} \in \Aut(\g[[\tau]])$ acting by 
\begin{equation*}
T^{\Omega}_{\alpha, p}(e_{\beta}) = e_{\beta}(1 - \tau^p e_{\alpha})^{\bra \alpha, \beta\ket \Omega}. 
\end{equation*} 
The automorphism $T^{\Omega}_{\alpha, p}$ preserves the Poisson Lie bracket on $\g[[\tau]]$.
\end{definition}
\begin{definition}\label{lift} Fix a continuous family of stability data on $\g$ parametrised by a topological space $\mathcal{U}$, and let $\Omega$ be the underlying (single or multi-valued) spectrum. A \emph{lift of the family to $\g[[\tau]]$} parametrised by an open subset $\mathcal{U}' \subset \mathcal{U}$ is a continuous family of stability data $\widetilde{a}(\alpha, Z) e_{\alpha}$ on $\g[[\tau]]$ such that on the generic locus in $\mathcal{U}'$ the group element corresponding to \eqref{eq:stokesdata2}, namely
\begin{equation*}
\exp_{D(\widehat{\g}_{V}[[\tau]])}\left(\sum_{Z(\alpha) \in - \ell} - \widetilde{a}(\alpha, Z) [ e_{\alpha}, -] \right) 
\end{equation*}
(where now $\widetilde{a}(\alpha, Z) \in \Q[[\tau]]$) has a factorisation
\begin{equation*}
\prod_{Z(\beta) \in -\ell}T^{\Omega(\beta, Z)}_{\beta, p(\beta, Z)}
\end{equation*}
for some function $p$ such that 
\begin{enumerate}
\item[$(i)$] for $\alpha \in \Gamma$, $k \in \Z$ we have $p(k \alpha, Z) = | k | p(\alpha, Z)$,
\item[$(ii)$] there is a norm on $\Gamma \otimes \C$ such that 
\begin{equation*} 
p(\alpha, Z) \geq  || \alpha ||
\end{equation*}
for all $\alpha \in \Gamma$ with $\Omega(\alpha, Z) \neq 0$. 
\end{enumerate}
\end{definition}
Condition $(i)$ will ensure that the certain formulae which hold in the positive case are still true for lifts to $\g[[\tau]]$, while $(ii)$ is a finiteness condition for $\tau$-graded components. We learned the idea of lifts of stability data to $\g[[\tau]]$ (as well as to more general artinian rings) from \cite{gps}.  
\begin{remark} It is an interesting question when such lifts exist for a family of stability data parametrised by a general $\mathcal{U}$. It is not hard to show that a lift always exists for the ``double" of a positive family of stability data, i.e. when we have a positive spectrum (so that central charges take values in a fixed strictly convex cone) $\Omega(\gamma, Z)$ parametrised by $Z \in \mathcal{U} \subset \Hom(\Gamma, \C)$, which we simply extend to be symmetric by the condition $\Omega(\gamma, Z) = \Omega( - \gamma, Z)$. The problem is what happens as the strictly convex cone opens up to a half-plane. In the categorical case of a 3CY triangulated category $\mathcal{C}$ lifts to $\g[[\tau]]$ are strictly related to the choice of a finite length heart of a bounded $t$-structure, and as the positive cone of central charges degenerates to a half-plane the heart may change (generically by a simple tilt).    
\end{remark}

\section{The connections $\nabla(Z)$ from stability data}\label{sec:gmn}

Suppose that we have a continuous family of stability data $(Z, a(Z))$ on $\g$, parametrised by some open set $\mathcal{U} \subset \Hom(\Gamma, \C)$, which satisfy the positivity property. Equivalently we have a spectrum $\Omega(\gamma, Z)$, $Z \in \mathcal{U}$, satisfying the wall-crossing formula and positivity. The aim of this Section is to prove an existence result for a family of connections whose generalised monodromy is prescribed by the spectrum $\Omega$. At the same time we will derive explicit formulae for flat sections (see \eqref{iterativeSol}) and connection matrix elements (see \eqref{explicitA}). We will recall along the way most of the technical terms involved (such as framed connections, Stokes data and isomonodromy), following \cite{Boalch3,bt_stab}. The existence result and explicit formulae will be used in the following sections to study the Bridgeland-Toledano Laredo connection as a scaling limit (Theorem \ref{thm:limit}), and most importantly to establish a connection with tropical geometry and the methods of \cite{gps} (Theorems \ref{thm:tropical1} and \ref{thm:tropical2}).

\begin{theorem}\label{prop:GMNconstruction} Let $\Omega$ be a spectrum giving a continuous family of stability data parametrised by an open subset $\mathcal{U} \subset \Hom(\Gamma, \C)$. Let $Z \in \mathcal{U}$ correspond to generic stability data.
\begin{enumerate} 
\item[(i)] If $\Omega$ is positive (Definition \ref{positivity}) then there exists a meromorphic framed connection $(\nabla(Z),g(Z))$ on $\PP^1$, with values in $D^*(\widehat{\g})$, of the form
\begin{equation*}
\nabla(Z) = d - \left(\frac{1}{z^2} \A^{(-1)}(Z) + \frac{1}{z}\A^{(0)}(Z) + \A^{(1)}(Z)\right)dz,
\end{equation*}
such that $(\nabla(Z),g(Z))$ has Stokes data (generalized monodromy) at $z = 0$ given by the rays $\ell_{\gamma}(Z)$ and factors $\Ad S_{\ell_{\gamma}(Z)}$ of \eqref{eq:stokesdata2}. The family $\nabla(Z)$ extends to a real analytic isomonodromic family of framed meromorphic connections (globally on $\PP^1$) on all of  $\mathcal{U}$. The derivations $\A^{(-1)}$ and $\A^{(1)}$ are related by the symmetry \eqref{conjugation} below.
\item[(ii)] If the spectrum $\Omega$ is symmetric instead (Definition \ref{symmetry}) then the same result holds on every open subset $\mathcal{U}' \subset \mathcal{U}$ on which we have a continuous lift to $\g[[\tau]]$. The corresponding $\nabla(Z)$ take values in $D^*(\g[[\tau]])$. 
\end{enumerate}
\end{theorem}
\noindent The proof of this Theorem consists of several steps which are carried out in \ref{sec:gmn:sub:start} - \ref{sec:gmn:sub:finish} below. Our methods are different from those in \cite{bt_stab} and are very much inspired by \cite{gmn}. We will discuss the construction in detail in the case of positive stability data and sketch how this can be adapted to symmetric stability data in Section \ref{symmetricRemark}.

\subsection{Irregular connections on $\PP^1$}\label{subsec:btlgeneral}

Let $P$ be the holomorphically trivial, principal $\Aut^*(\widehat{\g})$-bundle on $\mathbb{P}^1$, that is the inverse limit of the system of holomorphically trivial principal bundles corresponding to the groups $\Aut^*(\g_{\leqslant k})$. 

A $D^*(\widehat{\g})$-valued meromorphic function $\A$ on $\PP^1$ is an inverse system of meromorphic functions $\A_{\leqslant k}$ with values in the finite-dimensional vector spaces $D^*(\g_{\leqslant k})$. For a choice of local coordinate $z$ in $\PP^1$, we will consider meromorphic connections on $P$ of the form
\begin{equation*}
\nabla = d - \A\,dz,
\end{equation*} 
given by the inverse limit of a system of meromorphic connections
\begin{equation}\label{eq:GMNformfin}
\nabla_{\leqslant k} = d - \A_{\leqslant k}dz.
\end{equation}
A local gauge transformation $Y\colon U \to \Aut^*(\g)$ is an inverse system of holomorphic maps $Y_{\leqslant k} \colon U_k \to \Aut^*(\g_{\leqslant k})$ on a chain of nonempty open subsets $\cdots U_{k+1} \subset U_k \cdots \subset U\subset \PP^1$. We use the standard notation
\begin{equation*}
Y \cdot \A = (\partial_z Y) Y^{-1} + Y \A Y^{-1}.
\end{equation*}
In the rest of the paper, we focus on connections $\nabla$ with a second order pole at $z = 0$ and simple dependence on $z$, of the form
\begin{equation}\label{eq:GMNform}
\nabla = d - \left(\frac{1}{z^2} \A^{(-1)} + \frac{1}{z}\A^{(0)} + \A^{(1)}\right)dz,
\end{equation}
where $\A^{(j)}\in D^*(\widehat{\g})$ are \emph{constant} in $z$. We choose the formal type (that is the gauge equivalence class under formal power series gauge transformations, see \cite{Boalch2}) at the origin to be
\begin{equation}\label{eq:formtype}
d + \frac{Z}{z^2} dz,
\end{equation}
with $Z \in \Hom(\Gamma,\CC)$ (regarded as an element of $D^*(\widehat{\g})$, acting by $Z(e_{\alpha}) = Z(\alpha)e_{\alpha}$). It will be convenient to work with the following notion (see \cite{Boalch2}).

\begin{definition} A (compatibly) framed connection is a pair $(\nabla,g)$ given by a connection $\nabla$ as above and an element $g \in \Aut^*(\widehat{\g})$ such that $g^{-1} \cdot \A^{(-1)} = - Z$.  
\end{definition}

We introduce now \emph{Stokes data} for a framed connection $(\nabla,g)$ following \cite{Boalch1,bt_stab}. We define an \emph{anti-Stokes ray} of $\nabla$ to be of the form $\ell_\gamma  = - \RR_{> 0} Z(\gamma)$ for $\gamma \in \Gamma \backslash \{0\}$, and say that a ray is \emph{admissible} if it is not an anti-Stokes ray (the set of anti-Stokes rays supporting a nontrivial Stokes factor need not be finite, as explained below). By definition, an admissible ray for $\nabla$ is admissible for each $\nabla_{\leqslant k}$. 

A flat section (or solution) $X \colon U \to \Aut^*(\widehat{\g})$ for $\nabla$ on an open $U \subset \PP^1$ is an inverse system of holomorphic flat sections $X_{\leqslant k} \colon U_k \to \Aut^*(\g_{\leqslant k})$ for $\nabla_{\leqslant k}$, that is, satisfying
\begin{equation*}
\partial_z X_{\leqslant k} = \A_{\leqslant k} X_{\leqslant k}
\end{equation*}
on a chain of nonempty open subsets $\cdots U_{k+1} \subset U_k \cdots \subset U$. The previous formula should be understood as acting on an arbitrary element of $\g_{\leqslant k}$, where we use the standard notation for the composition of maps. Note that a solution provides a local description of the connection
\begin{equation}\label{eq:conlocal}
\cA = (\partial_z X)X^{-1}.
\end{equation}
Given an admissible ray $\ell \subset \CC^*$ for $\nabla$, define $\mathbb{H}_\ell \subset \CC$ to be the open half-plane containing $\ell$ and with boundary perpendicular to $\ell$. Then, there exists a unique \emph{fundamental} or \emph{canonical} solution, defined on the \emph{fixed} half-plane $\mathbb{H}_\ell$,  
\begin{equation*}
X_\ell \colon \mathbb{H}_\ell \to \Aut^*(\widehat{\g})
\end{equation*}
with prescribed asymptotics $X_\ell e^{-Z/z} \to g$ as $z \to 0$ in $\mathbb{H}_\ell$ (compare with \cite[Th. 6.2]{bt_stab}). This follows from an analogous result for the corresponding system of connections $\nabla_{\leqslant k}$ (see \cite[Th. 3.1 and Lem. 3.3]{Boalch1}). The solution $X_\ell$ is called the \emph{canonical fundamental solution} of $(\nabla,g)$ corresponding to the admissible ray $\ell$.

Following \cite{bt_stab} we can use this result to define Stokes factors for $(\nabla, g)$. This requires a little care since the set of anti-Stokes rays supporting a nontrivial Stokes factor is not necessarily discrete. If $\ell_1 \neq - \ell_2$ are two admissible rays, ordered so that the closed sector $\overline{\Sigma} \subset \CC^*$ swept by clockwise rotation from $\ell_1$ to $\ell_2$ is convex, there is a unique element $S_{\ell_1,\ell_2} \in \Aut^*(\widehat{\g})$ such that on $\Sigma$ we have
\begin{equation*}
X_{\ell_2} = X_{\ell_1} S_{\ell_1,\ell_2}.
\end{equation*}
\begin{definition}
$(\nabla,g)$ admits a Stokes factor $S_\ell \in \Aut^*(\widehat{\g})$ along the anti-Stokes ray $\ell$ if for all $k$ the
elements $S_{\ell_1,\ell_2}$ tend to $S_\ell$ in $\g_{\leqslant k}$ as the admissible rays $\ell_1,\ell_2$ tend to $\ell$ in such a way that $\ell$ is always contained in the corresponding closed sector.
\end{definition}
Similarly to \cite[Prop. 6.3]{bt_stab} one can show that $\nabla$ admits a Stokes factor along any anti-Stokes ray, and the data of the formal type and Stokes factors at $z = 0$ characterize the gauge equivalence class in a punctured disc at $z = 0$. 

Our goal in this Section is to construct a family of connections as above with \emph{prescribed Stokes factors} $\Ad S_{\ell}$. The \emph{isomonodromy} (i.e. constant generalized monodromy) property is then precisely the wall-crossing formula \eqref{eq:wall-crossing}, and would hold automatically.

We start with a continuous family of positive elements $(Z, a(Z))$ of $\Stab(\g)$ parametrised by an open set $
\mathcal{U} \subset \Hom(\Gamma, \C)$. In particular all the rays $\ell_{\gamma}(Z) = - \R_{> 0} Z(\gamma)$ for $a(\gamma) \neq 0$ are contained in a fixed half-space $\mathbb{H}'$. 

\subsection{Riemann-Hilbert (RH) factorisation problem}\label{sec:gmn:sub:start} Following ideas of \cite{gmn}, the construction of $(\nabla(Z),g(Z))$ from generic $(Z, a(Z))$ can be turned into a Riemann-Hilbert factorisation problem, that is the construction of a sectionally holomorphic function (i.e. a holomorphic function with branch-cuts) 
\begin{equation*}
X(Z) = X(z, Z)\!: \C^* \to \Aut^*(\widehat{\g}),  
\end{equation*}
with prescribed jump across $\ell_{\gamma}(Z)$, $\gamma \in \Gamma$ given by the automorphism \eqref{eq:stokesdata2}. More precisely we seek a family $X(z, Z)$ of functions of $z$ taking values in $\Aut^*(\widehat{\g})$, parametrised by $Z$, with the following properties:
\begin{enumerate}
\item $X(Z)$ is an $\Aut^*(\widehat{\g})$-valued holomorphic function in the complement of the rays $\ell_{\gamma}(Z)$ with $a(\gamma) \neq 0$;
\item $X(Z)(e_{\alpha})$ extends to a holomorphic function with values in $\widehat{\g}$ in a neighborhood of $\ell_{\alpha}(Z)$;
\item fix a ray $\ell \subset \mathbb{H}'$. For every $z_0 \in \ell$ denote by $X(z^+_0)$ the limit of $X(z, Z)$ as $z \to z_0$ in the counterclockwise direction. Similarly let $X(z^-_0)$ denote the limit in the clockwise direction. Both limits exist, and they are related by
\begin{equation*}
X(z^+_0) = X(z^-_0) \Ad S_\ell;   
\end{equation*}
\item there exists $g(Z) \in \Aut^*(\widehat{\g})$ such that $\lim_{z \to 0} X(z, Z) e^{-Z/z} = g(Z)$ along directions non-tangential to anti-Stokes rays.

\end{enumerate} 

\begin{remark}\label{RHrmk}  
Condition (1) means that each element of the inverse system $X_{\leqslant k}(Z)$ should be holomorphic in the complement of the finite set of rays $\ell_{\gamma}(Z)$ for which the class of $a(\gamma)$ in $\g_{\leqslant k}$ is nonzero. Similar clarifications apply to the other conditions. In particular the limit in condition (4) is understood as taken in a finite-dimensional (but arbitrary) quotient $\g_{\leqslant k}$, i.e. we \emph{do not} require that the limit holds uniformly in $z$ for all $\g_{\leqslant k}$.
\end{remark}

Given a solution of the Riemann-Hilbert factorization problem $(1)$-$(4)$, we can construct the framed connection in the obvious way: $\nabla(Z)$ is given by formula \eqref{eq:conlocal} and $g(Z)$ by condition $(4)$. Note that the prescribed jump \eqref{eq:stokesdata2} across an anti-Stokes ray $\ell_{\gamma}(Z)$ is independent of $z$, so the jumps cancel out in the local expression \eqref{eq:conlocal} which thus patches over a collection of sectors between anti-Stokes rays to all $\CC^*$. Therefore it defines a meromorphic connection on $\PP^1$ with (possibly) poles at $z = 0,\infty$. Provided that the restriction of $X(Z)$ to sectors between anti-Stokes rays admits a suitable analytic continuation, continuity of the family of stability data will ensure that $(\nabla(Z),g(Z))$ is isomonodromic. In what follows, we will construct explicit solutions $X(Z)$ of the Riemann-Hilbert factorization problem and prove that the corresponding framed connections fulfill the requirements of Theorem \ref{prop:GMNconstruction}.

\subsection{Integral operator}\label{subsec:operator} A basic technique to solve Riemann-Hilbert factorization problems consists in finding fixed points for singular integral operators, involving integration along the jump contour (see e.g. \cite{musk, FIKN}). 

Recall that the stability data $(Z, a(Z))$ are equivalent to the data $(Z, \{\Omega(\gamma, Z)\})$. For each $Y \in \End(\widehat{\g})$ we have
\begin{equation}\label{preTBop}
Y\Ad S_\ell(e_\alpha) = Y(e_{\alpha}) * \prod_{Z(\gamma)\in - \ell}(1 - Y(e_{\gamma}))^{\Omega(\gamma) \bra \gamma, \alpha\ket}.
\end{equation}
This leads to consider integral operators, acting on suitable $\End(\widehat{\g})$-valued holomorphic functions $Y$, of the form
\begin{align}\label{TBop}
&\nonumber \mathcal{Z}[Y](z)(e_{\alpha})\\
& = e_{\alpha} \exp_*\left(L(\alpha) + \sum_{\gamma} \Omega(\gamma)\bra \gamma, \alpha\ket \int_{\ell_{\gamma}(Z)}\frac{dz'}{z'}\rho(z, z')\log_*(1 - Y(z')(e_{\gamma}))\right),
\end{align}
summing over all $\gamma \in \Gamma$ with $\Omega(\gamma) \neq 0$. Here $L(z, Z)$ takes values in $\Hom(\Gamma, \C)$ and $\rho(z, z')$ is a Cauchy-type integration kernel which must be chosen appropriately. This choice of $\mathcal{Z}$ should be compared to \cite[Eq. 5.13]{gmn} and will be ultimately justified by Lemma \ref{lem:RHsolution}. Following \cite{gmn} we will actually choose  
\begin{align*}
L(z, Z) &= \frac{Z}{z} + z \bar{Z},\\
\rho(z, z') &= \frac{1}{4\pi i}\frac{z' + z}{z' - z}.
\end{align*}
This choice of $L$ and $\rho$ is crucial for the solution of RH to provide a connection of the form \eqref{eq:GMNform}. 

We regard \eqref{TBop} as a formal expression for a moment. Notice that we have formally
\begin{equation*}
\mathcal{Z}[Y](e_{\alpha} * e_{\beta}) = \mathcal{Z}[Y](e_{\alpha}) * Z[Y](e_{\beta}),
\end{equation*}
so, when it is well defined, $\mathcal{Z}[Y]$ automatically preserves the commutative product, that is it maps the set of algebra endomorphisms to itself. We will soon prove that it actually maps to the set of algebra \emph{automorphisms}. Finally we will prove that a fixed point for $\mathcal{Z}$ exists and gives a solution to our Riemann-Hilbert problem (Lemma \ref{lem:RHsolution}).

We start by rewriting the integral operator in terms of the original stability data. For this, we set $\dt(\gamma) e_\gamma = - a(\gamma)$, that is
\begin{equation*}
\dt(\gamma') = \sum_{n > 0,\,n | \gamma'} \frac{\Omega(n^{-1}\gamma')}{n^2}.
\end{equation*}
\begin{remark} The notation DT reflects the way in which the ``BPS state counts" $\Omega$ are related to Donaldson-Thomas invariants; in the present case it is of course purely formal.
\end{remark}

\noindent Using that $Y(z)$ is an homomorphism for the commutative product $*$ on $\widehat \g$, we have
\begin{align*}  
\sum_{\gamma } \Omega(\gamma ) \log_*(1 - Y(z')(e_{\gamma})) \gamma &= \sum_{\gamma }\Omega(\gamma ) \sum_{k > 0} \frac{1}{k^2} Y(z')(e_{k \gamma}) k\gamma \\
&= \sum_{\gamma } Y(z')(e_{\gamma}) \dt(\gamma ) \gamma.
\end{align*}
Let us introduce the ``base point" function
\begin{equation*}
X^0 = e^L \colon \mathbb{C}^* \to \Aut^*(\widehat{\g}).
\end{equation*}
Then we can rewrite the action of $\mathcal{Z}$ as
\begin{equation}\label{eq:ZDT}
\mathcal{Z}[Y](z)(e_{\alpha}) = X^0(z)(e_{\alpha})\exp_*\left(\sum_{\gamma } \bra \gamma , \alpha \ket \dt(\gamma ) \int_{\ell_{\gamma }}\frac{d z'}{ z'}\rho( z,  z') Y(z')(e_{\gamma})\right).
\end{equation}
\begin{lem}\label{AutoLemma} Suppose $Y(z)$ is such that $\mathcal{Z}[Y](z)$ is well-defined. Then $\mathcal{Z}[Y](z)$ takes values in $\Aut^*(\widehat{\g})$.  
\end{lem}
\begin{proof} By \eqref{eq:ZDT} the function $Z[Y](z)$, if well defined, is an algebra endomorphism and has the form
\begin{equation*}
Z[Y](z)(e_{\alpha}) = e_{\alpha} \exp(\bra \alpha, x\ket)
\end{equation*}
for a fixed $x \in \Gamma\otimes \widehat{\g}$. We need to show this is invertible. After choosing a basis of $\Gamma$ the problem of inverting $Z[Y](z)$ thus becomes that of inverting a mapping of the form $s_i \mapsto s_i \exp(\Phi_i({\bf s}))$ for suitable formal power series $\Phi_i({\bf s})$ in $m$ indeterminates. By the classical multivariate Lagrange inversion formula (which we recall in Section \ref{subsec:explicitformula}, when we will discuss explicit formulae for the connection form) the inverse exists and has again the form $s_i \mapsto s_i \exp(\Psi_i({\bf s}))$ for certain formal power series $\Psi_i({\bf s})$.  
\end{proof}
To find fixed points for the operator $\mathcal{Z}$ we wish to iterate in the integral equation \eqref{TBop}. Initial points for this iteration are provided by the following definition.

\begin{definition}\label{def:admis}
An $\Aut^*(\widehat{\g})$-valued function $Y$ is called \emph{admissible} if
\begin{enumerate}
\item $Y$ is well-defined and holomorphic in the complement of the rays $\ell_{\gamma}(Z)$ with $a(\gamma) \neq 0$ (see Remark \ref{RHrmk});
\item the limit $\lim_{z \to 0,\infty} Y e^{-L}(z)(e_\alpha)$ exists in $\widehat{\g}$ for all $\alpha \in \Gamma_{\geqslant 0}$, in any direction non-tangential to $\ell_{\gamma}(Z)$ with $a(\gamma) \neq 0$;
\item $Y(e_\gamma)$ extends to a $\widehat{\g}$-valued holomorphic function in a neighborhood of $\ell_{\gamma}(Z)$ with the same property.
\end{enumerate}
\end{definition}
The holomorphic function $X^0 = e^L \colon \mathbb{C}^* \to \Aut^*(\widehat{\g})$ is trivially admissible. This will be our choice of initial point for the iteration. To run the iterative process we need the following.

\begin{lem}\label{lem:admis}
Assume that $(Z,a(Z))$ is strongly generic. If $Y\!: \C^* \to \Aut^*(\widehat{\g})$ is admissible then so is $\mathcal{Z}[Y]$.
\end{lem}

\subsection{Basic estimates}\label{BesselLemma} The Lemma follows from elementary estimates which we now establish (its proof is given at the end of the present Section). For later applications we will also allow a scaling parameter $R > 0$. The estimates needed for the proof of Lemma \ref{lem:admis} are obtained setting $R = 1$. We need to study the integral 
\begin{equation}\label{basicIntegral}
\int_{\R_{< 0} e^{i\psi} c} \frac{d z'}{z'}\frac{z' + z}{ z' - z} \exp(R( z'^{-1}c +  z' \bar{c}))
\end{equation}
where $c\in \C^*$, $\psi$ is a sufficiently small angle, and $z \notin \R_{< 0} e^{i\psi} c$. We also look at the integral along an arc,
\begin{equation}\label{arcIntegral}
\int_{|\psi|<\eps} \frac{d z'}{ z'}\frac{ z'+ z}{ z'- z} \exp(  R ( z'^{-1} c +  z' \bar{c}))
\end{equation}
where $ z' \in \R_{< 0} e^{i\psi} c$, $| z'|$ is fixed, $ z \notin \R_{< 0} c$ and $\eps$ is small enough. We claim that \eqref{basicIntegral} converges for sufficiently small $|\psi|$, and in fact its modulus is bounded above by 
\begin{equation*}
\frac{C}{2 R |c|}\exp(-2 R |c|)
\end{equation*} 
for $R$ large enough (for a constant $C$ depending on the angular distance of $z$ from $\R_{< 0} e^{i\psi} c$). This follows from the change of variable $ z' = -e^{s + i \psi} c$ for real $s$ and $\psi$, reducing \eqref{basicIntegral} to
\begin{equation*}
\int^{+\infty}_{-\infty} ds \frac{-e^{s + i\psi} c + \z}{-e^{s + i\psi} c - \z}\exp(- R|c|(e^{-s - \log|c| - i\psi} + e^{s + \log|c| + i\psi}))
\end{equation*}
The modulus of this is bounded above by 
\begin{equation*}
C \int^{+\infty}_{-\infty} ds |\exp(-2 R |c| \cosh(s + \log|c|+ i\psi))|,
\end{equation*} 
where $C$ is a constant depending on the angular distance of $\z$ from $\R_{< 0} e^{i\psi} c$. For $\psi$ sufficiently small this is in turn bounded by 
\begin{equation*}
C \int^{+\infty}_{-\infty} ds' \exp(-2 R |c| \cosh(s')),
\end{equation*}
where $C$ is a new (possibly larger) constant, depending also on $\psi$. We recognize the integral as a Bessel function, and we find that \eqref{basicIntegral} converges for $|\psi| \ll 1$, and moreover it is actually bounded above by $\frac{C}{2  R |c|}\exp(-2  R |c|)$ for $R$ large enough, as required (alternatively this upper bound can be quickly derived e.g. the Laplace approximation method). 

Similarly with the same change of variable \eqref{arcIntegral} becomes
\begin{equation*}
i \int^{+\eps}_{-\eps} d\psi \frac{-e^{s + i\psi}c + \z}{-e^{s + i\psi}c - \z}\exp(-  R|c| e^{-s - \log|c| - i\psi})\exp(-  R|c| e^{s + \log|c| + i\psi}). 
\end{equation*}
One can check that this vanishes for $s \to \pm \infty$ (for fixed, sufficiently small $\eps$). 
 
\begin{proof}[Proof of Lemma \ref{lem:admis}]
First we show that for admissible $Y$, $\alpha \in \Gamma_{\geqslant 0}$, $z \in \C^*$ away from the rays $\ell_{\gamma}$ with $a(\gamma) \neq 0$ the right hand side of \eqref{TBop} is a well defined element of $\widehat{\g}$. Note that $z$ lies away from the rays of integration in \eqref{eq:ZDT}, and hence the function $(z')^{-1}\rho(z,z')$ is holomorphic on each $\ell_{\gamma}$ appearing in that formula. For fixed $k > 0$ the projection of $\mathcal{Z}[Y](e_\alpha)$ on $\g_{\leqslant k}$ involves finitely many $\gamma$.  We claim that all the
integrals appearing are convergent. To see this we first expand $\log_*(1 - Y(e_\gamma))$ and use the boundary conditions in Definition \ref{def:admis} (2) to see that each integral
is dominated by the sum of a fixed finite number of integrals of the form
\begin{equation*}
C \int_{\ell_{\gamma}(Z)} \frac{d z'}{z'}\frac{z' + z}{ z' - z} \exp( z'^{-1}Z(\gamma) +  z' \overline{Z(\gamma)})
\end{equation*}
for some constant $C$. By the basic estimates for \eqref{basicIntegral}, these are all convergent. Then it follows from standard theory that $\mathcal{Z}[Y](e_\alpha)$ is a holomorphic function of $z \in U$ (since $(z')^{-1}\rho(z,z')$ is, and by the above convergence). On the other hand it is also holomorphic in a neighborhood of $\ell_{\alpha}(Z)$ since by the genericity property for $Z$ the integral along $\ell_{\gamma}(Z) = \ell_{\alpha}(Z)$ appears weighted by a factor of $\langle\gamma,\alpha\rangle$ and therefore vanishes. Finally we
can use the basic estimate for \eqref{basicIntegral} to take the limit 
$
\lim_{z \to 0} \mathcal{Z}[Y] e^{-L}(z)(e_\alpha)
$ 
in a direction non-tangential to $\ell_{\gamma}(Z)$ with $a(\gamma) \neq 0$: by the definition of $L$ this is the constant element of $\widehat{\g}$ given by
\begin{equation*}
e_{\alpha} \exp_*\left(\frac{1}{4\pi i}\sum_{\gamma} \Omega(\gamma)\bra \gamma, \alpha\ket \int_{\ell_{\gamma}(Z)}\frac{dz'}{z'}\log_*(1 - Y(z')(e_{\gamma}))\right)
\end{equation*}
The same argument applies to the $z \to \infty$ limit.
\end{proof} 
 
\subsection{Application of Plemelj's theorem} The link between the integral operator $\mathcal{Z}$ and the Riemann-Hilbert problem follows from standard theory and relies on a result from elementary complex analysis (Plemelj's theorem, see e.g. \cite{musk}). This says that if $\ell \subset \C$ is a smooth (oriented) arc and $f(z', z)$ is uniformly H\"older continuous and integrable for $z' \in \ell$ and all $z$ then fixing $z_0 \in \ell$ we have
\begin{equation}\label{PlemeljThm}
\lim_{z \to z^{\pm}_0} \frac{1}{2\pi i}\int_{\ell} \frac{ f(z', z)}{z' - z} dz' = \pm \frac{1}{2}f(z_0, z_0) + \frac{1}{2\pi i} \operatorname{pv} \int_{\ell} \frac{ f(z', z_0)}{z' - z_0} dz'
\end{equation}
where $z \to z^+_0$ ($z \to z^-_0$) denotes the limit taken in the region to the left (respectively to the right) of $\ell$ and the symbol $\operatorname{pv}$ denotes the (well defined, convergent) principal value of the (divergent) Cauchy integral.

Suppose now $Y(z)$ is an admissible function and $(Z,a(Z))$ is generic. Fix a ray $\ell$ in our positive half-space $\mathbb{H}'$ and a point $z_0 \in \ell$, and denote by $\mathcal{Z} [Y](z^+_0)(e_{\alpha})$ the limit of $\mathcal{Z}[Y](z)(e_{\alpha})$ as $ z \to  z_0$ in the counterclockwise direction. Similarly let $\mathcal{Z} [Y](z^-_0)(e_{\alpha})$ denote the limit in the clockwise direction.

\begin{lem}\label{JumpLemma}
Both limits exist, and they are related by
\begin{equation*} 
\mathcal{Z} [Y](z^+_0)(e_{\alpha}) = \mathcal{Z} [Y](z^-_0)(e_{\alpha}) \prod_{Z(\gamma)\in - \ell}(1 - Y(z_0)(e_{\gamma}))^{\Omega(\gamma) \bra \gamma, \alpha\ket}. 
\end{equation*}
\end{lem}

\begin{proof}
Note first that $Y(z_0)(e_{\gamma})$ in \eqref{JumpLemma} is well-defined because $Y$ is admissible. Fix $\gamma \in \Gamma_{> 0}$. We claim that Plemelj's theorem gives 
\begin{align}\label{JumpLemmaIdentity}
\nonumber \lim_{ z \to z^{\pm}_0}\frac{1}{4\pi i}\int_{\ell_{\gamma }}\frac{d z'}{ z'}\frac{ z' +  z}{ z' -  z}\log(1 - Y(z')(e_{\gamma}) ) &= \pm \frac{1}{2}\log(1 - Y(z_0)(e_{\gamma}))\\ 
&+ \frac{1}{4\pi i}\operatorname{pv}\int_{\ell_{\gamma }}\frac{d z'}{ z'}\frac{ z' +  z_0}{ z' - z_0}\log(1 -  Y(z') (e_{\gamma})),
\end{align}
where the last principal value integral is convergent. Equation \eqref{JumpLemma} follows at once from this identity. To establish it recall that $\log(1 - Y(z')(e_{\gamma}) )$ must be interpreted as an element of $\widehat{\g}$, i.e. a formal power series. We have indeed 
\begin{equation*}
\frac{1}{4\pi i} \int_{\ell_{\gamma}}\frac{d\z'}{\z'}\frac{ z' +  z}{ z' -  z}\log(1 - Y(z')(e_{\gamma})) = \sum_{k \geq 1} \frac{1}{4\pi i}\int_{\ell_{\gamma}}dz'\frac{1}{z' - z}\left(1 + \frac{z}{z'}\right)\frac{1}{k}(Y(z')(e_{\gamma}))^k,
\end{equation*}
and applying \eqref{PlemeljThm} to each summand, choosing in turn
\begin{equation*}
f(z') = \frac{1}{2} \left(1 + \frac{z}{z'}\right)\frac{1}{k}(Y(z')(e_{\gamma}))^k
\end{equation*}
shows that the left hand side of \eqref{JumpLemmaIdentity} equals the well defined formal power series 
\begin{equation*}
\pm \frac{1}{2}\sum_{k \geq 1}  \frac{1}{k} (Y(z_0)(e_{\gamma}))^k + \frac{1}{4\pi i} \sum_{k\geq 1} \operatorname{pv} \int_{\ell_{\gamma}}dz'\frac{1}{z' - z_0}\left(1 + \frac{z_0}{z'}\right)\frac{1}{k}(Y(z')(e_{\gamma}))^k.
\end{equation*}
This is precisely the right hand side of \eqref{JumpLemmaIdentity}.
\end{proof}

\subsection{Fixed point}\label{sec:GMN:sub:fixed} Following \cite[App. C]{gmn}, we construct now a solution of the singular integral equation 
\begin{equation}\label{TBA}
X = \mathcal{Z}[X], 
\end{equation}
by iteration from $X^0 = e^L$, and show that it solves the Riemann-Hilbert factorization problem. Recall that $L(z) = z^{-1}Z + z \overline{Z}$ and hence the solution will depend (in a complicated way) on the parameter $Z$.

Set $\mathcal{Z}^{(i)} = \mathcal{Z} \circ \cdots \circ \mathcal{Z}$ ($i$ times) and consider the sequence of algebra automorphisms 
$X^{(i)} = \mathcal{Z}^{(i)}[X^0]$
for $i \geq 0$. We claim that $X^{(i)}(z)$ converges as $i \to \infty$.  
This follows from an explicit calculation. It follows from \eqref{eq:ZDT} that 
\begin{align}\label{TBinduction}
\nonumber X^{(i)}( z)(e_{\alpha}) &= X^0(z)(e_{\alpha}) \\
&\sum_{\mathbf{k}}\frac{1}{\mathbf{k}!} \sum_{\{\gamma_1, \ldots, \gamma_l\}}\prod_j \left(\bra \dt(\gamma_j)\gamma_j, \alpha \ket \int_{\ell_{\gamma_j}}\frac{d z'}{ z'}\rho( z,  z') X^{(i-1)}(z')(e_{\gamma_j}) \right)^{\mathbf{k}_j}, 
\end{align}
where we sum over ordered partitions $\mathbf{k}$ of arbitrary length $l$, and we take the product over all unordered collections of pairwise distinct $\{\gamma_1, \ldots, \gamma_l\}\subset\Gamma$. Let us denote by $T$ a connected rooted tree, with set of vertices $T^0$ and set of edges $T^1$. We assume that the vertices of $T$ are decorated by $\Gamma$, i.e. there is a map $\alpha\colon T^0 \to \Gamma$, $v \to \alpha(v)$. Introduce a factor 
\begin{equation*}
\W_{\T} = \frac{\dt(\gamma_{\T}) \gamma_{T}}{|\Aut(\T)|} \prod_{v \to w} \bra \dt(\alpha(v))\alpha(v), \alpha(w)\ket,
\end{equation*}
where $\gamma_{\T}$ denotes the label of the root of $\T$, and $\Aut(\T)$ is the automorphism group of $\T$ as a decorated, rooted tree. To each $\T$ we also attach a ``propagator" $\G_{\T}(z, Z)$ which is a $\g$-valued holomorphic function in the complement of those $\ell_{\gamma}$ rays with $a(\gamma) \neq 0$, and is defined inductively by
\begin{equation}\label{propagators}
\G_{\T}(\z, Z) = \int_{\ell_{\gamma_{\T}}} \frac{d\z'}{\z'}\rho(\z, \z') X^0(z', Z)(e_{\gamma_{\T}}) \prod_{\T'}\G_{\T'}(\z', Z),
\end{equation}
where $\{\T'\}$ denotes the set of (connected, rooted, decorated) trees obtained by removing the root of $\T$ (setting $\G_{\emptyset}(\z) = 0$, but with the convention that an empty product equals $1$). We claim that 
\begin{equation}\label{TBinduction2}
X^{(i)}(z)(e_{\alpha}) = X^0(e_{\alpha})\sum_{{\bf h}} \frac{1}{{\bf h}!} \sum_{\{\T_1, \ldots, \T_m\}} \prod_r \left( \bra \W_{\T_r} \G_{\T_r}(z), \alpha \ket\right)^{{\bf h}_r},  
\end{equation} 
where we sum over all partitions ${\bf h}$ with arbitrary length $m$ and all collections of $m$ pairwise distinct trees $\{\T_1, \ldots, \T_m\}$ as above, each of which has depth at most $i - 1$. (Notice that $\W_{\T_r} \G_{\T_r}(\z)$ lies in $\Gamma\otimes\g$, and we have extended the pairing $\bra - , - \ket$ by $\g$-linearity). The claim can be proved by induction on $i$. It holds for $i = 1$ by inspection. Assuming \eqref{TBinduction2} holds for a given $i-1$ we see by  \eqref{TBinduction} that $X^{(i)}(z)(e_{\alpha})$ can be written as  
\begin{align*}
&X^0(z)(e_{\alpha}) \sum_{\mathbf{k}} \frac{1}{\mathbf{k}!} \sum_{\{\gamma_1, \ldots, \gamma_l\}} \\
&\prod_j \left(\bra \dt(\gamma_j)\gamma_j, \alpha \ket \int_{\ell_{\gamma_j}}\frac{d z'}{ z'}\rho( z,  z') X^0(z')(e_{\gamma_j})\sum_{{\bf h}} \frac{1}{{\bf h}!} \sum_{\{T_1, \ldots, T_r\}} \prod_r \left( \bra \alpha, \W_{\T_r} \G_{\T_r}(z)\ket\right)^{{\bf h}_r}\right)^{\mathbf{k}_j}.
\end{align*}  
For each collection $\{T_1, \ldots, T_r\}$ and each $j$ we can define a new tree $T'_j$, with length at most $i$, by gluing a common root decorated by $\gamma_j$. We then recognise the quantity in large parentheses as the pairing $\bra W_{T'_j} G_{T'_j}(z),\alpha \ket$, and the sum above as
\begin{equation*}
X^0(z)(e_{\alpha}) \sum_{\mathbf{k}} \frac{1}{\mathbf{k}!} \sum_{\{T'_{1}, \ldots, T'_{j}\}}(\bra W_{T'_j} G_{T'_j}(z), \alpha \ket  )^{\mathbf{k}_j}
\end{equation*}
as required. 

Recalling the definition of an $\Aut^*(\widehat{\g})$-valued holomorphic function, equation \eqref{TBinduction2} shows that the sequence $X^{(i)}$ converges for $i \to \infty$ to a limit $X = X(z, Z)$, which is a solution of \eqref{TBA}, given explicitly by 
\begin{equation}\label{iterativeSol}
X(z, Z)(e_{\alpha}) = X^0(z, Z)(e_\alpha) \exp \bra \alpha, -\sum_{T} W_{T}(Z)\G_{T}(z, Z)\ket, 
\end{equation}   
where we sum over arbitrary (decorated, rooted) trees (and use $\bra x, \alpha\ket = \bra \alpha, - x\ket$ for all $x\in \Gamma$). This is the analogue of \cite{gmn} equation (C.26). Our discussion so far can be summarised in the following result.

\begin{lem}\label{lem:RHsolution}
The fixed point $X(z, Z)$ defined by \eqref{iterativeSol} is an algebra automorphism which solves RH, i.e. it satisfies the conditions (1) - (4) of Section \ref{sec:gmn:sub:start}. The automorphism $g(Z) = \lim_{z \to 0} X e^{-Z/z}$ of condition (4) is given by
\begin{equation}\label{eq:gZ}
g(Z)(e_{\alpha}) = e_\alpha * \exp \bra \alpha, -\sum_{T} W_{T}\G^0_{T}(Z)\ket, 
\end{equation}
where $\G^0_{T}$ is defined by
\begin{equation}\label{propagatorsgZ}
\G^0_{\T}(Z) = \frac{1}{4\pi i}\int_{\ell_{\gamma_{\T}}} \frac{d\z'}{\z'} X^0(z', Z)(e_{\gamma_{\T}}) \prod_{\T'}\G_{\T'}(\z', Z),
\end{equation}
\end{lem}
\begin{proof}
That $X(z, Z)$ satisfies conditions (1) - (3) of Section \ref{sec:gmn:sub:start} follows at once from its construction as a fixed point of $\mathcal{Z}$ in Section \ref{sec:GMN:sub:fixed} combined with Lemma \ref{AutoLemma} and Lemma \ref{JumpLemma}.

It remains to be seen that the limit \eqref{eq:gZ} exists and is given by \eqref{eq:gZ}, \eqref{propagatorsgZ} (recall that this means that the limit holds in an arbitrary, finite dimensional quotient $\g_{\leqslant k}$, see Remark \ref{RHrmk}). The argument is very similar to the proof of Lemma \ref{lem:admis}: by the basic estimates of Section \ref{BesselLemma}, for each tree $T$ we can take the limit as $z \to 0$ in a non-tangential direction under the integral \eqref{propagators} which defines $G_T(z, Z)$, since the integrands as $z$ tends to $0$ in this way are all bounded by the same integrable function. It is straighforward to check that the bound is uniform in a fixed, arbitrary $\g_{\leqslant k}$ . Then the expression \eqref{propagatorsgZ} follows simply by $\lim_{z \to 0} \rho(z, z') = \frac{1}{4\pi i}$.
\end{proof}

\subsection{Definition of the framed connection}  

We define our framed connection $(\nabla(Z),g(Z))$ with automorphism $g(Z)$ given by \eqref{eq:gZ} and
\begin{equation*}
\nabla(Z) = d - (\partial_z X)X^{-1}dz,
\end{equation*}
where $X = X(z, Z)$ is given by \eqref{iterativeSol}. Note that a priori the connection form $(\partial_z X)X^{-1} dz$ (projected to an arbitrary $\g_{\leqslant k}$) is only well defined in the complement of the rays $\ell_{\gamma}(Z)$ with $a(\gamma, Z) \neq 0$. However we claim that the connection form extends to all $\C^*$, defining a meromorphic connection on $\PP^1$ with poles at $z = 0,\infty$. As we already mentioned, the reason it extends is that the jump \eqref{eq:stokesdata2} across a ray $\ell_{\gamma}(Z)$ is independent of $z$, so the jumps cancel out in the expression $(\partial_z X)X^{-1}$. We show now that the resulting connection fulfills the requirements of Theorem \eqref{prop:GMNconstruction}.

We first prove that $\nabla(Z)$ has a double pole at $z = 0$ and $z = \infty$, and therefore is of the form \eqref{eq:GMNform}. Consider the $\Aut^*(\widehat{\g})$-valued map
\begin{equation*}
Y = X (X^0)^{-1},
\end{equation*}
given explicitly by
\begin{equation*}
Y(z, Z)(e_{\alpha}) =  e_\alpha \exp \bra \alpha, - \sum_{\T} \W_{\T}\G_{\T}(\z, Z)\ket.
\end{equation*}
In each finite-dimensional quotient $\g_{\leqslant k}$ it makes sense to consider a sector $\Sigma$ between consecutive rays $\ell_{\gamma}$ with $a(\gamma) \neq 0$. In a fixed $\g_{\leqslant k}$, and inside $\Sigma$, $Y_{\leqslant k}$ is holomorphic, and exactly as in the proof of Lemma \ref{lem:RHsolution} (i.e. by the basic estimates of Section \ref{BesselLemma}) we have
\begin{equation}\label{eq:limitY}
\lim_{z \to 0} Y(z) = g(Z),\, \lim_{z \to \infty} Y(z) = Y_{\infty}(Z)
\end{equation} 
where
\begin{align}\label{eq:Yinfty}
\nonumber Y_{\infty}(Z)(e_{\alpha}) &= e_\alpha * \exp \bra \alpha, -\sum_{T} W_{T}\G^{\infty}_{T}(Z)\ket, \\
\G^{\infty}_{\T}(Z) &= - G^0_{T}(Z) = -\frac{1}{4\pi i}\int_{\ell_{\gamma_{\T}}} \frac{d\z'}{\z'} X^0(z', Z)(e_{\gamma_{\T}}) \prod_{\T'}\G_{\T'}(\z', Z).
\end{align}
The opposite sign is due simply to the integration kernel:
\begin{equation*}
\lim_{z \to 0} \rho(z, z') = \frac{1}{4\pi i},\,\lim_{z \to \infty} \rho(z, z') = -\frac{1}{4\pi i}.
\end{equation*}

Consider the meromorphic connection $\nabla^0(Z)$ on $\PP^1$ given by
\begin{equation*}
\nabla^0 = d - \left(- \frac{Z}{z^2} + \bar{Z}\right) dz.
\end{equation*}
and notice that $X^0$ provides a fundamental solution to $\nabla^0$. Thus we find an alternative description of $\nabla(Z)$ as a gauge transformation of $\nabla^0$ inside the sector $\Sigma$, that is
\begin{align*}
\nabla(Z) = Y \cdot \nabla^0 = d - Y \left(- \frac{Z}{z^2} + \bar{Z}\right) Y^{-1} + (\partial_z Y)Y^{-1}.
\end{align*}
From the existence of the limits \eqref{eq:limitY} we deduce that the restriction of $\nabla(Z)$ to $\Sigma$ has the same behaviour at $0$, $\infty$ as $\nabla^0$ (i.e. double poles). But since we already proved that $\nabla(Z)$ extends to all $\C^*$, by Liouville's Theorem it actually takes the form
\begin{equation*}
\nabla(Z) = d - \left(\frac{1}{z^2} \A^{(-1)} + \frac{1}{z}\A^{(0)} + \A^{(1)}\right)dz 
\end{equation*}
for some $\A^{(i)} \in D^*(\widehat{\g})$. Moreover $\nabla(Z) = Y \cdot \nabla^0$ gives $g(Z)^{-1} \cdot \A^{(-1)} = - Z$, so $(\nabla(Z),g(Z))$ is a compatibly framed connection with order two poles at $0$ and $\infty$, and no other singularities.

\begin{remark} 
The automorphism $g(Z)$ lies in the subset of $\Aut^*(\widehat{\g})$ given by elements of the form $e_{\alpha} \mapsto e_{\alpha} \exp_*(\bra \alpha, x \ket)$ for some $x \in \Gamma\otimes\widehat{\g}$. This subset of automorphisms is preserved by inversion (by the Lagrange formula, see Section \ref{subsec:explicitformula}). There is an involution $( - )^*$ on this set, acting by $x \mapsto - x$. According to \eqref{eq:Yinfty} we have $g(Z) = Y^*_\infty(Z)$. On the other hand the relation $\nabla(Z) = Y \cdot \nabla^0$ gives $Y^{-1}_{\infty} \cdot \A^{(1)} = \bar{Z}$. Using $g^{-1} \cdot \A^{(-1)} = - Z$ we find that $\A^{(-1)}$ uniquely determines $\A^{(1)}$:
\begin{align}\label{conjugation}
\nonumber \A^{(1)} &= Y_{\infty} \cdot \bar{Z}\\
&= g^* \cdot -\overline{g^{-1} \cdot \A^{(-1)}}.
\end{align}
\end{remark}

\subsection{Formal type and Stokes data}\label{sec:GMN:sub:Stokes} 

By expanding the kernel $\rho(z, z')$ as a formal power series in $z$ around $z = 0$ we can regard $Y$ as a formal gauge transformation (an element of $\Aut^*(\widehat{\g}[[z]])$), taking the germ of $\nabla^0(Z)$ at $0$ to the germ of $\nabla(Z)$. In other words the formal type of $\nabla(Z)$ at $0$ is the type of $\nabla^0$. The gauge transformation $h(z)$ acting by
\begin{equation*}
h(z)(e_{\alpha}) = e_{\alpha}\exp_*(-z \bar{Z}(\alpha))
\end{equation*}
is well defined near $z = 0$ (it has an essential singularity at $\infty$), and it takes the formal type of $\nabla^0$ at $0$ to  
\begin{equation*}
d + \frac{Z}{z^2} dz.
\end{equation*}   
This proves that the $\nabla(Z)$ has the desired formal type, with anti-Stokes rays given by $\ell_{\gamma}(Z) = - \R_{> 0} Z(\gamma)$ for $\gamma \in \Gamma$.

For each finite dimensional quotient $\g_{\leqslant k}$, the restriction $X_{\leqslant k}$ to a sector $\Sigma$ as above is a fundamental solution of $\nabla(Z)_{\leqslant k}$ with asymptotics as $z \to 0$ given by condition $(4)$ in Section \ref{sec:gmn:sub:start}. These solutions differ by the action of \eqref{eq:stokesdata2} along an anti-Stokes ray $\ell$. To prove that these automorphisms are actually the Stokes factors of $\nabla(Z)$ it is enough to show that a solution given by $X_{\leqslant k}|_{\Sigma}$ induces, by analytic continuation, a fundamental solution on a supersector $\widehat{\Sigma}$ preserving the asymptotics. We will perform a formally identical check later in Section \ref{sec:limit:sub:Stokes}, so we do not reproduce the argument here.

Similarly we can calculate the Stokes data at $\infty$. Setting $w = \frac{1}{z}$ and arguing as before, we see that the formal type of $\nabla$ at $\infty$ is
\begin{equation*}
d + \frac{\bar{Z}}{w^2} dw,
\end{equation*}   
with anti-Stokes rays given by $\ell_{\gamma}(\bar Z) = - \R_{> 0} \bar{Z}(\gamma)$ for $\gamma \in \Gamma$. We claim that the attached Stokes factors are given again by \eqref{eq:stokesdata2}. To see this consider the flat section $\bar{X}(w, Z) = X(w^{-1}, Z)$. 
We have 
\begin{equation*} 
\bar{X}(w, Z)(e_{\alpha}) = X^0(w^{-1})(e_\alpha) \exp \bra \alpha, -\sum_{\T} \W_{\T} G_{T}(w^{-1})\ket, 
\end{equation*}
with 
\begin{equation*} 
G_{T}(w^{-1}, Z) = \int_{- \R_{>0}Z(\gamma_{T})} \frac{d z'}{\z'}\frac{z' + w^{-1}}{z' - w^{-1}} X^0(z', Z)(e_{\gamma_{\T}}) \prod_{ T'} G_{ T'}( z', Z).
\end{equation*}
Making the change of variable $z' = \frac{1}{w'}$ we can rewrite 
\begin{equation*} 
G_{T}(w^{-1}, Z) = \int_{- \R_{>0}\bar{Z}(\gamma_{T})} \frac{d w'}{w'}\frac{w' + w}{w' - w} X^0(w'^{-1}, Z)(e_{\gamma_{\T}}) \prod_{ T'} G_{ T'}( w'^{-1}, Z).
\end{equation*}
Keeping in mind the construction of a fixed point of $\mathcal{Z}$ in Section \ref{sec:GMN:sub:fixed} (in particular \eqref{iterativeSol}) we realise that $\bar{X}(w, Z)$ is a fixed point of an operator formally identical to $\mathcal{Z}$, obtained by replacing the integration contours $\ell_{\gamma}(Z)$ with $\ell_{\gamma}(\bar{Z})$ and the reference function $X^0(z, Z)$ with $X^0(w^{-1}, Z) = \exp(w Z + w^{-1} \bar{Z})$. Then the proof of Lemma \ref{JumpLemma} shows that the jump of $\bar{X}(w, Z)$ across $\ell_{\gamma}(\bar{Z})$ is the same as the jump of $X(z, Z)$ across $\ell_{\gamma}(Z)$.

\subsection{Extension and isomonodromy of $\nabla(Z)$}\label{sec:gmn:sub:finish} So far we have constructed $\nabla(Z)$ under the assumption that $(Z,a(Z))$ is generic. Now we show that the flat sections $X(z, Z)$ extend smoothly in $Z$ across the nongeneric locus. It follows that $\nabla(Z)$ is a real analytic isomonodromic family. 

\begin{lem}
For fixed $z$, the function $X(z, Z)$ is continuous in the central charge $Z$ as long as no rays $\ell_{\gamma}(Z)$ with $\Omega(\gamma, Z) \neq 0$ hit $z$.
\end{lem}

\begin{proof}
Suppose that $Z_0$ is a non-generic central charge: there are anti-Stokes rays $\ell_{\gamma}(Z)$ and $\ell_{\eta}(Z)$ that come together for $Z = Z_0$ with $\bra \gamma, \eta \ket \neq 0$. We assume that $\gamma, \eta \in \Gamma$ are primitive and that the restriction of $Z_0$ to the complement in $\Gamma$ of the sublattice spanned by $\gamma, \eta$ is generic. Let $(Z(t), a(t))$ be a continuous family of stability data on $\g$ para\-metrised by $[0, 1]$, with $Z(t_0) = Z_0$ for some $t_0 \in (0, 1)$, and $V$ be a strictly convex sector that contains $\ell_{\gamma}, \ell_{\eta}$ for all $t$. For definiteness we suppose that the $Z_t$ counterclockwise order is $\ell_{\gamma}, \ell_{\eta}$ for $t < t_0$, respectively $\ell_{\eta}, \ell_{\gamma}$ for $t > t_0$. Thus the rays $\ell_{\gamma}, \ell_{\eta}$ coincide only when $t = t_0$. We assume that all other anti-Stokes rays remain constant. All these assumptions are enough to describe what happens generically on the wall of marginal stability; the extension to the whole wall is very similar. Let us also write $Z^+$ (and $Z^-$) for a central charge $Z(t)$ with $t > t_0$ (respectively $t < t_0$) and $\Omega^{\pm}$ for the corresponding (locally constant) spectra.

We see from its construction that $X(z, t) = X(z, Z(t))$ has finite limits when $t \to t^{\pm}_0$. These limits $X(z, t^{\pm}_0)$ are sectionally holomorphic functions of $z \in \C^*$ with values in $\Aut^*(\widehat{\g})$. Their jumps across all rays distinct from $\ell_{\gamma}, \ell_{\eta}$ are the same. The jump of $X(t^{-}_0)$ across $\ell_{\gamma}(t_0) = \ell_{\eta}(t_0)$ is given by $\prod^{\to, Z^-}_{\ell_{\gamma'} \subset V} T^{\Omega^-(\gamma')}_{\gamma'}$. Similarly the jump of $X(t^{+}_0)$ across the same ray is given by $\prod^{\to, Z^+}_{\ell_{\gamma'} \subset V} T^{\Omega^+(\gamma')}_{\gamma'}$. As the family $(Z(t), a(t))$ is continuous by assumption, we have $\prod^{\to, Z^-}_{\ell_{\gamma'} \subset V} T^{\Omega^-(\gamma')}_{\gamma'} = \prod^{\to, Z^+}_{\ell_{\gamma'} \subset V} T^{\Omega^+(\gamma')}_{\gamma'}$. Therefore the composition $X^{-1}(z, t^-_0) \circ X(z, t^+_0)$ is holomorphic in the variable $z \in \C^*$. Recall that we have $X(z, Z) = Y(z, Z) X^0(z, Z)$ and that $Y(Z)$ has well-defined $z \to 0$, $z \to \infty$ limits. Applying Liouville's Theorem to any finite dimensional quotient $\g_{\leqslant k}$ it follows that $X^{-1}(z, t^-_0) \circ X(z, t^+_0)$ is a $z$-constant automorphism of $\widehat{\g}$. We claim that in fact this is the identity automorphism for all $Z$. To see this notice that $X(z, t^{\pm}_0)$ do not just solve the Riemann-Hilbert problem of Section \ref{sec:gmn:sub:start}: by construction, they are both fixed points of the \emph{same} integral operator $\mathcal{Z}$ given by \eqref{preTBop} and \eqref{TBop}, where for $\ell = \ell_{\gamma}(t_0) = \ell_{\eta}(t_0)$ we set 
\begin{equation*}
S_{\ell} = \prod^{\to, Z^-}_{\ell_{\gamma'} \subset V} T^{\Omega^-(\gamma')}_{\gamma'} = \prod^{\to, Z^+}_{\ell_{\gamma'} \subset V} T^{\Omega^+(\gamma')}_{\gamma'}.
\end{equation*} 
So we have two fixed points of $\mathcal{Z}$ differing by a $z$-constant gauge transformation $h(Z) \in \Aut(\widehat{\g})$. Let $X(z, Z)$ be any fixed point of $\mathcal{Z}$. Then we have
\begin{equation}\label{FixedPointGauge}
h(Z) X(z, Z) (e_{\alpha}) = h(Z) \mathcal{Z}[X](z, Z) (e_{\alpha})
\end{equation}
and since $h(Z)$ is a constant algebra endomorphism the left hand side of \eqref{FixedPointGauge} can be written as 
\begin{equation*}
h(Z)(e_{\alpha})\exp_*\left(L(\alpha) + \sum_{\gamma} \Omega(\gamma)\bra \gamma, \alpha\ket \int_{\ell_{\gamma}(Z)}\frac{dz'}{z'}\rho(z, z')\log_*(1 - h(Z)X(z', Z)(e_{\gamma}))\right).
\end{equation*}
On the other hand if $h(Z) X(z, Z) (e_{\alpha})$ is still a fixed point of $\mathcal{Z}$ the right hand side of \eqref{FixedPointGauge} equals $\mathcal{Z}[h X](z, Z)(e_{\alpha})$, that is
\begin{equation*}
e_{\alpha}\exp_*\left(L(\alpha) + \sum_{\gamma} \Omega(\gamma)\bra \gamma, \alpha\ket \int_{\ell_{\gamma}(Z)}\frac{dz'}{z'}\rho(z, z')\log_*(1 - h(Z)X(z', Z)(e_{\gamma}))\right).
\end{equation*}
Comparing these two expressions we see that $h(Z) X(z, Z)$ can only be a fixed point of $\mathcal{Z}$ is $h(Z)(e_{\alpha}) = e_{\alpha}$ for arbitrary $\alpha$, i.e. $h(Z)$ is the identity.
\end{proof}
We need to improve this to real analyticity in $Z$. The key observation is that there exists an $\Aut(\widehat{\g})$-valued holomorphic function $\widetilde{X}(z, Z, W)$ in a nonempty open subset of $\Hom(\Gamma, \C) \times \Hom(\Gamma, \C)$ such that $X(z, Z) = \widetilde{X}(z, Z, \bar{Z})$. The function $\widetilde{X}(z, Z, W)$ is given simply by replacing the factors of $X^0(z, Z)$ in the expression \eqref{iterativeSol} with the new automorphisms 
\begin{equation*}
X^0(z, Z, W)(e_{\alpha}) = e_{\alpha}\exp_*(z^{-1} Z(\alpha) + z W(\alpha)).
\end{equation*}   
The resulting $\widetilde{X}(z, Z, W)$ is well defined and holomorphic in the open subset $A \subset \Hom(\Gamma, \C) \times \Hom(\Gamma, \C)$ where $W$ lies in a sufficiently small open neighbourhood of $\bar{Z}$ (depending on $Z$), and is a fixed point of the obvious analogue of the operator \eqref{TBop}. We denote by $\Hom^0(\Gamma, \C) \subset \Hom(\Gamma, \C)$ the open locus of generic elements for our fixed family of positive stability data. Thus arguing precisely as above we conclude that $\widetilde{X}(z, Z, W)$ is a bounded analytic function in the open subset $A \cap (\Hom^0(\Gamma, \C) \times \Hom(\Gamma, \C))$ which extends to a continuous function on $A$ (i.e. across the non-generic locus). We want to show that it is in fact holomorphic. Note that the non-generic locus $A \cap ((\Hom(\Gamma, \C) \setminus \Hom^0(\Gamma, \C)) \times \Hom(\Gamma,\C)) \subset A$ is generically a submanifold of real codimension $1$. By Hartogs' Theorem it suffices to show that $\widetilde{X}(z, Z, W)$ is holomorphic across this smooth part of the non-generic locus. Let us denote by $f(\zeta)$ the restriction of $\widetilde{X}(z, Z, W)$ to the locus where all but one of the components of either $Z$ or $W$ are fixed to some values. This is a bounded, continuous complex function in an open set $\Omega \subset \C$. Since the non-generic locus is generically smooth, for generic fixed components of $Z$, $W$ the function $f(\zeta)$ is holomorphic outside the support of a smooth arc $s \subset \Omega$. Let $\pi$ denote a loop in $\Omega$. By writing $\pi$ as the sum of loops with a common, oppositely oriented part lying on $s$ if necessary, one checks that for all sufficiently small loop $\pi \subset \Omega$ we have $\int_{\pi} f(\zeta) d\zeta = 0$, so by Morera's Theorem $f(\zeta)$ is holomorphic in all $\Omega$. It follows that $\widetilde{X}(z, Z, W)$ is separately analytic in all the components of $Z$, $W$ and so analytic.

\subsection{Explicit formula for $\nabla(Z)$}\label{subsec:explicitformula}

We provide now a more explicit formula for the $D^*(\widehat{\g})$-valued connection form
\begin{equation*}
\A\, dz = \left(\frac{1}{z^2} \A^{(-1)}(Z) + \frac{1}{z}\A^{(0)}(Z) + \A^{(1)}(Z)\right) dz
\end{equation*}
in terms of a (complicated) sum of graph integrals. Recall that locally $\mathcal{A}$ is given by the composition of linear maps 
\begin{equation}\label{eq:AX}
\mathcal{A} = -(\del_z X) X^{-1},
\end{equation}
where 
\begin{align*}
\del_z X (e_{\alpha}) &= \del_z (X(z)(e_{\alpha}))\\
&= X(z)(e_{\alpha}) (\del_z L(z, Z)(\alpha) - \bra \alpha, \sum_{T} W_T \del_z G_T(z, Z)\ket). 
\end{align*}
To derive our formula we will use that, because of its specific form, $X(z)$ can be inverted explicitly via multivariate Lagrange inversion (see e.g. \cite{gessel}).

Let $m = \operatorname{rank}(\Gamma)$. Choose a basis $\gamma_i$ for $\Gamma$, inducing an isomorphism $\Gamma \cong \Z^{m}$, and consider the associated basis of derivations $\partial_i$ given by the dual basis, so that $\partial_i(e_{\gamma_j}) = \delta_{ij}$ (see Section \ref{sec:gdef}). We can write uniquely
\begin{equation*}
\mathcal{A} = \sum_{i=1}^m \mathcal{A}(e_{\gamma_i})\partial_i.
\end{equation*}
The main difficulty in computing $\mathcal{A}(e_{\gamma_i})$ is to find the inverse image $X^{-1}(e_{\gamma_i})$. For this we consider an auxiliary problem, and introduce variables 
\begin{align*}
s_i &= e_{\gamma_i},\\
{\bf s} &= (s_1, \ldots, s_{m}). 
\end{align*} 
Using our fixed isomorphism we can regard $\alpha \in \Gamma_{> 0}$ as a positive vector in $\alpha_i \in \Z^m$. Then we can attach to $\alpha \in \Gamma_{> 0}$ the monomial in ${\bf s}$ given by
\begin{equation*}
{\bf s}^{\alpha} = \prod^m_{i = 1} s^{\alpha_i}_i \in \C[{\bf s}],
\end{equation*}
For $i = 1, \ldots, m$ define  
\begin{equation*}
\Phi_i({\bf s}) = L(z, Z)(\gamma_i) - \sum_{T} \bra \gamma_i,  W_T(Z) \ket G_T(z, Z) \in \C[[{\bf s}]]
\end{equation*}
where we replace replace $e_{\alpha}$ by ${\bf s}^{\alpha}$ throughout. Our problem is reduced to that of inverting the function
\begin{equation*}
\mathbf{s} \mapsto (s_1 \exp\left(\Phi_1({\bf s})\right), \ldots, s_m \exp\left(\Phi_m({\bf s})\right)).
\end{equation*}
If we suppose that we can solve the equations
\begin{equation}\label{LagrangeEqu}
E_i({\bf s}) = s_i \exp\left(-\Phi_i(E_1({\bf s}), \ldots, E_{ m}({\bf s}))\right),
\end{equation} 
then one checks directly that the inverse is given by 
\begin{equation*}
\mathbf{s} \mapsto (E_1({\bf s} ), \ldots, E_{ m}({\bf s})).
\end{equation*}
Given a formal power series $F \in \C[[\bf s]]$, we denote the coefficient of ${\bf s}^{\alpha}$ in $F({\bf s})$ by the symbol $[{\bf s}^{\alpha}] F$.

\begin{lem} There exist unique $E_i({\bf s}) \in \C[[{\bf s}]]$ solving \eqref{LagrangeEqu}. Moreover the coefficient of ${\bf s}^{\alpha}$ in $E_i({\bf s})$ is given by 
\begin{equation}\label{GoodEqu}
[{\bf s}^{\alpha}] E_i({\bf s} ) = [{\bf s}^{\alpha}] \det(\delta_{p q} + s_p \del_{q} \Phi_p({\bf s} )) s_i \exp\left(-\sum_j \alpha_j \Phi_j({\bf s} )\right).
\end{equation}
\end{lem}
\begin{proof} Applying Good's multivariate Lagrange inversion formula (see e.g. \cite{gessel} Theorem 3, equation (4.5)\footnote{When comparing to equation (4.5) in loc. cit. one should set $g_i({\bf s}) = \exp(- \Phi_i({\bf s}))$ which leads to a simplification and relative minus sign in the determinant $K$.}) over the ground field $\C$ shows that there exists a unique solution $(E_1, \ldots, E_{ m})$ of \eqref{LagrangeEqu} where $E_i \in \C[[\bf s]]$ are given by \eqref{GoodEqu}.  
\end{proof}
\begin{corollary} For $i = 1, \ldots, m$ we have 
\begin{align*}
X^{-1}(e_{\gamma_i}) &= \sum_{\alpha \in \Gamma_{> 0}} e_{\alpha}[{\bf s}^{  \alpha}] E_i({\bf s} )\\ 
&= \sum_{\alpha \in \Gamma_{> 0}} e_{\alpha}[{\bf s}^{\alpha}] \det(\delta_{p q} + s_p \del_{q} \Phi_p({\bf s} )) s_i \exp\left(-\sum_j \alpha_j \Phi_j({\bf s} )\right).  
\end{align*}
\end{corollary}

From the previous result and \eqref{eq:AX} we now obtain the desired formula for $\mathcal{A}$.

\begin{proposition}\label{cor:explicitA} For $i = 1, \ldots, m$ we have
\begin{align}\label{explicitA}
\nonumber \mathcal{A}(e_{\gamma_i}) = &-\sum_{\alpha \in \Gamma_{> 0}}  [{\bf s}^{\alpha}] \det(\delta_{p q} + s_p \del_{q} \Phi_p({\bf s} )) s_i \exp\left(-\sum_j \alpha_j \Phi_j({\bf s} )\right) \\& X(e_{\alpha}) (\del_z L(z, Z)(\alpha) - \bra \alpha, \sum_{T} W_T(Z) \del_z G_T(z, Z)\ket). 
\end{align}
\end{proposition}

\begin{remark} It is shown in \cite{us} that a similar explicit formula can also be derived in the case of a symmetric spectrum.
\end{remark}
\begin{example}[Single-ray solution]\label{sec:GMN:sub:OV}
 There is a basic example given by $\Gamma \cong \Z^2$ generated by $\gamma, \eta$ with $\bra \gamma, \eta\ket = 1$, and with stability data such that $\Omega(\gamma) = 1$ with all other $\Omega$ vanishing. In particular the only decorated trees $T$ we need to consider are single-vertex trees $\{k \gamma\}\bullet$, $k > 0$, with $W_{\{k \gamma\}\bullet} = \dt(k\gamma)k\gamma = \frac{1}{k}\gamma$. This is analogous to the Ooguri-Vafa solution of \cite{gmn}. Using \eqref{explicitA} one can derive explicit expressions for the connections $\nabla(Z)$:
\begin{align}\label{OVconnection}
\nonumber \A^{(-1)}(Z) &= Z - Z (\gamma) \sum_{k > 0} \frac{1}{4 k \pi i }\int_{\ell_{\gamma}} \frac{dz'}{z'} \exp(L(z', Z)(k\gamma)) \ad(e_{k\gamma}),\\
\nonumber \A^{(0)}(Z)  &= - (Z(\gamma) + \bar{Z}(\gamma)) \sum_{k > 0} \frac{1}{2 k \pi i }\int_{\ell_{\gamma}} dz' \exp(L(z', Z)(k\gamma)) \ad(e_{k\gamma}),\\ 
\A^{(1)}(Z)  &= -\bar{Z} + \bar{Z} (\gamma) \sum_{k > 0} \frac{1}{4 k \pi i }\int_{\ell_{\gamma}} \frac{dz'}{z'} \exp(L(z', Z)(k\gamma)) \ad(e_{k\gamma}).
\end{align}
Indeed in this case we have $s_1 = e_{\gamma}$, $s_2 = e_{\eta}$,
\begin{align*}
\Phi_1(s_1, s_2) &= L(z, Z)(\gamma) = z^{-1} Z(\gamma) + z \bar{Z}(\gamma),\\
\Phi_2(s_1, s_2) &= L(z, Z)(\eta) - \sum_{k > 0} \bra \eta, \dt(k\gamma)k\gamma\ket G_{\{k\gamma\}\bullet}(z, Z)\\
&= z^{-1} Z(\eta) + z \bar{Z}(\eta) + \sum_{k > 0}  s^k_1 \frac{1}{k} \int_{\ell_{\gamma}(Z)} \frac{dz'}{z'} \rho(z,z')\exp(L(z', Z)(k\gamma)).
\end{align*}
So $\Phi_1(s_1, s_2)$ is in fact a constant $\Phi_1$ and $\Phi_2(s_1, s_2)$ is a formal power series in $s_1$. The $2 \times 2$ matrix $\delta_{p q} + s_p \del_{q} \Phi_p({\bf s} )$ is given by
\begin{equation*}
\left(\begin{matrix} 
1 + s_1 \del_{s_1} \Phi_{1} & s_1 \del_{s_2}\Phi_{1}\\
s_2 \del_{s_1}\Phi_{2} & 1 + s_2 \del_{s_2} \Phi_{2}
\end{matrix}
\right)
=
\left(
\begin{matrix}
1 & 0\\
s_2 \del_{s_1}\Phi_{2} & 1
\end{matrix}
\right)
\end{equation*}
and so has determinant $1$. According to \eqref{explicitA} we have
\begin{align*} 
\A(e_{\gamma}) &= -\sum_{h, k} [s^h_1 s^k_2] s_{1} \exp(- h \Phi_1(s_1, s_2) - k \Phi_2(s_1, s_2))\\
& X(e_{h\gamma + k\eta}) (\del_z L(z, Z)(h\gamma + k\eta) - \bra h\gamma + k\eta , \sum_{T} W_T(Z) \del_z G_T(z, Z)\ket).
\end{align*}
Since $\Phi_1$ is a constant and $\Phi_2$ only depends on $s_1$ the only nonvanishing term appears for $h = 1, k = 0$. Moreover since the only nonvanishing weights $W_T(Z)$ are $W_{\{k\gamma\}\bullet}(Z)$ and are proportional to $\gamma$ the contribution from the sum over trees vanishes. So we find
\begin{align*}
\A(e_{\gamma}) &= -\exp(-\Phi_1) X(e_{\gamma}) \del_z L(z, Z)(\gamma)\\
&= (z^{-2} Z(\gamma) - \bar{Z}(\gamma)) e_{\gamma}
\end{align*}
which is precisely the action of \eqref{OVconnection} on $e_{\gamma}$. A similar but more involved computation yields the action on $e_{\eta}$:
\begin{align*}
\A(e_{\eta}) &= \sum_{h, k} [s^h_1 s^k_2] s_{2} \exp(- h \Phi_1 - k \Phi_2(s_1))\\
& X(e_{h\gamma + k\eta}) (\del_z L(z, Z)(h\gamma + k\eta) - \bra h\gamma + k\eta , \sum_{T} W_T(Z) \del_z G_T(z, Z)\ket)\\
&=\sum_{h \geq 0} [s^h_1] \exp(- h \Phi_1 - \Phi_2(s_1))\\
& X(e_{h\gamma + \eta}) (\del_z L(z, Z)(h\gamma + \eta) - \bra h\gamma + \eta , \sum_{k>0} W_{\{k\gamma\}\bullet}(Z) \del_z G_{\{k\gamma\}\bullet}(z, Z)\ket)\\
&= X(e_{\eta})\sum_{h \geq 0} [s^h_1] \exp(- h \Phi_1 - \Phi_2(s_1))\\
& X(e_{h\gamma}) (\del_z L(z, Z)(h\gamma + \eta) - \bra \eta , \sum_{k>0} W_{\{k\gamma\}\bullet}(Z) \del_z G_{\{k\gamma\}\bullet}(z, Z)\ket)\\ 
&= X(e_{\eta})\sum_{h \geq 0} [s^h_1] \exp(- h \Phi_1 - \Phi_2(s_1))\\
& \exp(h\Phi_1) (e_{\gamma})^h (h\del_z L(z, Z)(\gamma) + \del_z L(z, Z)(\eta) + \sum_{k>0} \frac{1}{k} \del_z G_{\{k\gamma\}\bullet}(z, Z))\\
& = (\del_z L(z, Z)(\eta) + \sum_{k>0} \frac{1}{k} \del_z G_{\{k\gamma\}\bullet}(z, Z)) X(e_{\eta}) \sum_{h \geq 0} [s^h_1] \exp(- \Phi_2(s_1))(e_{\gamma})^h+\\
&+ \del_z L(z, Z)(\gamma) X(e_{\eta}) \sum_{h \geq 0} h [s^h_1] \exp(- \Phi_2(s_1))(e_{\gamma})^h\\
& = (\del_z L(z, Z)(\eta) + \sum_{k>0} \frac{1}{k} \del_z G_{\{k\gamma\}\bullet}(z, Z)) X(e_{\eta}) \exp(-\Phi_2(s_1))|_{s_1 = e_{\gamma}}+\\
&+ \del_z L(z, Z)(\gamma) X(e_{\eta}) e_{\gamma} \del_{s_1} \exp(-\Phi_2(s_1))|_{s_1 = e_{\gamma}}\\
& = (\del_z L(z, Z)(\eta) + \sum_{k>0} \frac{1}{k} \del_z G_{\{k\gamma\}\bullet}(z, Z)) e_{\eta} - \del_z L(z, Z)(\gamma) e_{\eta} e_{\gamma} \del_{s_1} \Phi_2(s_1)|_{s_1 = e_{\gamma}}\\
&= e_{\eta} (\del_z L(z, Z)(\eta) + \sum_{k>0} \frac{1}{k} \del_z G_{\{k\gamma\}\bullet}(z, Z)) - \del_z L(z, Z)(\gamma)\sum_{k > 0} G_{\{k\gamma\}\bullet}(z, Z)).
\end{align*}
Just as before the term $\del_z L(z, Z)(\eta)$ contributes the action of $z^{-2} Z + \bar{Z}$ on $\eta$ to \eqref{OVconnection}. To see the other contributions to \eqref{OVconnection} we need to integrate by parts in $\del_z G_{\{k\gamma\}\bullet}(z, Z)$, using the special property of the kernel
\begin{equation*}
\del_z \rho(z, z') = - \frac{z'}{z}\del_{z'}\rho(z, z').
\end{equation*}
For a fixed $k > 0$ we find
\begin{align*}
\frac{1}{k}\del_z G_{\{k\gamma\}\bullet}(z, Z) &= \frac{1}{k}\int_{\ell_{\gamma}}\frac{dz'}{z'} \del_{z}\rho(z, z') \exp(L(z',Z)(k\gamma))\\
&= \frac{1}{k z} \int_{\ell_{\gamma}} dz' \rho(z, z') \del_{z'} \exp(L(z',Z)(k\gamma))\\
&= \frac{1}{z} \int_{\ell_{\gamma}} \frac{dz'}{z'} \rho(z, z') \left(-\frac{1}{z'} Z(\gamma) + z' \bar{Z}(\gamma)\right) \exp(L(z',Z)(k\gamma)).
\end{align*}
This must be combined with the corresponding term 
\begin{equation*}
- \del_z L(z, Z)(\gamma) G_{\{k\gamma\}\bullet}(z, Z)) = \frac{1}{z}\int_{\ell_{\gamma}} \frac{dz'}{z'} \rho(z, z') \left(\frac{1}{z} Z(\gamma) - z \bar{Z}(\gamma)\right) \exp(L(z',Z)(k\gamma)),
\end{equation*}
yielding the contribution
\begin{align*}
\frac{1}{z} \frac{1}{4\pi i}\int_{\ell_{\gamma}} \frac{dz'}{z'} \frac{z'+z}{z' - z} \left( \frac{z' - z}{z z'} Z(\gamma) + (z' - z)\bar{Z}(\gamma)\right) \exp(L(z',Z)(k\gamma)). 
\end{align*}
Simplifying and collecting terms gives the fixed $k$ contribution to \eqref{OVconnection}. 
\end{example}

\subsection{Symmetric case}\label{symmetricRemark} At a fixed point in the space of stability data it is possible to allow a symmetric spectrum (i.e. to allow $\Omega(\gamma) = \Omega(-\gamma)$) in the construction above by working over the Poisson $\C[[\tau]]$-algebra $\g[[\tau]]$. The construction goes through verbatim by considering integral operators of the form
\begin{align*} 
\mathcal{Z}'[Y](z)(e_{\alpha})  = & e_{\alpha} \exp_*\Big( L(z)(\alpha) \\ &  + \sum_{\gamma} \Omega(\gamma)\bra \gamma, \alpha\ket \int_{\ell_{\gamma}(Z)}\frac{dz'}{z'}\rho(z, z')\log_*(1 - \tau^{p(\gamma, Z)} Y(z')(e_{\gamma}))\Big),
\end{align*}
for a function $p(\gamma, Z) > 0$ such that $p(m \gamma, Z) = |m| p(\gamma, Z)$ and $p(\gamma, Z) \geq || \gamma ||$, and we still find a fixed point given by
\begin{equation} 
X'(e_{\alpha}) = X^0(e_\alpha) \exp \bra \alpha, -\sum_{\T} \tau^{p(T, Z)}\W_{\T}\G_{\T}\ket  
\end{equation}
with $p(T, Z) = \sum_{v \in T^0} p(\alpha(v), Z)$. This solves a Riemann-Hilbert problem which is formally identical to the one we described, but with monodromy (Poisson) automorphisms which are compositions of operators $T'^{\Omega}_{\gamma}$ acting on $\g[[\tau]]$ by
\begin{equation*}
T'^{\Omega}_{\gamma}(e_{\alpha}) = e_{\alpha}(1 - \tau^{p(\gamma, Z)} e_{\gamma})^{\bra \gamma, \alpha\ket \Omega}.
\end{equation*}
In favorable cases such a pointwise lift of our stability data to $\g[[\tau]]$ can be extended continuously at least in small open subsets of $\mathcal{U}$, according to our Definition \ref{lift} (these lifts were first considered in \cite{gps}). This approach is pursued further in \cite{us} to study convergence problems in the presence of symmetric stability data.   

\section{The $R \to 0$ conformal limit}
\label{sec:limit}
In this Section we prove that, up to a sequence of explicit gauge transformations, our connections $\nabla(Z)$ specialize to $\nabla^{BTL}(Z)$ in the \emph{conformal limit}
\begin{equation*}
Z \mapsto RZ,\quad z = R t,\quad R \to 0
\end{equation*}
discussed in the Introduction, and suggested by physical considerations in \cite{gmn} and \cite{gaiotto}. 

\subsection{Existence and uniqueness of the conformal limit}

Consider the rescaled connections
\begin{equation*}
\nabla_t (RZ) = d - \left(\frac{1}{t^2} R^{-1}\A^{(-1)}(RZ) + \frac{1}{t} \A^{(0)}(RZ) + R \A^{(1)}(RZ) \right)dt.
\end{equation*}
In this Section we show that the frames
\begin{equation*}
g(R) = g(RZ) = \lim_{z \to 0} Y(z, RZ),  
\end{equation*}
given explicitly by
\begin{equation*}
g(RZ)(e_{\alpha}) =  e_\alpha \exp \bra \alpha, - \sum_{\T} \W_{\T}\lim_{z \to 0}\G_{\T}(\z, RZ)\ket,
\end{equation*}
provide a sequence of $z$-constant gauge transformations such that the limit of 
\begin{equation*}
g(R)\cdot \nabla_t(RZ)
\end{equation*} 
as $R \to 0$ exists. In order to compare this limit with the Bridgeland-Toledano Laredo connections we need the following result.
\begin{theorem}[\cite{bt_stab, bt_stokes}]\label{BTLthm}
Let $\Omega$ be a spectrum giving a continuous, positive family of stability data parametrised by $\mathcal{U} \subset \Hom(\Gamma, \C)$. Let $Z \in \mathcal{U}$ correspond to generic stability data. Then there exists a unique meromorphic connection $\nabla^{BTL}(Z)$ on $\PP^1$, with values in $\Hom(\Gamma, \C) \ltimes \g$, of the form
\begin{equation*}
\nabla^{BTL}(Z) = d - \left(\frac{Z}{t^2} + \frac{f(Z)}{z}\right)dt,
\end{equation*}
such that $\nabla^{BTL}(Z)$ has Stokes data (generalized monodromy) at $t = 0$   
given by the rays $\ell_{\gamma}(Z)$ and factors $S_{\ell_{\gamma}(Z)}$ of \eqref{eq:stokesdata}. The family $\nabla^{BTL}(Z)$ extends to a holomorphic isomonodromic family of (trivially framed) meromorphic connections on $\PP^1$ on all of $\mathcal{U}$.
\end{theorem}

\noindent Recall from Section \ref{sec:gdef} that $\Hom(\Gamma, \C) \ltimes \widehat{\g}$ lies naturally in $D^*(\widehat{\g})$ as a Lie subalgebra, and hence $\nabla^{BTL}(Z)$ induces a $D^*(\widehat{\g})$-valued connection with the same Stokes factors as $\nabla(Z)$ in Theorem \ref{prop:GMNconstruction}.

\begin{theorem}\label{thm:limit}
The limit $\lim_{R \to 0} g(R) \cdot \nabla_t(RZ)$ exists and has the form
\begin{equation*}
\widehat{\nabla}_t = d - \left( - \frac{Z}{t^2} + \frac{\tilde{f}(Z)}{t}\right)dt, 
\end{equation*}
with the same Stokes data as $\nabla^{BTL}( Z)$. It follows that $\widehat{\nabla}_t$ actually coincides with the $D^*(\widehat{\g})$-valued connection induced by $\nabla^{BTL}(-Z)$. 
\end{theorem} 
The opposite sign is due to different conventions: if $\nabla^{BTL}(Z)$ is defined for $Z \in \mathbb{H}^n$ then $\nabla(RZ)$ is defined on $(-\mathbb{H})^n$. Notice that in particular Theorem \ref{thm:limit} gives a new proof of the conformal invariance of Joyce functions: these must be invariant under scaling because $\nabla^{BTL}( Z)$ is the limit point of the scaling.

\begin{remark}
Note that the $R \to 0$ limit of $\nabla_t (RZ)$ does not exist if we do not gauge by $g(R)$. This can be seen already in the single-ray solution discussed in Example \ref{sec:GMN:sub:OV}: we have
\begin{equation*}
R^{-1}\A^{(-1)}(R Z) =  Z - Z (\gamma) \sum_{k > 0} \frac{1}{4 k \pi i }\int_{\ell_{\gamma}} \frac{dz'}{z'} \exp(L(z', R Z)(k\gamma)) \ad(e_{k\gamma}).
\end{equation*}
One can check that the integrals $\int_{\ell_{\gamma}} \frac{dz'}{z'} \exp(L(z', R Z)(k\gamma))$ are real, positive and diverge logarithmically as $R \to 0$. 
\end{remark}

\begin{remark}\label{InverseScaling1} One may also consider another limit
 \begin{equation*}
Z \mapsto RZ,\quad z = \frac{t}{R}, \quad R \to 0
\end{equation*}
i.e. the rescaled connections
\begin{equation*}
\nabla'_{t}(RZ) = d - \left(\frac{1}{t^2} R \A^{(-1)}(RZ) + \frac{1}{t} \A^{(0)}(RZ) + R^{-1} \A^{(1)}(RZ) \right)dt.
\end{equation*}
This is very similar to the conformal limit, but it is useful for some purposes. In particular we will use it in the proof of Lemma \ref{LimitConnectionLemma}. We will see that the analogue of Theorem \ref{thm:limit} holds, with the \emph{same} choice of gauge transformations $g(R)$, and with limiting connections of the form
\begin{equation*}
\widehat{\nabla}'_{t} = d - \left(\frac{\tilde{f}(\bar{Z})}{t} + \bar{Z}\right)dt.
\end{equation*}
 \end{remark}

The proof of Theorem \ref{thm:limit} is carried out in several steps in sections  \ref{sec:limit:sub:start} - \ref{sec:limit:sub:Stokes}. An extension to the case of symmetric stability data is briefly discussed in Section \ref{symmetricConformal}. The main technical tool is the introduction of a new integral equation, namely
\begin{align}\label{intEqu}
\nonumber h(t)(e_{\alpha}) = &e_{\alpha} \exp\Big(\frac{1}{2\pi i} \sum_{\alpha'}\bra \alpha, \alpha'\ket\dt(\alpha')\\ &\int_{\ell_{\alpha'}}\frac{d z}{z}\frac{t}{z - t} \exp(  Z(\alpha') \z^{-1} +   R^2 \bar{Z}(\alpha') z) h(z)(e_{\alpha'}) \Big),
\end{align}
which admits a well-defined $R \to 0$ limit 
\begin{equation}\label{intEqulimit}
h(t)(e_{\alpha}) =  e_{\alpha} \exp\Big(\frac{1}{2\pi i} \sum_{\alpha'}\bra \alpha, \alpha'\ket\dt(\alpha') \int_{\ell_{\alpha'}}\frac{d z}{z}\frac{t}{z - t} \exp(  Z(\alpha') \z^{-1}) h(z)(e_{\alpha'}) \Big).
\end{equation}
Solutions of \eqref{intEqulimit} correspond to fundamental solutions for the limiting connection $\widehat{\nabla}_t$. The calculation of the Stokes factors for $\widehat{\nabla}_t$ and the proof of uniqueness are discussed in Section \ref{sec:limit:sub:Stokes}; the latter follows essentially from \cite[Theorem 6.6]{bt_stokes}. 

\subsection{A new integral operator}\label{sec:limit:sub:start} As usual we work in a fixed, arbitrary finite-dimensional quotient $\g_{\leqslant k}$. So it makes sense to consider convex sectors between consecutive anti-Stokes rays with nontrivial $\Omega$. In the following we always denote by $\Sigma$ one such sector, bounded by rays $\ell_{\gamma}, \ell_{\gamma'}$ with $\Omega(\gamma), \Omega(\gamma') \neq 0$. We assume that $\ell_{\gamma'}$ follows $\ell_{\gamma}$ in the clockwise order. The \emph{supersector} $\widehat{\Sigma}$ corresponding to $\Sigma$ is the open region bounded by $e^{-i\frac{\pi}{2}}\ell_{\gamma}, e^{i\frac{\pi}{2}}\ell_{\gamma'}$. 

Our strategy to show that the limit $\widehat{\nabla}_t$ exists is to show first that the local flat sections of $\nabla(RZ)_{z = Rt}$ given by the restriction of $X(R t , R Z)$ to a sector have a finite $R \to 0$ limit after gauging them by $g^{-1}(R)$, i.e. to show that the limit
\begin{equation*}
\widehat{X}(t, Z) = \lim_{R \to 0} g^{-1}(R) X(R t, R Z)
\end{equation*}
exists. Notice that 
\begin{equation*}
\lim_{R \to 0} X^0(R t, R Z) (e_{\alpha}) = e_{\alpha} \exp_*(t^{-1} Z(\alpha) ). 
\end{equation*}
As $X(R t, R Z)$ is the composition $Y(R t, R Z) X^0(R t, R Z)$ in $\Aut^*(\widehat{\g})$, it is enough to consider the limit
\begin{equation*}
\lim_{R\to 0} g^{-1}(R) Y(R t, R Z).
\end{equation*}
In order to study this limit we consider the $\Aut^*(\widehat{\g})$-valued function $h$, holomorphic in $\Sigma$, given by
\begin{equation*}
h(t , Z, R) = g^{-1}(R) Y(R t, R Z).
\end{equation*}
\begin{lem}\label{IntEquationLemma}
The function $h(t, Z, R)$ is a solution of the integral equation
\begin{align}\label{intEqu0}
\nonumber h(t)(e_{\alpha}) = &e_{\alpha} \exp\Big(\frac{1}{2\pi i} \int^t_{t_0} dt'\sum_{\alpha'}\bra \alpha, \alpha'\ket\dt(\alpha')\\ &\int_{\ell_{\alpha'}}\frac{d z'}{( z' - t')^2} \exp(  Z(\alpha') \z'^{-1} +   R^2 \bar{Z}(\alpha') z') h(z')(e_{\alpha'}) \Big),
\end{align}
where the outer integral is computed along some path in the simply connected domain $\Sigma$ starting from a fixed base point $t_0$.
\end{lem}

\begin{proof}
By definition, for all $\alpha \in \Gamma$ we have
\begin{equation*}
Y(R t)(e_{\alpha}) = e_{\alpha} * \exp_* \bra \alpha, \sigma(R t) \ket
\end{equation*}
where $\sigma(R t)$ is a sum of integrals of the form 
\begin{equation}\label{sigma}
\frac{1}{4\pi i}\int_{\ell_{\alpha'}} \frac{d\z'}{\z'} \frac{\z' + Rt}{\z' - Rt}\,I(\z', R),
\end{equation}
one for each graph $\T$ in a class of decorated, rooted trees. Here $I(z', R)$ is still an iterated integral, with values in $\Gamma \otimes \widehat{\g}$, of the form 
\begin{equation*}
\alpha'\dt(\alpha')\exp( R (Z(\alpha') \z'^{-1} + \bar{Z}(\alpha') \z') ) e_{\alpha'} * r(z', R)
\end{equation*}
for some $\g$-valued function $r(z', R)$. By rescaling $z' \to Rz'$, we can rewrite each term in \eqref{sigma} as
\begin{equation*} 
\frac{1}{4\pi i} \int_{\ell_{\alpha'}} \frac{d\z'}{\z'} \frac{\z' + t}{\z' - t}\,I(R\z', R), 
\end{equation*}
so now $I(R\z', R)$ has the form
\begin{equation*}
\alpha'\dt(\alpha')\exp(Z(\alpha') \z'^{-1} + R^2 \bar{Z}(\alpha') \z') e_{\alpha'} * r(z', R).
\end{equation*}
Next we compute
\begin{align*}
\del_t Y(R t)(e_{\alpha}) = \bra \alpha, \sigma'(R t)\ket * Y(R t)(e_{\alpha})
\end{align*}
where $\sigma'(R t)$ is again a sum of terms labelled by the same diagrams $ T$, of the form
\begin{equation*}
\frac{1}{4\pi i}\int_{\ell_{\alpha'}}\frac{d\z'}{\z'} \frac{2 R \z'}{(\z' - R t)^2} I(\z', R) = \frac{1}{2\pi i}\int_{\ell_{\alpha'}}\frac{d\z'}{(\z' - t)^2} I(R\z', R).
\end{equation*}
Going back to $h(t)$, we have
\begin{align*}
\del_t h(t)(e_{\alpha}) &= g^{-1}(R) \del_t  Y(Rt)(e_{\alpha})\\
&= g^{-1}(R) ( \bra \alpha, \sigma'(R t)\ket *  Y(R t)(e_{\alpha}) )\\
&= g^{-1}(R) ( \bra \alpha, \sigma'(R t)\ket ) * h(t)(e_{\alpha}).
\end{align*}

In the last equation we used the algebra automorphism property of $ g^{-1}$. By the same property, the factor $ g^{-1}(R) ( \bra \alpha, \sigma'(R t)\ket ) $ splits into a sum of terms of the form 
\begin{equation*}
 g^{-1}(R)\left( \frac{1}{2\pi i}\int_{\ell_{\alpha'}}\frac{d\z'}{(\z' - t)^2} \bra\alpha, I(R\z', R)\ket \right),
\end{equation*}
one for each $ T$. In each finite-dimensional quotient $\g_{\leqslant k}$, the latter term can be rewritten as
\begin{equation}\label{intermediate}
\frac{1}{2\pi i}\int_{\ell_{\alpha'}}\frac{d\z'}{(\z' - t)^2} g^{-1}( \bra\alpha, I(R\z', R)\ket ).
\end{equation}
Recalling the iterative definition of $I(R\z', R)$, this equals in turn
\begin{equation*}
\frac{1}{2\pi i}\bra \alpha, \alpha'\ket\dt(\alpha')\int_{\ell_{\alpha'}}\frac{d\z'}{(\z' - t)^2} \exp(  Z(\alpha') \z'^{-1} +   R^2 \bar{Z}(\alpha') \z') g^{-1} (e_{\alpha'} \bra \alpha', I'(R\z', R) \ket)
\end{equation*}
for some ``residual" (iterated) integral $I'(R\z', R)$. Indeed, the above rewriting can be seen as the operation of removing the root $a$ (labelled by $\alpha'$) from the fixed tree $\T$ corresponding to \eqref{intermediate}, leaving a finite set of disconnected trees $\T \setminus \{a\}$. Notice that we allow the empty tree in this residual set, corresponding to a factor $1$. Now we let the original $\T$ behind \eqref{intermediate} vary among all trees with root $a$ labelled by $\alpha'$, and sum over all the corresponding integrals, getting for each fixed $\alpha'$,
\begin{align*}
\frac{1}{2\pi i}\bra \alpha, \alpha'\ket\dt(\alpha')\int_{\ell_{\alpha'}}\frac{d\z'}{(\z' - t)^2} & \exp(  Z(\alpha') \z'^{-1} +   R^2 \bar{Z}(\alpha') \z')\\ & g^{-1}(e_{\alpha'} \bra \alpha', \sum_{\rm disconnected} I'(R\z', R) \ket).
\end{align*}
By a standard combinatorial principle
\begin{equation*}
\sum \rm{disconnected}  = \exp \left(\sum \rm{connected}\right),
\end{equation*}
so for each $\alpha'$ the last integral equals 
\begin{align*}
\frac{1}{2\pi i}\bra \alpha, \alpha'\ket\dt(\alpha')\int_{\ell_{\alpha'}}\frac{d\z'}{(\z' - t)^2} & \exp(  Z(\alpha') \z'^{-1} +   R^2 \bar{Z}(\alpha') \z') \\& g^{-1}(e_{\alpha'} \exp \bra \alpha', \sum_{\rm connected} I'(R\z', R) \ket).
\end{align*}
that is  
\begin{equation*}
\frac{1}{2\pi i} \bra \alpha, \alpha'\ket\dt(\alpha')\int_{\ell_{\alpha'}}\frac{d\z'}{(\z' - t)^2} \exp(  Z(\alpha') \z'^{-1} +   R^2 \bar{Z}(\alpha') \z') h( z')(e_{\alpha'}).
\end{equation*}
The upshot is that we have found an integro-differential equation for $h(t)$, namely
\begin{align*}
\del_t h(t)(e_{\alpha}) = &\Big(\frac{1}{2\pi i} \sum_{\alpha'}\bra \alpha, \alpha'\ket\dt(\alpha')\\&\int_{\ell_{\alpha'}}\frac{d\z'}{(\z' - t)^2} \exp(  Z(\alpha') \z'^{-1} +   R^2 \bar{Z}(\alpha') \z') h(\z')(e_{\alpha'}) \Big) * h(t)(e_{\alpha}).
\end{align*}
We can turn this into the integral equation \eqref{intEqu0}.
\end{proof}

The following result follows now from a careful application of Fubini's Theorem.

\begin{lem}\label{FubiniLemma}
The function $h(t, Z, R)$ is a solution of the integral equation \eqref{intEqu}.
\end{lem}
\begin{proof}
Since $h(t,RZ)$ is asymptotic to the identity as $t \to 0,\infty$, for all $R \geq 0$ we have
\begin{equation*}\int^t_{t_0} \int_{\ell_{\gamma_{\T}}} \left|dt' \frac{d z}{(z - t')^2} \exp(  Z(\alpha') z^{-1} +   R^2 \bar{Z}(\alpha') z ) h(z)(e_{\alpha'})\right| < \infty.
\end{equation*}
At the same time, if $\ell_{\alpha'}$ is not contained in $\Sigma$, the integral
\begin{equation*}\int_{\ell_{\alpha'}} \int^t_{t_0} \left|dt' \frac{d z}{(z - t')^2} \exp(  Z(\alpha') z^{-1} + R^2 \bar{Z}(\alpha') z ) h(z)(e_{\alpha'})\right|
\end{equation*}
equals
\begin{equation*}
\int_{\ell_{\alpha'}} dz \exp(Z(\alpha') z^{-1} + R^2 \bar{Z}(\alpha') z ) \left|h(z)(e_{\alpha'})\right|\int^t_{t_0} \left| \frac{dt'}{(z - t')^2}\right|.
\end{equation*}
Since the integration path from $t_0$ to $t$ is compact for a fixed $t$, the inner integral is $O(|z|^{-2})$ as $z \to \infty$. Thefore the second integral is also finite for all $R \geq 0$. Then for a fixed $t$ we can apply Fubini to rewrite
\begin{equation*}
\dt(\alpha')\int^t_{t_0} dt' \int_{\ell_{\alpha'}}\frac{d z}{( z - t')^2} \exp(  Z(\alpha') \z^{-1} +   R^2 \bar{Z}(\alpha') z) h(z)(e_{\alpha'})
\end{equation*}
as
\begin{equation*}\label{Hfunction2}
\dt(\alpha')(t - t_0)\int_{\ell_{\alpha'}}\frac{d z}{(t - z)(t_0 - z)} \exp(Z(\alpha') z^{-1} + R^2 \bar{Z}(\alpha') z ) h(z)(e_{\alpha'}).
\end{equation*}
By a limiting argument, we can now choose $t_0 = 0$ (which strictly speaking is not in $\Sigma$).
\end{proof}
\begin{remark}\label{InverseScaling2} Completely analogous results hold for the alternative scaling introduced in Remark \ref{InverseScaling1}, with the \emph{same} choice of gauge transformations $g(R)$. In particular one considers the limit as $R \to 0$ of flat sections of $g(R) \cdot \nabla'_{t}(RZ)$,
\begin{equation*}
\widehat{X}'(t, Z) = g^{-1}(R) X(R^{-1} t, R Z).
\end{equation*}
Notice that 
\begin{equation*}
\lim_{R \to 0} X^0(R^{-1} t, R Z) (e_{\alpha}) = e_{\alpha} \exp_*(t \bar{Z}(\alpha) ). 
\end{equation*}
As $X(R^{-1} t, R Z)$ is the composition $Y(R^{-1} t, R Z) X^0(R^{-1} t, R Z)$ in $\Aut^*(\widehat{\g})$, it is enough to study the existence of the limit for the function
\begin{equation*}
h'(t, Z, R) =  g^{-1}(R) Y(R^{-1} t, R Z).
\end{equation*}
Again this is done by showing first that $h'(t, Z, R)$ satisfies a suitable integral equation, analogous to \eqref{intEqu}. Indeed following the proofs of Lemma \ref{IntEquationLemma} and Lemma \ref{FubiniLemma} verbatim, replacing the scaling $z' \mapsto Rz'$ with $z' \mapsto R^{-1} z'$, shows that $h'(t, Z, R)$ is a solution to 
\begin{align}\label{InverseIntEqu}
\nonumber h'(t)(e_{\alpha}) = &e_{\alpha} \exp\Big(\frac{1}{2\pi i} \sum_{\alpha'}\bra \alpha, \alpha'\ket\dt(\alpha')\\ &\int_{\ell_{\alpha'}}\frac{d z}{z}\frac{t}{z - t} \exp(R^2Z(\alpha') \z^{-1} + \bar{Z}(\alpha') z) h'(z)(e_{\alpha'}) \Big).
\end{align}
\end{remark}
\subsection{Fixed point}\label{sec:BTLfixed} Of course \eqref{intEqu} looks quite similar to the integral equation for $X(z)$. The advantage is that now it is easy to compute $R \to 0$ limits. To see this we leave aside $h(t, Z, R) = g^{-1}(R) Y(R t, R Z)$ for a moment and consider instead solutions of \eqref{intEqu} obtained by iteration just as in Section \ref{sec:GMN:sub:fixed}. By the definition of $g(R)$ we have, for all $R > 0$, 
\begin{equation*}
\lim_{t\to 0} g^{-1}(R) Y(R t, R Z) = I,
\end{equation*}
so we look at the solution $\tilde{h}(t)$ obtained by iteration starting from $\tilde{h}^{0}(t) = I$. A priori $h, \tilde{h}$ are different solutions. We will prove that in fact $h = \tilde{h}$ and that the latter has a well-defined $R \to 0$ limit, so that the same holds for $h$.
\begin{remark}
Notice that all the integrals
\begin{equation*}
\int_{\ell_{\alpha'}}\frac{d z}{z}\frac{1}{z-t} \exp(Z(\alpha') z^{-1} + R^2 \bar{Z}(\alpha') z ) e_{\alpha'}
\end{equation*} 
(corresponding to $\tilde{h}^{0}(t) = I$) have a well defined $R \to 0$ limit, namely
\begin{equation*}
\int_{\ell_{\alpha'}}\frac{d z}{z}\frac{1}{z-t} \exp(Z(\alpha') z^{-1}) e_{\alpha'}.
\end{equation*} 
So at least the first iteration from $\tilde{h}^{0}(t) = I$ has a well defined limit as $R \to 0$. This does \emph{not} happen for the first iteration of the integral equation for $X$ starting from $X^0$: that is already divergent. 
\end{remark}

Just as in the case of the integral operator $\mathcal{Z}$ of Section \ref{sec:gmn} we have an expression for the iterative solution of \eqref{intEqu}, namely (with the usual notation)
\begin{equation*}
\tilde{h}(t)(e_{\alpha}) = e_{\alpha}\exp \bra \alpha, -\sum_{\T} \W_{\T} \H_{\T}(t)\ket,
\end{equation*}
where
\begin{equation}\label{Hfunction3}
\H_{\T}(t) =  e_{\gamma_{\T}} * \frac{1}{2\pi i} t \int_{\ell_{\gamma_{T}}}\frac{d z}{z} \frac{1}{z - t} \exp(Z(\gamma_{\T}) z^{-1} + R^2 \bar{Z}(\gamma_{\T}) z ) \prod_j \H_{\T_j}(z).
\end{equation}
(with the usual convention that $H_{\emptyset}(t) = 0$ and that an empty product equals $1$).
\begin{lem}
The limit of $\tilde{h}(t)$ as $R \to 0$ exists and is given by 
\begin{equation}\label{fZero}
\lim_{R \to 0}\tilde{h}(t)(e_{\alpha}) = e_{\alpha}\exp \bra \alpha, - \sum_{ T} \W_{ T} \widehat{H}_{ T}(t)\ket,
\end{equation}
where
\begin{equation}\label{HZero1}
\widehat{ H}_{\T}(t) =  e_{\gamma_{\T}} * \frac{1}{2\pi i} t \int_{\ell_{\gamma_{T}}}\frac{d z}{z} \frac{1}{z - t} \exp(Z(\gamma_{\T}) z^{-1}) \prod_j \widehat{H}_{\T_j}(z).
\end{equation}

\end{lem}
\begin{proof}
We work with the alternative formula 
\begin{equation}\label{Hfunction}
\H_{\T}(t) =  e_{\gamma_{\T}} * \frac{1}{2\pi i} \int^t_{t_0} dt' \int_{\ell_{\gamma_T}}\frac{d z}{(z - t')^2} \exp(  Z(\gamma_{\T}) z^{-1} +   R^2 \bar{Z}(\gamma_{\T}) z ) \prod_j \H_{\T_j}(z).
\end{equation}
Let as assume inductively that each $\H_{\T_j}(z)$ is of order $o(|z|^{\eps})$ as $z \to \infty$ for all $\eps > 0$ (i.e. it grows less than any positive power), and that this holds uniformly in $R$ (this is certainly true for the identity). Similarly let us assume inductively that each $\H_{\T_j}(z)$ is bounded as $z \to 0$. The result will then follow applying Fubini's Theorem similarly as before. By hypothesis, one can show that the inner integral is convergent and in fact of order $O(|t'|^{-1}) \cdot o(|t'|^{\eps})$ for all $\eps > 0$ as $t' \to \infty$, uniformly as $R \to 0$. At the same time it is uniformly bounded near $t' = 0$ for all $R$. In particular $\H_{\T}(t)$ does have a well-defined $R \to 0$ limit, namely just
\begin{equation*}
\widehat{\H}_{\T}(t) = e_{\gamma_{\T}} * \frac{1}{2\pi i}\int^t_{t_0} dt' \int_{\ell_{\gamma_T}}\frac{d z}{(z - t')^2} \exp(  Z(\gamma_{\T}) z^{-1}) \prod_j \widehat{H}_{ T_j}(z).
\end{equation*}
Going back to \eqref{Hfunction}, we need to check the inductive hypotheses. But since the inner integral is $O(|t'|^{-1}) \cdot o(|t'|^{\eps})$ uniformly, $\H_{\T}(t)$ is $O(\log |t|) \cdot o(|t|^{\eps})$ for all $\eps > 0$ as $|t| \to \infty$, which is again of order $o(|t|^{\eps})$ for all $\eps > 0$. Also, the integrand is uniformly bounded as $t' \to 0$ for all $R$, so the same is true for $ H_{ T}(t)$ as $t\to 0$.
\end{proof}

\begin{lem}
We have $h(t, Z, R) = \tilde{h}(t, Z, R)$. In particular the function $h(t)$ has a finite limit in $\Aut^*(\widehat{\g})$ as $R \to 0$, which we denote by $\widehat{h}(t)$, and which is given by \eqref{fZero}, \eqref{HZero1}.
\end{lem}
\begin{proof}
Both $h(t)$ and $\tilde{h}(t)$ are solutions to \eqref{intEqu}. One can check that this implies that $h^{-1}(t)\tilde{h}(t)$ (the composition in $\Aut^*(\widehat{\g})$) is holomorphic on $\C^*$. On the other hand both $h(t)$ and $\tilde{h}(t)$ are bounded as $t \to 0$ and $t \to \infty$. For $h(t)$ this follows from the same property for $Y(R t, R Z)$, while for $\tilde{h}(t)$ it follows from \eqref{Hfunction3} above. So $h^{-1}(t)\tilde{h}(t) \in \Aut^*(\widehat{\g})$ is a constant, and since $\lim_{t \to 0} h^{-1}(t)\tilde{h}(t) = I $ by construction, it must be the identity.
\end{proof}  
\begin{remark}\label{InverseScaling3} Identical results hold for the function $h'(t) = g^{-1}(R) Y(R^{-1} t, R Z)$ introduced in Remark \ref{InverseScaling2}. The key observation is that by the definition of $g(R)$ we still have for all fixed $R > 0$
\begin{equation*}
\lim_{t\to 0} g^{-1}(R) Y(R^{-1} t, R Z) = I,
\end{equation*}
so we still compare with solutions, this time of equation \eqref{InverseIntEqu}, constructed by iteration from $(\tilde{h}')^0 = I$. This shows that the function $h'(t, Z, R)$ has a well defined limit as $R \to 0$, acting by 
\begin{equation*} 
\lim_{R \to 0}h'(t, Z, R)(e_{\alpha}) = e_{\alpha}\exp \bra \alpha, - \sum_{ T} \W_{ T} \widehat{H}'_{ T}(t)\ket,
\end{equation*}
where
\begin{equation*} 
\widehat{ H}'_{\T}(t) =  e_{\gamma_{\T}} * \frac{1}{2\pi i} t \int_{\ell_{\gamma_{T}}}\frac{d z}{z} \frac{1}{z - t} \exp(\bar{Z}(\gamma_{\T}) z ) \prod_j \widehat{H}'_{\T_j}(z).
\end{equation*}
\end{remark}
\subsection{The limit connection $\widehat{\nabla}_t$ from limit flat sections}\label{sec:limit:sub:finish} We are now in a position to set
\begin{align*}
\widehat{X}(t) = \lim_{R \to 0} g^{-1}(R) X(R t, R Z) = \lim_{R \to 0} g^{-1}(R) Y(R t, R Z) X^0(R t, R Z) = \widehat{h}(t) \widehat{X}^0(t),
\end{align*}
where
\begin{equation*}
\widehat{X}^0(t)(e_{\alpha}) = e_{\alpha} \exp_*(t^{-1} Z(\alpha) ). 
\end{equation*}
Notice that $\widehat{X}^0(t)$ is naturally a flat section for the connection
\begin{equation*}
\widehat{\nabla}^0 = d - \left( - \frac{Z}{t^2}\right)dt.
\end{equation*}
This implies that, in each sector $\Sigma$, $\widehat{X}(t, Z)$ is a flat section of the pullback connection $\widehat{\nabla}|_{\Sigma} = \widehat{h}^{-1}(t)|_{\Sigma} \cdot \widehat{\nabla}^0$. By precisely the same argument as in Section \ref{sec:gmn}, the $\widehat{\nabla}|_{\Sigma}$ for various $\Sigma$ glue to a connection on $\C^* \subset \PP^1$ in each finite-dimensional quotient $\g_{\leqslant k}$, and taking an inverse limit we find a well-defined $D^*(\widehat{\g})$-valued connection $\widehat{\nabla}$. 
\begin{lem}\label{LimitConnectionLemma} The limit connection has the form $\widehat{\nabla}_t = d - \left( - \frac{Z}{t^2} + \frac{\tilde{f}}{t}\right)dt$ for some $t$-constant derivation $\tilde{f}$.
\end{lem}
\begin{proof} Notice that $\widehat{\nabla}_t$ is in fact (a posteriori, i.e. once we know that the limit exists) the $R \to 0$ limit of the connections $g(R) \cdot \nabla_t$, which have the form
\begin{equation*}
g(R) \cdot \nabla_t (Z) = d - \left(\frac{1}{t^2}R^{-1}\tilde{\A}^{(-1)}(R) + \frac{1}{t}\tilde{\A}^{(0)}(R) + R \tilde{\A}^{(1)}(R) \right)dt
\end{equation*} 
for some $t$-constant $D^*(\widehat{\g})$-valued functions $\tilde{A}^{(i)}(R)$, for which $\lim_{R \to 0} R^{i} \tilde{A}^{(i)}(R)$ exist. 

Checking that $\lim_{R \to 0} R^{-1} \tilde{A}^{(-1)}(R) = - Z$ amounts to computing the principal part of $\widehat{\nabla}_t$. This can be computed by looking at the leading asymptotic behaviour as $t \to 0$ of flat sections in sectors. Recall that we have such a section given by $\widehat{X}(t) = \widehat{h}(t) \widehat{X}^0(t)$. The action of $\widehat{h}(t)$ is given by $\eqref{fZero}, \eqref{HZero1}$ and so we have $\widehat{h}(t) \to I$ as $t \to 0$. So $\widehat{X}(t)$ is asymptotic to $\exp(t^{-1}Z)$ as $t \to 0$. This means precisely that the principal part of $\widehat{\nabla}_t$ is $- t^{-2} Z$.

We claim that $\lim_{R \to 0} R \tilde{A}^{(1)}(R) = 0$. This would prove the Lemma, with $\tilde{f} = \lim_{R \to 0} \tilde{A}^{(0)}(R)$. To prove the claim we turn to the alternative scaling introduced in Remark \ref{InverseScaling1}. Recall that the corresponding rescaled connections take the form
\begin{equation*}
\nabla'_{t}(RZ) = d - \left(\frac{1}{t^2} R  \A^{(-1)}(R) + \frac{1}{t} \A^{(0)}(R) + R^{-1}  \A^{(1)}(R) \right)dt
\end{equation*} 
and so 
\begin{equation*}
g(R) \cdot \nabla'_{t}(RZ) = d - \left(\frac{1}{t^2} R  \tilde{\A}^{(-1)}(R) + \frac{1}{t} \tilde{\A}^{(0)}(R) + R^{-1} \tilde{\A}^{(1)}(R) \right)dt.
\end{equation*} 
According to Remarks \ref{InverseScaling2} and \ref{InverseScaling3} these connections have a well defined, finite limit as $R \to 0$, so $\lim_{R \to 0} R^{- i} \tilde{A}^{(i)}(R)$ exist and are finite. This gives no further information for $i = -1, 0$, but for $i = 1$ shows that $\lim_{R \to 0} R^{- 1} \tilde{A}^{(1)}(R)$ exists and is finite, from which it follows that $\lim_{R \to 0} R \tilde{A}^{(1)}(R) = 0$. 
\end{proof}
 
\subsection{Stokes factors at $0$ and uniqueness}\label{sec:limit:sub:Stokes} Regarding $\widehat{h}(t)$ as a formal gauge transformation (i.e. an element of $\Aut^*(\widehat{\g}[[z]])$), we see that the formal equivalence type of $\widehat{\nabla}_t(Z)$ is $d + \frac{Z}{t^2} dt$. Therefore the anti-Stokes rays of $\widehat{\nabla}(t, Z)$ are $\ell_{\alpha}, \alpha \in \Gamma$. Notice also that since $\widehat{X}(t)$ is the $R \to 0$ limit of $g^{-1}(R) X(R t, R Z)$, we have immediately for an anti-Stokes ray $\ell \subset \mathbb{H}'$ and a point $z_0 \in \ell$ 
\begin{equation*} 
\widehat{X}(z^+_0)(e_{\alpha}) = \widehat{X}(z^-_0)(e_{\alpha}) \prod_{Z(\gamma)\in \ell}(1 - \widehat{X}(z_0, Z)(e_{\gamma}))^{\Omega(\gamma) \bra \gamma, \alpha\ket}
\end{equation*}
(informally, $\widehat{X}(z^+_0) = \widehat{X}(z^-_0) \circ \prod_{\gamma \in \ell} T^{\Omega(\gamma)}_{\gamma} $). Recall that $\Sigma$ always denotes a sector bounded by consecutive anti-Stokes rays $\ell_{\gamma}, \ell_{\gamma'}$ with nontrivial $\Omega$. In order to show that $\widehat{\nabla}$ has the required Stokes factors it is enough to prove that the solution $\widehat{h}(t)|_{\Sigma}$ in a fixed sector $\Sigma$ extends to a holomorphic function in the supersector $\widehat{\Sigma}$ bounded by $e^{-i\frac{\pi}{2}}\ell_{\gamma}$ and $e^{i\frac{\pi}{2}}\ell_{\gamma'}$, with $\widehat{h}(t) \to I$ in $\widehat{\Sigma}$.   
\begin{lem} The solution $\widehat{h}(t)|_{\Sigma}$ extends to a solution in the supersector $\widehat{\Sigma}$ which tends to $I$ as $t \to 0$.
\end{lem}
\begin{proof} The solution $\widehat{h}(t)|_{\Sigma}$ is expressed in terms of the integrals $\widehat{H}_{\T}(t)$ by \eqref{fZero}. Let $T$ be any decorated tree with $\ell_{\gamma_T} \subset \widehat{\Sigma}$ and $\Omega(\gamma_T) \neq 0$. Thus the corresponding function $\widehat{H}_{\T}(t)$ is holomorphic in $\Sigma$ (because we are working in a fixed, arbitrary $\g_{\leqslant k}$, and $\Sigma$ is bounded by \emph{consecutive} rays), and it is holomorphic with a single branch-cut in $\widehat{\Sigma}$ (along the ray $\ell_{\gamma_T}$). We claim that we can choose a branch of such a $\widehat{H}_{\T}(t)$, extending its branch on $\Sigma$ given by the integral representation \eqref{HZero1}, which is holomorphic and vanishes as $t \to 0$ in $\widehat{\Sigma}$. By \eqref{fZero} this would clearly be enough to give the required extension of $\widehat{h}(t)|_{\Sigma}$. 

We argue by induction on the length of $ T$. The result certainly holds for $\widehat{H}_{\emptyset}(t) = 0$. Next look at 
\begin{equation*}
\widehat{H}_{\T}(t) =  e_{\gamma_{\T}} * \frac{1}{2\pi i} t \int_{\ell_{\gamma_\T}}\frac{d z}{z} \frac{1}{z - t} \exp(Z(\gamma_{\T}) z^{-1}) \prod_j \widehat{H}_{\T_j}(z).
\end{equation*}
By induction we assume that we have already chosen the correct branches of $\widehat{H}_{\T_j}(z)$ on $\widehat{\Sigma}$. We write $H(z)$ for the product of these analytic continuations, i.e. on $\Sigma$,
\begin{equation*}
H(z) = \prod_j \widehat{H}_{\T_j}(z).
\end{equation*}
We can assume for definiteness that $\ell_{\gamma_T}$ follows the sector $\Sigma$ in the counterclockwise order (the other case is completely similar). For fixed $t \in \widehat{\Sigma}$ we consider the integral
\begin{equation*}
e_{\gamma_{\T}} * \frac{1}{2\pi i} t \int_{\pi_t}\frac{d z}{z} \frac{1}{z - t} \exp(Z(\gamma_{\T}) z^{-1}) H(z).
\end{equation*} 
where $\pi_t \subset \widehat{\Sigma}$ is any path from $0$ to $\infty$ which lies strictly within the open sector bounded by $\R_{> 0} t$ and $\ell_{\gamma'}$. It is straightforward to check that this integral (depending on $t$) extends the definition of $\widehat{H}_{\T}(t)$ to a holomorphic function on $\widehat{\Sigma}$ which vanishes as $t \to 0$. 
\end{proof}

To prove uniqueness of the conformal limit, we use that $\widehat{\nabla}_t(Z)$ has the same formal type and Stokes factors as $\nabla^{BTL}(-Z)$. From \cite[Theorem 2]{Boalch3}, it follows that $\widehat{\nabla}(Z)$ and $\nabla^{BTL}(-Z)$ are gauge equivalent. By the particular form of the connections (second order pole at zero and no holomorphic terms), the gauge transformation taking $\nabla^{BTL}(-Z)$ to $\widehat{\nabla}(Z)$ must be $t$-constant, given by $g_0 \in \Aut^*(\widehat \g)$. Then, we have
\begin{equation*}
g_0 \tilde f(Z) g_0^{-1} = \ad f(-Z)
\end{equation*}
and therefore $\tilde f(Z)$ is necessarily a divergence-free derivation (see Section \ref{sec:gdef}). The equality $ \widehat{\nabla}(Z) = \nabla^{BTL}(-Z)$ follows now from the uniqueness part of Theorem 6.6 in \cite{bt_stokes}. 

\subsection{Symmetric case}\label{symmetricConformal} The construction of $\widehat{\nabla}(Z)$ above carries over immediately to the symmetric case described in Section \ref{symmetricRemark}, that is when we have a continuous lift $T'^{\Omega}_{\gamma}$, yielding a family of connections with values in derivations of $\g[[\tau]]$. Following the notation there the relevant integral operator takes the form
\begin{align*} 
\nonumber h(t)(e_{\alpha}) = &e_{\alpha} \exp\Big(\frac{1}{2\pi i} \int^t_{t_0} dt'\sum_{\alpha'}\bra \alpha, \alpha'\ket \tau^{p(\alpha', Z)}\dt(\alpha')\\ &\int_{\ell_{\alpha'}}\frac{d z'}{( z' - t')^2} \exp(  Z(\alpha') \z'^{-1} +   R^2 \bar{Z}(\alpha') z') h(z')(e_{\alpha'}) \Big).
\end{align*}  
The $R \to 0$ fixed point is given explicitly by
\begin{equation*} 
h(t)(e_{\alpha}) = e_{\alpha}\exp \bra \alpha, - \sum_{ T} t^{p(T, Z)}\W_{ T} \widehat{H}_{ T}(t)\ket.
\end{equation*}

\section{The $R \to \infty$ limit and tropical geometry}
\label{sec:tropical}
In this Section we give precise statements and proofs of our two results (Theorems \ref{thm:tropical1} and \ref{thm:tropical2}) which relate the $R \to \infty$ behaviour of the sum over graphs expansion \eqref{intro:X} for flat sections to tropical geometry and the Gross-Pandharipande-Siebert approach to wall-crossing. Sections \ref{sec:tropical:sub:statement}-\ref{sec:tropical:sub:curves} are devoted to Theorem \ref{thm:tropical1} and sections \ref{sec:tropical:sub:invts}-\ref{sec:tropical:sub:finish2} to Theorem \ref{thm:tropical2}.
\subsection{General setup}\label{sec:tropical:sub:statement} Recall we place ourselves in the model case when $\Gamma$ is generated by two elements $\gamma, \eta$ with $\bra \gamma, \eta\ket = \kappa > 0$ (this may be regarded as the $\kappa$-Kronecker quiver case). We will choose for definiteness a family $Z \in \mathcal{U}$ parametrised by a connected, open subset $\mathcal{U} \subset \Hom(\Gamma, \C)$ for which $Z(\gamma), Z(\eta)$ lie in the positive quadrant. We will write $Z^{\pm}$ for a point in the open subset $U^{\pm}$ of $\mathcal{U}$ where $\pm \Im Z(\gamma)/Z(\eta) > 0$, and we assume that $U^{\pm}$ are nonempty. We fix the continuous family of stability data on $\g$ parametrised by $\mathcal{U}$ and uniquely characterised by $\Omega(\gamma, Z^+) = \Omega(\eta, Z^+) = 1$, with all other $\Omega(\alpha, Z^+)$ vanishing. The locally constant function $a\!: \Gamma \to \g$ underlying the stability data in $U^+$ is given simply by $a(k\gamma) = -\frac{1}{k^2}\gamma$, $a(h\eta) = -\frac{1}{h^2}\eta$ for $h, k > 0$, with all other values vanishing. While this family may look quite simple the corresponding irregular connections $\nabla(RZ)$ are already almost as complicated as in the most general case. 

Choose a fixed $z^* \in \C^*$ with $\Re z \Im z < 0$. We consider trees $T$ such that $W_T(Z^+) \neq 0$, i.e. their vertices are decorated by positive multiples of the basic vectors $\gamma$ or $\eta$. For each tree $T$ with more than a single vertex, the special function $G_{T}(z^*, RZ)$ is only sectionally holomorphic in $Z \in \mathcal{U}$: it is discontinuous along the critical locus where $\Im Z(\gamma)/Z(\eta) = 0$. 

The idea we wish to implement is very simple: we will rewrite $G_{T}(z^*, R Z^+)$ as a sum of iterated integrals over rays $\ell(Z^-)$, similar in form to $\pm G_{T'}(z^*, R Z^-)$ for various $T'$, with the only difference that the integrands involve $X^0(z, R Z^+)$. Since $X^0$ is continuous in $Z$, $G_{T}(z^*, R Z^+)$ will be asymptotically equal to the sum of these terms $\pm G_{T'}(z^*, R Z^-)$ as $|Z^+ - Z^-| \to 0$, and this gives an effective way to see which linear combination of the special functions $G_{T'}(z^*, R Z^-)$ replaces $G_{T}(z^*, R Z^+)$ in the expansion of $X(z^*, R Z^-)$. The theorem below makes this idea precise, and characterises the single-vertex term $G_{\bullet}(z^*, R Z^-)$ in this linear combination in terms of certain tropical graphs.

\vskip.2cm
\noindent\textbf{Theorem \ref{thm:tropical1} (precise statement).} \emph{There exists an expansion} 
\begin{equation}\label{expansion}
G_{T}(z^*, R Z^+) = \sum_{T'} \pm G_{T'}(z^*, R Z^-) + r(|Z^+ - Z^-|)
\end{equation}
\emph{where $r(|Z^+ - Z^-|) \to 0$ as $|Z^+ - Z^-| \to 0$, and we sum over a finite set of rooted trees $T'$, not necessarily distinct, decorated by $\Gamma$. Let $\beta \in \Gamma$ denote the sum $\sum_i \alpha(i)$ of all decorations of $T$.} 

\emph{The terms corresponding to a single-vertex tree in \eqref{expansion} are labelled by a finite set of graphs $C_i$ containing $|T^0|$ external $1$-valent vertices and with $3$-valent internal vertices. These terms are all equal to $G_{\beta}(z^*, R Z^-)$ up to sign, and differ by a well defined factor $\eps(C_i) = \pm 1$ which is uniquely attached to the graph $C_i$.}

\emph{Moreover the graphs $C_i$ come naturally with extra combinatorial data which endows them with the structure of the combinatorial types of a finite set of rational tropical curves immersed in the plane $\R^2$.}

\emph{Finally the single-vertex terms in \eqref{expansion} are uniquely characterised by their asymptotic behaviour: they are of order}
\begin{equation*}
(2 |Z^-(\beta)| R )^{-1} \exp(- 2 |Z^-(\beta)| R)e_{\beta},
\end{equation*}
\emph{as $R \to \infty$, \emph{uniformly} as $|Z^+ - Z^-| \to 0$.}

\noindent  
\vskip.2cm

The notion of a (type of) rational tropical curve immersed in $\R^2$ is briefly reviewed in Section \ref{sec:tropical:sub:curves}.

For the sake of simplicity we will assume in the course of the proof that the lattice element $\beta$ is primitive in $\Gamma$. In this case the (types of) tropical curves $C_i$ which appear are all connected. The case when $\beta$ is not primitive is only notationally heavier, and involves disconnected $C_i$.  

Before tackling the general case, it is helpful to illustrate the asymptotic expansion for $G_T(R)$ in terms of the graphs $C_i$ starting with the simplest case when $T$ is the decorated tree with a single edge $\gamma \to \eta$. By definition, for $Z^+ \in U^+$ we have 
\begin{align*}
G_{T}(z^*, R Z^+) = &\int_{\ell_{\gamma}(Z^+)} \frac{d z_1}{z_1} \rho(z^*, z_1) X^0(z_1, Z^+)(e_{\gamma})\\ 
&\int_{\ell_{\eta}(Z^+)}\frac{dz_2}{z_2}\rho(z_1, z_2)X^0(z_2, Z^+)(e_{\eta}).
\end{align*} 
One way to compute the change of $G_T$ across the wall is to first rewrite it in terms of integrals over rays $\ell(Z^-)$ for a point $Z^- \in U^-$. Recall that $X^0(z)$ is holomorphic in $\C^*$, and that $\rho(z, z')$ is a meromorphic function of $z'$ with a simple pole at $z' = z$, with $\operatorname{Res}_{z} \rho(z, z') = (2\pi i)^{-1}$. By our choice of $z^*$ and the definitions of $U^{\pm}$, it follows that we can rewrite
\begin{align}
G_{T}(z^*, R Z^+) &= \int_{\ell_{\gamma}(Z^-)} \frac{d z_1}{z_1} \rho(z^*, z_1) X^0(z_1)(e_{\gamma}) \int_{\ell_{\eta}(Z^+)}\frac{dz_2}{z_2}\rho(z_1, z_2)X^0(z_2)(e_{\eta})\label{tropExample1}\\
\nonumber &= \int_{\ell_{\gamma}(Z^-)} \frac{d z_1}{z_1} \rho(z^*, z_1) X^0(z_1)(e_{\gamma}) \int_{\ell_{\eta}(Z^-)}\frac{dz_2}{z_2}\rho(z_1, z_2)X^0(z_2)(e_{\eta})\\
&+ \int_{\ell_{\gamma}(Z^-)} \frac{d z_1}{z_1} \rho(z^*, z_1) X^0(z_1)(e_{\gamma}) * X^0(z_1)(e_{\eta}).\label{tropExample2} 
\end{align}  
where all the $X^0$ are evaluated at $Z = Z^+$. The last term \eqref{tropExample2} comes from the residue theorem when we push $\ell_{\eta}(Z^+)$ over to $\ell_{\eta}(Z^-)$, crossing the first integration ray $\ell_{\gamma}(Z^-)$. Notice that we have the simple but crucial algebra endomorphism property
\begin{equation*}
X^0(z_1)(e_{\gamma}) * X^0(z_1)(e_{\eta}) = X^0(z_1)(e_{\gamma + \eta}).
\end{equation*}  
It is this property that allows to relate the wall-crossing behaviour of the graph integral $G_T$ to tropical curves in $\R^2$ in the general case. 

In the present example, recalling that $X^0$ is continuous across the critical locus where $\Im Z(\gamma) / Z(\eta) = 0$, we find
\begin{equation*}
G_T(z^*, R Z^+) = G_{\gamma + \eta}(z^*, R Z^-) + G_T(z^*, R Z^-) + r(|Z^+ - Z^-|)
\end{equation*}
where $r \to 0$ as $|Z^+ - Z^-|\to 0$. The single-vertex term is the unique term with asymptotic behaviour
\begin{equation*}
G_{\gamma + \eta}(z^*, R Z^-) \sim (2 |Z^-(\gamma + \eta)| R )^{-1} \exp(- 2 |Z^-(\gamma + \eta)| R)e_{\gamma + \eta}  
\end{equation*}
\emph{uniformly} for $|Z^+ - Z^-|\to 0$.

There is an obvious graph $C$ which we can attach to the computation above, displayed in Figure \ref{tripod}. There are edges $E_1, E_2$ labelled by the two factors in \eqref{tropExample1}, and $E_3$ labelled by the residue term \eqref{tropExample2}. 
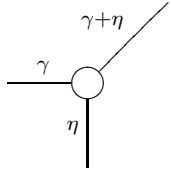
\begin{figure}[ht]
\begin{center}
\centerline{
\xymatrix {
       &      &   & \\
  \ar^{\gamma}@{-}[r]   &   *++[o][F-]{}\ar^{\gamma + \eta}@{-}[ur]  &   & \\
       &   \ar^{\eta}@{-}[u]   &   &}
}
\end{center}
\caption{The simplest tropical type.}\label{tripod}
\end{figure}
These edges meet in a single vertex $V$, and come with attached integral vectors $\alpha(E_1) = \gamma$, $\alpha(E_2) = \eta$ and $\alpha(E_3) = \gamma + \eta$. It is natural to think of $E_1, E_2$ as incoming in $V$, and $E_3$ as outgoing from $V$. Keeping track of this orientation, we have the balancing condition
\begin{equation*}
- \alpha(E_1) - \alpha(E_2) + \alpha(E_3) = 0
\end{equation*} 
due to the algebra endomorphism property of $X^0$.
\subsection{Expansion for $G_T(z^*, R Z^+)$ across the critical locus}\label{sec:tropical:sub:start} Fix a decorated tree $T$ as in Section \ref{sec:GMN:sub:fixed}, with $W_T(Z^+) \neq 0$. We will show that the simple analysis of the previous Section can be carried for $G_T(z^*, RZ)$, up to leading order terms as $R \to 0$, establishing \eqref{expansion}. The precise form of the expansion \eqref{expansion} depends on a choice of total order for the vertices of $T$. We simply fix one such total order, without assuming that it is compatible with the natural orientation of $T$ as a rooted tree (i.e. flowing away from the root).

The proof is recursive: starting from the integral $G_T(z^*, RZ)$ attached to the decorated tree $T$, we introduce an operation which produces a new finite set of graph integrals $S(T)$, also attached to decorated trees, on which the operation can be iterated, and whose sum equals $G_T(z^*, RZ)$ up to a vanishing error term as $|Z^+ - Z^-|\to 0$. We show that the procedure terminates and the resulting expansion for $G_T(z^*, RZ)$ is precisely \eqref{expansion}. 

In order to carry out the iteration process we need to consider a more general class of iterated integrals $I(T')$ which determine totally ordered, decorated rooted trees $T'$. We require that there is a bijective correspondence between vertices of $T'$ and factors of the form
\begin{equation}\label{pieceIntegral}
\int_{\ell_j} \frac{dz_j}{z_j} \rho(z_i, z_j) X^0(z_j, R Z^+)(e_{\alpha(j)})
\end{equation}  
appearing in $I(T')$, such that the factor
\begin{equation*}
\int_{\ell_i} \frac{dz_i}{z_i} \rho(z_k, z_i) X^0(z_i, R Z^+)(e_{\alpha(i)})\int_{\ell_j} \frac{dz_j}{z_j} \rho(z_i, z_j) X^0(z_j, R Z^+)(e_{\alpha(j)})
\end{equation*}  
appears if and only if there is an arrow $i \to j$ in $T'$. Moreover the integration contours must be rays $\ell_i \subset \C^*$, and we require that the following property holds: there is at most a single vertex of $i$ of $T'$ such that $\ell_i$ is not one of the rays $\ell_{\alpha(i)}(Z^{\pm})$. This property will be preserved by our iteration process. Note that the integral $I(T')$ contains more information than just the decorated tree $T'$. 

In particular $T$ is attached to the iterated integral $G_T(z^*, R Z^+)$ in this sense, and obviously satisfies the property. This will be the starting point of our iteration.
\begin{remark} In \eqref{pieceIntegral} we allow $z_i \in \ell_j$, i.e. we allow factors of the form
\begin{equation*}
\lim_{\ell' \to \ell} \int_{\ell'} \frac{d z_j}{\z_j}\rho( z_i,  z_j) X^0(z_j, Z^+)(e_{\alpha(j)})
\end{equation*} 
where $ z_{i} \in \ell$. However in this case \eqref{pieceIntegral} will be decorated with the direction in which $\ell'$ approaches $\ell$, using $\ell' \to\ell^{\pm}$ for the clockwise (respectively counterclockwise) direction.
\end{remark}

Let $T'$ be a tree which is attached to an iterated integral $I(T')$ in the sense above (e.g. the pair $T$, $G_T(z^*, R Z^+)$). We will construct from $I(T')$ a finite set of iterated integrals $S(T')$ attached to trees of the same type, obtained by applying the residue theorem. To save some space we set
\begin{equation*}
X^0_{\alpha}(z) = X^0(z, R Z^+)(e_{\alpha}).
\end{equation*}

In the following, we say that a ray $\ell$ separates $\ell_1, \ell_2$ if $\ell_1, \ell_2$ lie in different connected components of the complement of $\ell$ in the sector between $\ell_{\gamma}(Z^+), \ell_{\eta}(Z^+)$. We allow the limiting case in which $\ell_1 \to \ell$ in a component which does not contain $\ell_2$, or possibly $\ell_1 \to \ell$ and $\ell_2 \to \ell$ in different components. 

Consider the set of vertices $j \in T'$ for which one of the following occurs:
\begin{enumerate}
\item the corresponding factor in $I(T')$ has the form 
\begin{equation*}
\int_{\ell_{\alpha(j)}(Z^+)} \frac{d\z_j}{\z_j}\rho(\z_i, \z_j) X^0_{\alpha(j)}(\z_i) 
\end{equation*}
where $\alpha(j)$ is a positive multiple of $\gamma$ or $\eta$, or
\item it is of the form
\begin{equation*}
\int_{\ell} \frac{d\z_j}{\z_j}\rho(\z_{i}, \z_j) X^0_{\alpha(j)}(\z_j)
\end{equation*}
for some ray $\ell \subset \C^*$ which is not one of $\ell_{\alpha(j)}(Z^{\pm})$. 
\end{enumerate}
Note that the two conditions are obviously mutually exclusive. Recall we are assuming inductively that there is at most one vertex $i$ of $T'$ for which (2) holds. If the set of $j$ satisfying (1) or (2) is empty we simply set $S(T') = \{ T' \}$. Otherwise let us choose the single vertex $j$ satisfying (2), or the first vertex $j$ satisfying (1) with respect to the total order of $T'$ if no vertex satisfies (2). As $T$ is rooted, there is at most one arrow $i \to j$, and possibly several arrows $j \to k$.
Since $j$ satisfies (1) or (2), and in the latter case it is the only such vertex, the factor of $I(T')$ corresponding to $j$ fits into 
\begin{align*}
\int_{\ell_{\alpha(i)}(Z^{\pm})} \frac{d z_{i}}{ z_{i}}\rho( z_h, \z_i) X^0_{\alpha(i)}( z_i)&\int_{\ell} \frac{d z_j}{ z_j}\rho( z_i,  z_j)X^0_{\alpha}( z_j)\\
&\prod_{k}\int_{\ell_{\alpha(k)}(Z^{\pm})}\frac{d z_k}{ z_k}\rho( z_j,  z_k)X^0_{\alpha(k)}( z_k)
\end{align*}
where $\ell$ is either $\ell^+_{\alpha}$ or a ray which is distinct from $\ell^-_{\alpha}$, and $h \to i$. 

If none of the rays $\ell_{\alpha(i)}(Z^{\pm})$ and $\ell_{\alpha(k)}(Z^{\pm})$ separate $\ell$ and $\ell_{\alpha}(Z^-)$, we set $S(T') = \{T''\}$, with $T'' = T'$ and $I(T'')$ obtained from $I(T')$ by replacing $\ell$ in the factor above with $\ell^-_{\alpha(j)}$.

Otherwise we apply Fubini and rewrite the integral above in the form
\begin{align}\label{fubini}
\nonumber  \int_{\ell_{\alpha(i)}(Z^{\pm})} \frac{d z_{i}}{\z_{i}}\rho( z_h,  z_i) X^0_{\alpha(i)}( z_i)&\left(\prod_k\int_{\ell_{\alpha(k)}(Z^{\pm})}\frac{d z_k}{ z_k} X^0_{\alpha(k)}( z_k)\right)\\
&\int_{\ell} \frac{d z_j}{ z_j}\prod_k \rho( z_j,  z_k)\rho( z_i,  z_j) X^0_{\alpha}( z_j).
\end{align}
The function $\frac{1}{ z_j}\prod_k\rho(\z_j,  z_k)\rho( z_i,  z_j) X^0_{\alpha}( z_j)$ is holomorphic in the variable $ z_j \in \C^*\setminus \{ z_i,  z_k\}$, and has simple poles at $ z_i,  z_k$ with residues given respectively by $ - ( 2\pi i)^{-1}\rho( z_i,  z_k) X^0_{\alpha}( z_i)$ and $(2\pi i)^{-1}\rho( z_i,  z_k) X^0_{\alpha }( z_k)$. If we apply the residue theorem (justified by the estimates of integrals along an arc given in Section \ref{BesselLemma}) we can rewrite \eqref{fubini} as
\begin{align}\label{push2}
\nonumber  \int_{\ell_{\alpha(i)}(Z^{\pm})} \frac{d z_i}{ z_i}\rho( z_h,  z_i) X^0_{\alpha(i)}( z_i)&\left(\prod_k\int_{\ell_{\alpha(k)}(Z^{\pm})}\frac{d z_k}{ z_k}X^0_{\alpha(k)}( z_k)\right)\\
&\int_{\ell_{\alpha}(Z^{-})} \frac{d z_j}{ z_j}\prod_k\rho( z_j,  z_k) \rho( z_i,  z_j) X^0_{\alpha}( z_j)
\end{align}
plus residue terms
\begin{equation}\label{res2}
\mp \int_{\ell_{\alpha(i)}(Z^{\pm})} \frac{d z_i}{ z_i}\rho( z_h,  z_i) X^0_{\alpha(i) + \alpha}( z_i)\prod_{k}\int_{\ell_{\alpha(k)}(Z^{\pm})}\frac{d z_k}{ z_k}\rho(z_i, z_k) X^0_{\alpha(k)}( z_k)
\end{equation}
and
\begin{align}\label{res3}
\nonumber \pm \int_{\ell_{\alpha(i)}(Z^{\pm})} \frac{d z_i}{ z_i}\rho( z_h,  z_i) X^0_{\alpha(i)}( z_i)&\int_{\ell_{\alpha(k')}(Z^{\pm})} \frac{d z_{k'}}{ z_{k'}}\rho( z_i,  z_{k'}) X^0_{\alpha(k') + \alpha}( z_{k'})\\
&\prod_{k \neq k'}\int_{\ell_{\alpha(k)}(Z^{\pm})}\frac{d z_k}{ z_k}\rho( z_i,  z_k) X^0_{\alpha(k)}( z_k).
\end{align}
It is understood that the term \eqref{res2} is only present if $\ell_{\alpha(i)}(Z^{\pm})$ separates $\ell$ and $\ell_{\alpha}(Z^-)$, while a term \eqref{res3} appears for each $\ell_{\alpha(k')}(Z^{\pm})$ separating $\ell$, $\ell_{\alpha}(Z^-)$. The signs in \eqref{res2}, \eqref{res3} are determined according to whether $\ell$ moving to $\ell_{\alpha}(Z^{-})$ crosses $\ell_{\alpha(i)}(Z^{\pm})$ (respectively $\ell_{\alpha(k')}(Z^{\pm})$) in the clockwise, respectively counterclockwise direction.
\begin{remark} Following our convention, if $\ell_{\alpha}(Z^-)$ coincides with $\ell_{\alpha(i)}(Z^{\pm})$ or one of the $\ell_{\alpha(k)}(Z^{\pm})$ then the integral over $\ell_{\alpha}(Z^-)$ in \eqref{push2} is actually a limit of integrals over $\ell \to \ell^{\pm}_{\alpha(v)}(Z^-)$.
\end{remark}

We define trees $T''$ in $S(T')$ corresponding to the iterated integrals \eqref{push2}, \eqref{res2}, \eqref{res3}. In the case of \eqref{push2} $T''$ coincides with $T'$. In the case of \eqref{res2} $T''$ is obtained by contracting the edge $i \to j$ in $T'$ to a single vertex decorated by $\alpha(i) + \alpha$. Similarly in the case of \eqref{res3} $T''$ is obtained by contracting the edge $j \to k'$ in $T'$ to a single vertex decorated by $\alpha + \alpha(k')$. Note that for each of these trees $T''$ the condition $(2)$ can hold for at most a single vertex, as required.

Starting from our original pair of $T$ and $I(T) = G_T(z^*, R Z^+)$, by construction the sequence of sets $S(T), S(S(T)), \dots$ stabilises after a finite number of steps; we let $S^{(p)}(T)$ denote the first set for which $S^{(p)}(T) = S^{(p+1)}(T)$. 

This finishes the construction of the expansion \eqref{expansion}. Indeed $G_T(z^*, R Z^+)$ is a sum of terms which are in bijection with elements of $S^{(p)}(T)$, and these all have the form
\begin{equation*}
\pm G_{T'}(z^*, R Z^-) + r_{T'}(|Z^+ - Z^-|)
\end{equation*}
where $r_{T'}(|Z^+ - Z^-|) \to 0$ as $|Z^+ - Z^-| \to 0$. The single-vertex terms in \eqref{expansion} are in bijection with trees $T_p \in S^{(p)}(T)$ which contain a single vertex.

\subsection{Highest order terms} We characterise the single-vertex terms in \eqref{expansion} by their asymptotic behaviour. Let $T_p$ denote a tree in $S^{(p)}(T)$ with a single vertex. Then by construction
\begin{equation*}
I(T_p) = \pm \int_{\ell_{\beta}(Z^-)} \frac{d z'}{ z'} \rho( z,  z') X^0_{\beta}(z').
\end{equation*}
for a unique sign attached to $T^p$ by orientations in the residue theorem, and where $\beta \in \Gamma$ is given by the sum of all the lattice elements attached to the vertices of $T$. By the results of Section \ref{BesselLemma} we have an expansion as $R \to 0$  
\begin{equation*}
I(T_p) \sim (2 |Z^- (\beta)| R )^{-1} \exp(- 2 |Z^-(\beta)| R) e_{\beta}, 
\end{equation*}
which holds uniformly as $|Z^+ - Z^-| \to 0$. 

Suppose now that $T_2 \in S^{(p)}(T)$ is a tree which contains more than a single vertex. We claim that this is subleading, i.e. there is an expansion of the form
\begin{equation}\label{remainder}
I(T_2) \sim \phi(|Z^+ - Z^-|) f(Z^{\pm}, R) e_{\beta}  
\end{equation}
where $\phi$ is bounded and $f$ satisfies 
\begin{equation*}
f(Z^{\pm}, R) = o((2 |Z^- (\beta)| R )^{-1} \exp(- 2 |Z^-(\beta)| R)).
\end{equation*}
Indeed by construction in this case we have
\begin{equation*} 
I(T_2) = \int \prod_{\{i \to j\}\subset \widetilde{T}_2} \frac{d z_j}{ z_j} \rho(z_i, z_j) X^0_{\alpha(j)}( z_j),
\end{equation*}
where the tree $\widetilde{T}_2$ contains an extra vertex $0$ mapping to the root of $T_2$, with $z_0 = z^*$. Iterating the argument in Section \ref{BesselLemma} sufficiently many times, one can then show that \eqref{remainder} holds with $f(Z^{\pm}, R)$ bounded by 
\begin{equation*}
\prod_{i} \frac{1}{2 R |Z^-(\alpha(i))|} e^{- 2 R |Z^-(\alpha(i))|}.
\end{equation*}  
This decays faster than $(2 |Z^- (\beta)| R )^{-1} \exp(- 2 |Z^-(\beta)| R)$, because 
\begin{equation*}
\sum_{i} |Z^-(\alpha(i))| > |Z^-(\beta)|. 
\end{equation*} 
This follows immediately from our assumption that $\beta$ is primitive in $\Gamma$ (using that $Z^-$ is nondegenerate). 

\subsection{Tropical graphs attached to highest order terms} It is now a simple matter to show that a tree $T_p \in S^{(p)}(T)$ containing a single vertex determines a graph $C$ containing only $1$-valent and $3$-valent vertices, and whose edges are decorated by elements of $\Gamma$. We will show that these extra data satisfy two relations,  which imply that $C_i$ is the combinatorial type of a tropical curve in $\R^2$ (this notion will also be recalled).  

The tree $T_p$ determines a unique  sequence of trees $T_r \in S^{(i)}(T)$, $r = 0, \dots, p$, its ancestors, with $T_0 = T$. Moreover by construction there are natural maps between the set of vertices 
\begin{equation*}
\varphi_r\!:  T^{0}_r \to  T^{0}_{r+1},
\end{equation*}
such that $\varphi_r$ is either a bijection, or maps two vertices $i_1, i_2$ to the same vertex $i \in  T^{0}_{r+1}$ (and is a bijection on $ T^{0}_r \setminus \{i_1, i_2\}$). 

We define a new graph $\widetilde{C} = \widetilde{C}(T_p)$ whose vertices are given by the set of all vertices of all ancestors of $T_p$ (including $T$ and $T_p$ itself, i.e. the union of the vertices sets $\bigcup^p_{r = 0} \T^{0}_r$), and such that two vertices $x, y \in \bigcup^p_{r = 0} \T^{0}_r$ are joined by an edge if and only if for some $r$ we have $x \in T^0_{r}$, $y \in T^0_{r + 1}$ and $y = \varphi_r(x)$. Since each map $\varphi_r$ is either a bijection or identifies two points, the internal vertices of $\widetilde{C}$ are either $2$-valent or $3$-valent. Figure \ref{TildeCFigure} illustrates the construction of $\widetilde{C}$ for the example of the tree $\{\gamma + \eta\}\bullet$  treated in Section \ref{sec:tropical:sub:statement}: in the picture dashed lines denote the edges of ancestors, solid lines the edges of $\tilde{C}$, and vertices are decorated by the direction of integration rays.
\begin{figure}[ht]
\begin{center}
\centerline{
\xymatrix{
\gamma^+ \ar@{-->}[rr]\ar@{-}[d] & & \eta^+\ar@{-}[d]\\
\gamma^-  \ar@{-->}[rr]\ar@{-}[dr]  & &\eta^+\ar@{-}[dl]\\
                                       &\gamma^-\ar@{-}[d] &\\
                                       &(\gamma+\eta)^-&
}
}
\end{center}
\caption{Construction of $\widetilde{C}$.}
\label{TildeCFigure}
\end{figure}
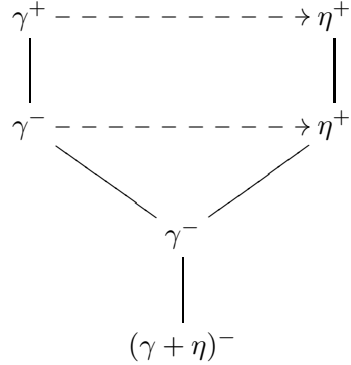 

Let $V$ be a $3$-valent vertex of $\widetilde{C}$, corresponding to a vertex of $T_r$. By construction, this determines a unique factor of the form \eqref{push2} in $I(T_{r-1})$, and $V$ corresponds to a unique nonzero residue term of the form \eqref{res2} or \eqref{res3}. This means that there there is a natural choice of incoming edges $E_1$, $E_2$, respectively an outgoing edge $E_{3}$. The edges $E_i$ come naturally with vectors $\alpha(E_i) \in \Gamma$: in the notation of \eqref{push2} - \eqref{res3} these are given by $(\alpha(i), \alpha, \alpha(i) + \alpha)$, respectively $(\alpha(i), \alpha, \alpha(k') + \alpha)$. Thus we always have the balancing condition
\begin{equation}\label{balancing}
-\alpha(E_1) - \alpha(E_2) + \alpha(E_3) = 0.
\end{equation}
Notice that the balancing condition is a direct consequence of the residue theorem and the property
\begin{equation*}
X^0(z', RZ)(e_{\alpha_1}) * X^0(z', RZ)(e_{\alpha_2}) = X^0(z', RZ)(e_{\alpha_1 + \alpha_2}).
\end{equation*}
At the same time we see that $\alpha(E_1), \alpha(E_2) \in \Gamma$ are linearly independent over $\Q$, otherwise the residue term with lattice element $\alpha(E_3) = \alpha(E_1) + \alpha(E_2)$ would not appear in \eqref{res2} - \eqref{res3}. Also, if $E$ and $E'$ are edges of $\widetilde{C}$ which are respectively outgoing and incoming to $3$-valent vertices $V, V'$, we must have $\alpha(E) = \alpha(E')$ (since no application of the residue theorem separates $V, V'$).  

Finally we define a graph $C$ obtained from $\widetilde{C}$ by forgetting all the internal $2$-valent vertices. In the case of the graph $\widetilde{C}$ depicted in Figure \ref{TildeCFigure} this produces the simplest tropical type of Figure \ref{tripod}.  

\subsection{Rational tropical curves in $\R^2$}\label{sec:tropical:sub:curves} We have attached to the special function $G_T(z^*, R Z^+)$ a finite collection of graphs $C_i$ with $3$-valent internal vertices, one for each single-vertex tree in $S^{(p)}(T)$ or, equivalently, one for each leading order term in the expansion 
\begin{equation*}
G_{T}(z^*, R Z^+) = \sum_{T'} \pm G_{T'}(z^*, R Z^-) + r(|Z^+ - Z^-|).
\end{equation*}
Each $C_i$ comes with the extra data of a decoration of its edges $E$ by elements $\alpha(E) \in \Gamma$, satisfying the above conditions of balancing \eqref{balancing} and linear independence.

The extra data say precisely that $C_i$ is the combinatorial type of a rational tropical curve immersed in $\R^2$. Following \cite{gps} Section 2.1, we define plane rational tropical curves as immersions of certain graphs in $\R^2$. 

Let $C$ denote a connected graph with only $3$-valent internal vertices. We suppose that $C$ is weighted, i.e. we have the extra data of a positive number $w(E)$ for every edge $E$ of $C$. We write $C$ as well for the topological model of the graph, and $C^o$ for the topological space obtained by removing all $1$-valent (external) vertices. A parametrised tropical curve in $\R^2$ is a proper map $h\!: C^o \to \R^2$, such that for all $E$, the map $h|_E$ is an embedding into an affine line of rational slope, and for which the following balancing condition folds. At each image of a vertex $h(V)$, we have well defined primitive vectors $ m_i \in \Z^2$ pointing out of $h(V)$ along the directions of the incident edges $E_1, E_2, E_3$. Then one requires
\begin{equation*}
w(E_1)m_1 + w(E_2)m_2 + w(E_3)m_3 = 0.
\end{equation*}
A rational plane tropical curve is then defined as the equivalence class of maps $h$ up to isomorphisms of the domain graph. Following \cite{gama} Section 2, the combinatorial type of a tropical curve is defined as the data of the underlying graph $C$, together with the vectors $m_i$ for each internal vertex $V$.

It is now clear that each of our graphs $C_i$ is the combinatorial type of a class of tropical curves. As an example, consider the tree 
\begin{equation*}
T = \{\gamma \to \eta  \to \gamma  \to 2\eta \}
\end{equation*} 
and fix the unique total order of vertices which is compatible with the orientation. Then the expansion for $G_T$ contains two leading order terms, labelled by the tropical types $C_1, C_2$ of Figures \ref{C1}, \ref{C2}.
\begin{figure}[ht]\label{tropType1}
\begin{center}
\centerline{
\xymatrix @-1pc{
                    &             &         &        &        &  &   \\
                    &             &         &        &        &     \\
                    &             &         &        &        &     \\
                    &             &         &        &    *++[o][F-]{}\ar_{2\gamma + 3\eta}@{-}[uuurr]    &     \\                    
\ar^{\gamma}@{-}[rr] &             & *++[o][F-]{}\ar^{2\gamma + \eta}@{-}[urr]         &        &        &      \\
\ar^{\gamma}@{-}[r]  & *++[o][F-]{}\ar_{\gamma + \eta}@{-}[ur]       &         &        &        &      \\
                                  & \ar^{\eta}@{-}[u]                                                                &         &        &  \ar^{2\eta}@{-}[uuu]      &  \\                                                
}
}
\end{center}
\caption{The tropical type $C_1$}\label{C1}
\end{figure}
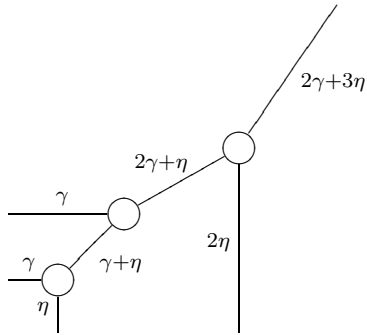

\begin{figure} 
\begin{center}
\centerline{
\xymatrix @-1pc{
                    &             &         &        &        &  &   \\
                    &             &         &        &        &     \\
                    &             &         &        &        &     \\
                    &             &         &        &   *++[o][F-]{}\ar^{2\gamma + 3\eta}@{-}[uuurr]     &     \\
                    &             &         &        &       &     \\                    
\ar^{\gamma}@{-}[rrr]  &             &           &  *++[o][F-]{}\ar_{\gamma+2\eta}@{-}[uur]      &        &      \\
\ar^{\gamma}@{-}[r]  & *++[o][F-]{}\ar^{\gamma + \eta}@{-}[uuurrr]       &         &        &        &      \\
                                  & \ar^{\eta}@{-}[u]                                                                &         &  \ar^{2\eta}@{-}[uu]      &         &  \\                                                
}
}
\end{center}
\caption{The tropical type $C_2$}\label{C2}
\end{figure}
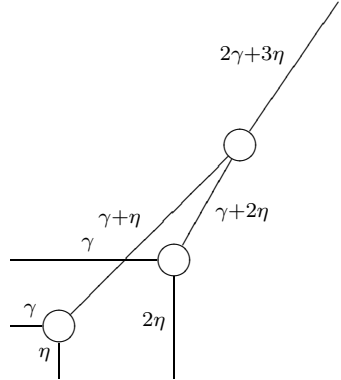 
  
\subsection{Tropical invariants}\label{sec:tropical:sub:invts} Just as for plane algebraic curves, there is a natural notion of degree for a plane rational tropical curve $(C^o, h)$ as above, see e.g. \cite{gama} Section 2: this is just the unordered collection of vectors $ - w(E_i) m_i \in \Z^2$ and $w(E_{\rm out}) m_{\rm out}$ attached to all the external edges of $C$. Notice that we allow $w(E_i) > 1$ for some or all the external edges. 

The enumerative theory of plane tropical curves of fixed degree through the expected number of general points is well established in all genera (going back to the foundational work of Mikhalkin \cite{mikhalkin}, see \cite{gama} for a result in the generality we need here). 

We will only be concerned with a very special enumerative invariant, which is described in detail in \cite{gps} Section 2.3. Choose $l_1$ general lines $\mathfrak{d}_{1j}$ with the same (positive, primitive) direction $d_1$, respectively $l_2$ general lines $\mathfrak{d}_{2j}$ in the direction $d_2$. We attach a positive integral weight $w_{ij}$ to the line $\mathfrak{d}_{ij}$. Look at the set of parametrised plane rational tropical curves $(C^o, h)$ having a collection of unbounded edges $E_{ij}, E_{\rm out}$, such that $h(E_{ij}) \subset \mathfrak{d}_{ij}$ and $w(E_{ij}) = w_{ij}$. By the balancing condition, the degree of these curves is determined by a weight vector ${\bf w} = ({\bf w}_1, {\bf w}_2)$, where each ${\bf w}_i$ is the collection of integers $w_{ij}$ (for $1 \leq i \leq 2$ and $1 \leq j \leq l_i$) such that  $1 \leq w_{i1} \leq w_{i2} \leq \dots \leq w_{i{l_i}}$. By the general theory, for generic $\mathfrak{d}_{ij}$ the number of isomorphism classes of parametrised curves $(C^o, h)$ as above is finite. Counting these tropical curves with the multiplicity of tropical geometry yields a number $N^{\rm trop}(\bf w) \in \N_{> 0}$, which is invariant under deformation of the constraints $\mathfrak{d}_{ij}$. 

Recall that the tropical multiplicity $\mu_V$ at a $3$-valent vertex $V \in h(C^o)$ with associated primitive vectors $m_i$ is defined as $|w(E_i)m_i \wed w(E_j) m_j|$ for $i \neq j$ (this is well defined by the balancing condition). The multiplicity of $(C^o, h)$ is $\prod_V \mu_V$, the product over all $3$-valent vertices. As an example $N^{\rm trop}((1,1),(1, 2)) = 8$ is computed by Figure \ref{Ntrop}. Notice that for the choice of constraints $\mathfrak{d}_{ij}$ displayed in the figure two combinatorial types appear: a curve of type $C_1$ and two curves of type $C_2$.   
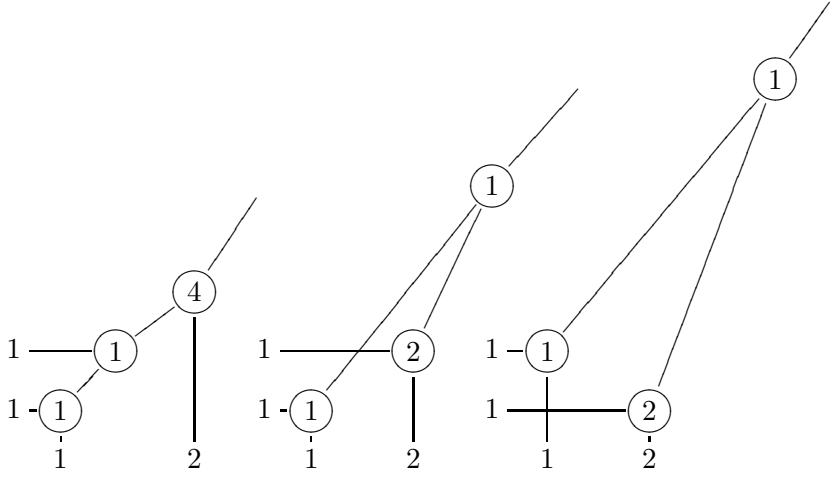
\begin{figure}[ht]
\begin{center}
\centerline{
\xymatrix @-1.8pc{
                    &             &         &        &        &  &   &                    &             &         &        &        &  & &                        &             &         &        &   & &  &  & & &\\
                    &             &         &        &        &  &   &                    &             &         &        &        &  &                        &             &         &        &   & &  &  & \\
                    &             &         &        &        &  &   &                    &             &         &        &        &  &                        &             &         &        &   & &  &  & \\
                    &             &         &        &        &  &   &                    &             &         &        &        &                   &             &         &        &   & &   & & *++[o][F-]{1}\ar@{-}[uuurr] \\
                    &             &         &        &        &  &   &                    &             &         &        &        &  &                        &             &         &        &   & &  &  & \\
                    &             &         &        &        &  &   &                    &             &         &        &        &  &                        &             &         &        &   & &  &   \\
                    &             &         &        &        &  &   &                    &             &         &        &        &   *++[o][F-]{1}\ar@{-}[uuurr]                       &             &         &        &  &      &     \\   
                    &             &         &        &        &     &                    &             &         &        &        &                         &             &         &        &  &      &     \\   
                    &             &         &        &        &     &                    &             &         &        &        &                &             &         &        &  &      &     \\   
                    &             &         &        &    *++[o][F-]{4}\ar@{-}[uuurr]    &     &                   &             &         &        &       &                         &             &         &        &  &      &     \\                      
1       \ar@{-}[rr] &             & *++[o][F-]{1}\ar@{-}[urr]         &      &      &    &1 \ar@{-}[rrrr]  &       &     & & *++[o][F-]{2}\ar@{-}[uuuurr]      &        &       1\ar@{-}[r]  & *++[o][F-]{1}\ar@{-}[uuuuuuurrrrrrr]            &           & &        &        &      \\
1 \ar@{-}[r]  & *++[o][F-]{1}\ar@{-}[ur]       &         &        &        &      & 1 \ar@{-}[r]  & *++[o][F-]{1}\ar@{-}[uuuuurrrrr]       &         &        &        &      &  1 \ar@{-}[rrrr]  &       &         &        &  *++[o][F-]{2}\ar@{-}[uuuuuuuurrrr]   &    &      \\ 
                                  & 1 \ar@{-}[u]                                                                &         &        &  2 \ar@{-}[uuu]      &  & & 1 \ar@{-}[u]    &         &  & 2 \ar@{-}[uu]      &         &          & 1 \ar@{-}[uu]                                                                &         & & 2 \ar@{-}[u]      &         &  \\                                                  
}
}
\end{center}
\caption{$N^{\rm trop}((1,1),(1, 2)) = 8$}\label{Ntrop}
\end{figure}

\begin{remark} Although we only defined the tropical invariants $N^{\rm trop}({\bf w})$ for two-components weight vectors ${\bf w}$, as explained in \cite{gps} Section 2.3, there is an obvious extension to an arbitrary number of components (with corresponding directions for the infinite ends).
\end{remark}
\subsection{Highest order terms and tropical invariants}\label{sec:tropical:sub:statement2} In the rest of this Section we will relate the (combinatorial types of) tropical curves $C_i$ constructed above to actual tropical invariants.

The $C_i$ attached to a single $T$, $G_T(z^*, R Z^+)$ all have the same tropical degree $\bf w$, which we will sometime denote by $\deg(T)$. The component ${\bf w}_1$ (${\bf w}_2$) can be identified with the set of multiples of $\gamma$ (respectively $\eta$) in the set of all decorations $\alpha(i)$ (in particular, $\bf w$ is independent of the arbitrary choice of a total order of vertices). Recall also that $C_i$ comes with a distinguished sign $\eps(C_i) = \pm 1$ (rather than a multiplicity), uniquely determined by the residue theorem through \eqref{res2} - \eqref{res3}. It is natural to consider the set of all trees $T$ defining the same degree $\bf w$, and to try and relate the sum $\sum_T \sum_i \eps(C_i(T))$ to $N^{\rm trop}({\bf w})$. Indeed we have the following result (recall that $\kappa$ denotes the value of the pairing $\bra \gamma, \eta\ket$, i.e. we are working with the $\kappa$-Kronecker quiver).
\vskip.2cm
\noindent{\textbf{Theorem \ref{thm:tropical2} (precise statement)}.} \emph{The sum over trees $T$ with $W_T(Z^+) \neq 0$ (i.e. decorated by positive multiples of $\gamma$ or $\eta$)} 
\begin{equation*}
\sum_{\deg(T) = {\bf w}} W_T \sum_i \eps(C_i(T)) 
\end{equation*}
\emph{equals the tropical invariant $N^{\rm trop}({\bf w})$, times the combinatorial factor in $c_{\bf w} \in \Gamma \otimes \Q$ given by}
\begin{equation*}
c_{\bf w} = \kappa^{l_1 + l_2} \frac{1}{|\Aut({\bf w})|}  \prod_{k, l} \frac{1}{w^2_{k l}} ( |{\bf w}|_1\gamma +  |{\bf w}|_2 \eta)
\end{equation*}
\vskip.2cm
\noindent Our proof is not direct, but relies instead on the methods of \cite{gps} Section 2.

\subsection{Tropical types and stability data}\label{sec:tropical:sub:start2} The functions $X(z, Z^{\pm}, R)$ induce flat sections of $\nabla(Z^{\pm}, R)$ on a supersector $\widehat{\Sigma}$ for $\nabla(Z^{-}, R)$, with the same asymptotics as $z \to 0$ (uniformly as $R \to \infty$). Since the connections $\nabla(Z^{\pm}, R)$ glue, choosing $z^* \in \widehat{\Sigma}$, when $|Z^+ - Z^-|\to 0$ we must have 
\begin{equation*}
X(z^*, R Z^+) - X(z^*, R Z^-) \to 0,
\end{equation*} 
uniformly as $R \to \infty$. So the same must be true for the difference
\begin{equation}\label{jump2}
\sum_T W_T(Z^+) G_T(z^*, R Z^+) - \sum_T W_T(Z^-) G_T(z^*, R Z^-).
\end{equation}

Let $T$ be a tree with $W_T(Z^+) \neq 0$ as usual. We have $W_T(Z^+) = W_T(Z^-)$ in this case. This follows since 
\begin{equation*}
\dt(h\gamma, Z^+) = \dt(h\gamma, Z^-), \quad \dt(k\eta, Z^+) = \dt(k\eta, Z^-).
\end{equation*} 
To check (for example) the first statement notice that for single-vertex trees we have
\begin{equation*}
W_{h\gamma}(Z^{\pm}) G_{h\gamma}(z^*, Z^{\pm}, R) = \dt(k\gamma, Z^{\pm}) \int_{\ell_{\gamma}(Z^{\pm})} \frac{d z}{z} \rho(z^*, z) X^0(z, Z^{\pm}(h\gamma), R). 
\end{equation*}
Therefore $W_{h\gamma}(Z^+) G_{h\gamma}(z^*, R Z^+) - W_{h\gamma}(Z^-) G_{h\gamma}(z^*, R Z^-)$ has the form 
\begin{equation*}
(\dt(h\gamma, Z^+) - \dt(h\gamma, Z^-))(2 R |Z^-(h\gamma)|)^{-1}\exp(-2R|Z^-(h\gamma)|) e_{h\gamma}
\end{equation*} 
as $|Z^+ - Z^-| \to 0$, uniformly as $R \to \infty$, and by Theorem \ref{thm:tropical1} (precise form) it cannot be cancelled by some other term in \eqref{jump2}.

Let us go back to the difference \eqref{jump2}. Pick a primitive $\beta \in \Gamma$. By the expansion \eqref{expansion}, the $e_{\beta}$ component of the first summand contains a distinguished sum of highest order terms
\begin{equation*}
\sum_{W_T(Z^+) \neq 0, \sum_i \alpha(i) = \beta} W_T \sum_i \eps(C_i(T)) \int_{\ell_{\beta}(Z^-)} \frac{d z}{z} \rho(z^*, z) X^0(z, Z^{-}(\beta), R),  
\end{equation*} 
which is uniquely characterised by its asymptotics as $|Z^+ - Z^-|\to 0$, $R \to \infty$. The unique term in the second summand of \eqref{jump2} with matching asymptotics is
\begin{equation*}
\dt(\beta, Z^-)\beta \int_{\ell_{\beta}(Z^-)} \frac{d z}{z} \rho(z^*, z) X^0(z, Z^{-}(\beta), R). 
\end{equation*}
We have proved
\begin{equation}\label{dt}
\dt(\beta, Z^-)\beta = \sum_{W_T(Z^+) \neq 0, \sum_i \alpha(i) = \beta} W_T \sum_i \eps(C_i(T)).
\end{equation}

\subsection{Refinement} Consider the set of all trees with $W_T(Z^+) \neq 0$, i.e. decorated with positive multiples of $\gamma, \eta$, and with total decoration $\sum_i \alpha(i) = \beta$. In the previous Section, we related the sum of the signs $\eps(C_i(T))$ attached to tropical types over all such trees to the stability data, that is the quantity $\dt(\beta, Z^-)$. 

Fix a weight vector $\bf w$ such that $\beta = |{\bf w}|_1\gamma +  |{\bf w}|_2 \eta$. We need to prove a more refined result, given a similar link between stability data and the sum 
\begin{equation*}
\sum_{W_T(Z^+) \neq 0, \deg(T) = {\bf w}} W_T \sum_i \eps(C_i(T)).
\end{equation*}
To achieve this we consider a larger lattice $\overline{\Gamma}$ mapping to $\Gamma$. Denoting by $l_i$ the length of ${\bf w}_i$, we take $\overline{\Gamma}$ to be generated by elements ${\gamma_1, \dots, \gamma_{\ell_1}}$ and $\eta_1, \dots, \eta_{\ell_2}$ such that 
\begin{equation*}
\bra \gamma_i, \gamma_j\ket = \bra \eta_i, \eta_j \ket = 0, \bra \gamma_i, \eta_j\ket = 1.
\end{equation*} 
The map $\pi\!:\Gamma \to \Z^2$ is given by $\pi(\gamma_i) = \gamma, \pi(\eta_j) = \eta$. There is of course a pullback family of elements of $\Hom(\overline{\Gamma}, \C)$ induced by our family $Z$; we will suppress the pullback in our notation. We look at the unique continuous family of stability data on $\g_{\overline{\Gamma}}$ which correspond to setting $\Omega(\gamma_i, Z^+) = \Omega(\eta_j, Z^+) =1$, with all other $\Omega(\alpha, Z^+)$ vanishing.

The analogues of the asymptotic expansion \eqref{expansion}, the construction of tropical types and of the argument leading to \eqref{dt} are straightforward; the only difference is that we consider now trees $\overline{T}$ which are labelled by positive multiples of $\gamma_i, \eta_j$. We still write $\sum_i \alpha(i)$ for the sum of all decorations of $\overline{T}$. Thus we have
\begin{equation*} 
\dt(\overline{\beta}, Z^-)\overline{\beta} = \sum_{W_{\overline{T}}(Z^+) \neq 0, \sum_i \alpha(i) = \overline{\beta}} W_{\overline{T}} \sum_i \eps(C_i(\overline{T})). 
\end{equation*} 
This is still not enough for our purposes. We need to impose the condition that the trees over which we sum have precisely $l_1 + l_2$ vertices. This is possible if we consider a formal version of our stability data setting $\Omega(\gamma_i, Z^+) = \Omega(\eta_j, Z^+) = \epsilon$, with all other $\Omega(\alpha, Z^+)$ vanishing. We can think of $\epsilon$ as a formal parameter or as an arbitrary rational number. The setup is unchanged, except that $W_{\overline{T}}$ and $\dt(\overline{\beta}, Z^-)$ will now be polynomials in the variable $\epsilon$. Therefore
\begin{equation}\label{dt2}
\dt(\overline{\beta}, Z^-)[\epsilon^{l_1 + l_2}]\epsilon^{l_1 + l_2}\overline{\beta} = \sum_{W_{\overline{T}}(Z^+) \neq 0, \sum_i \alpha(i) = \overline{\beta}, |\overline{T}^0| = l_1 + l_2} W_{\overline{T}} \sum_i \eps(C_i(\overline{T})), 
\end{equation} 
where $\dt(\overline{\beta}, Z^-)[\epsilon^{l_1 + l_2}]$ denotes the coefficient of the monomial $\epsilon^{l_1 + l_2}$. This is the refinement we need. Given a weight vector ${\bf w}$ as above, we construct an element $\bar{\beta}$ as
\begin{equation*}
\bar{\beta} = \sum^{l_1}_{i = 1} |w_{1i}| \gamma_i + \sum^{l_2}_{j = 1} |w_{2j}| \eta_j. 
\end{equation*}
The set of trees $\overline{T}$ such that $W_{\overline{T}}(Z^+) \neq 0$, $\sum_i \bar{\alpha}(i) = \bar{\beta}$ and $|\overline{T}^0| = l_1 + l_2$ is precisely the set $\overline{P}$ of rooted trees with $l_1 + l_2$ vertices decorated by 
\begin{equation*}
\{w_{11}\gamma_1, \dots, w_{1l_1}\gamma_{l_1}, w_{21}\eta_1, \dots, w_{2l_2}\eta_{l_2}\}.
\end{equation*}
There is a forgetful map from $\overline{P}$ to the set $P$ of rooted trees $T$ decorated by elements of $\Gamma$, with $W_T(Z^+) \neq 0$ and $\deg(T) = {\bf w}$, given by replacing $\gamma_i$ with $\gamma$ and $\eta_j$ with $\eta$. This is clearly onto. For $\overline{T}$ mapping to $T$, we have (after $\Q$-linear extension of $\pi$)
\begin{equation*}
\pi(W_{\overline{T}}) = \epsilon^{l_1 + l_2} \kappa^{-|T^0|}|\Aut(T)| W_T.
\end{equation*} 
We also have
\begin{equation*}
\eps(C_i(\overline{T})) = \eps(C_i(T)),
\end{equation*}
where the latter is computed with respect to the total order induced from $\overline{T}$. On the other hand, the fibre of $\overline{P} \to P$ over $T$ contains $(\Aut(T))^{-1}\Aut({\bf w})$ trees. Applying $\pi$ to both sides of \eqref{dt2} proves
\begin{equation}\label{dt3}
\dt(\overline{\beta}, Z^-)[\epsilon^{l_1 + l_2}] \beta  = |\Aut({\bf w})|\sum_{W_T(Z^+) \neq 0, \deg(T) = {\bf w}} \kappa^{-|T^0|}W_{T} \sum_i \eps(C_i(T)). 
\end{equation}
  
\subsection{Application of a result of \cite{gps}} In the last step of the proof we relate the stability data $\dt(\overline{\beta}, Z^-)[\epsilon^{l_1 + l_2}]\epsilon^{l_1 + l_2}$ to the tropical count $N^{\rm trop}({\bf w})$. This is where the techniques of \cite{gps} Section 2 are required. 

By its very definition, $\dt(\overline{\beta}, Z^-)[\epsilon^{l_1 + l_2}]$ admits the following description. Consider the ordered factorisation problem in $\Aut(\widehat{\g}_{\overline{\Gamma}})$ given by
\begin{equation}\label{orderedP}
\prod_{j} \Ad \exp(-\epsilon \operatorname{Li}_2( e_{\eta_j})) \prod_{i} \Ad \exp(-\epsilon \operatorname{Li}_2( e_{\gamma_i})) = \prod^{\to} \Ad \exp(-\Omega(\bar{\alpha}, Z^-)(\epsilon) \operatorname{Li}_2(e_{\bar{\alpha}}))
\end{equation}
where $\bar{\alpha} = \sum^{\ell_1}_{i = 1} a_i \gamma_i + \sum^{\ell_2}_{j = 1} b_j \eta_j$ and we are writing the operators from left to right in the clockwise order of $Z^+(\bar{\alpha}) = Z^+(\alpha)$, for $\alpha = \pi(\bar{\alpha})$. It is straightforward to check that, by the definition of $\overline{\Gamma}$, operators supported on the same ray commute (even if $Z^{-}$ is degenerate) so \eqref{orderedP} is well posed and admits a unique solution. To compute this, we compare \eqref{orderedP} with an ordered factorisation problem for automorphisms of a different algebra. As an intermediate step, let $R$ denotes the formal power series ring $R =  \C[[s_1, \dots, s_{\ell_1}, t_1, \dots, t_{\ell_2}]]$. If we notice that \eqref{orderedP} is equivalent to the factorisation problem over $\Aut(\g \otimes R)$
\begin{align}\label{orderedP2}
\nonumber &\prod_{j} \Ad \exp(-\epsilon \operatorname{Li}_2(t_j e_{\eta})) \prod_{i} \Ad \exp(-\epsilon \operatorname{Li}_2(s_i e_{\gamma}))\\
& = \prod^{\to}_{\ell} \Ad \exp\left(-\sum_{\alpha \in \ell} \dt(\alpha, Z^-)(\epsilon) e_{\alpha}\right),
\end{align}
in the following sense: $\dt(\alpha, Z^-)(\epsilon)$ will now be a polyomial in the variables $s_i, t_j$ (as well as $\epsilon$), and in fact
\begin{equation*}
\dt(\alpha, Z^-)(\epsilon) = \sum_{\pi(\bar{\alpha}) = \alpha} \dt(\bar{\alpha}, Z^-)(\epsilon) (s, t)^{\bar{\alpha}},
\end{equation*}
where we set $(s,t)^{\bar{\alpha}} = \prod_{i, j} s^{a_i}_i t^{b_j}_j$. Thus $\dt(\overline{\beta}, Z^-)[\epsilon^{l_1 + l_2}]$ appears as the coefficient of the monomial $\eps^{l_1 + l_2} (s,t)^{\bar{\beta}}$ in the polynomial $\dt(\beta, Z^-)(\epsilon)$.

We are now in a position to compare with the results of \cite{gps}. First, as in ibid. Section 1, we identify the operators appearing in \eqref{orderedP2} with (symplectic) automorphisms $\theta_{\alpha', f_{\alpha'}}$ of the power series ring $\C[\Gamma][[s_i, t_j]]$ over the group algebra $\C[\Gamma]$, acting (for primitive $\alpha'$) by
\begin{equation*}
\theta_{\alpha', f_{\alpha'}}(e_{\gamma}) = e_{\gamma} f^{\bra \alpha', \gamma \ket}_{\alpha'},\,\,\theta_{\alpha', f_{\alpha'}}(e_{\eta}) = e_{\eta} f^{\bra \alpha', \eta \ket}_{\alpha'}. 
\end{equation*}
The operators appearing on the left hand side of \eqref{orderedP2} act by
\begin{equation*}
\theta_{\eta, f_j}(e_{\gamma}) = e_{\gamma} (1 - t_j e_{\eta})^{\epsilon},\,\, \theta_{\eta, f_{\eta}}(e_{\eta}) = e_{\eta},  
\end{equation*}
respectively
\begin{equation*}
\theta_{\gamma, f_i}(e_{\gamma}) = e_{\gamma},\,\, \theta_{\gamma, f_i}(e_{\eta}) = e_{\eta}(1 - s_i e_{\gamma})^{\eps}.  
\end{equation*}
Therefore 
\begin{equation*}
\log f_j = \sum_{p\geq 1} \frac{\epsilon}{p} e_{p\eta} t^p_j,\,\,\log f_i = \sum_{q \geq 1}  \frac{\epsilon}{q} e_{q\gamma} s^q_i. 
\end{equation*}
In the notation of \cite{gps} Section 1 (p. 312), we have 
\begin{equation*}
a_{j p p} = \frac{\epsilon}{p^2},\,\, a_{i q q} = \frac{\epsilon}{q^2},
\end{equation*}
with all other $a_{j k l}, a_{i k' l'}$ vanishing. For $\alpha'$ primitive, the automorphism 
\begin{equation*}
\Ad \exp\left(-\sum_{k\geq 1} \dt(k \alpha', Z^-)(\epsilon) e_{k\alpha'}\right)
\end{equation*} 
is the same as $\theta_{\alpha', f_{\alpha'}}$, with $\log f_{\alpha'} = \sum_{k\geq 1} \dt(k \alpha', Z^-)(\epsilon) e_{k\alpha'}$. Let us go back to our $\bar{\beta} \in \overline{\Gamma}$ with $\beta = \pi(\bar{\beta})$ primitive. According to \cite{gps} Theorem 2.8, the coefficient of the monomial $(s, t)^{\overline{\beta}} e_{\beta}$ in $\log f_{\beta}$ admits a tropical description: it equals the sum
\begin{equation*}
\sum_{{\bf w}'} \frac{N^{\rm trop}({\bf w}')}{\Aut({\bf w}')} \prod_j \prod_m a_{j w'_{j m} w'_{j m}} \prod_i a_{i w'_{i n} w'_{i n}} 
\end{equation*}
summing over weight vectors ${\bf w}' = ({\bf w}'_{i n}, {\bf w}'_{j m})$, with $l_1 + l_2$ components, such that 
\begin{equation*}
\sum_n {\bf w}'_{i n} = {\bf w}_{1i},\,\, \sum_m {\bf w}'_{j m} = {\bf w}_{2 j}.
\end{equation*}   
The invariant $N^{\rm trop}({\bf w}')$ here is computed for a generic choice of constraints $\mathfrak{d}_{i n}$ with the same direction, and similarly $\mathfrak{d}_{j m}$ with the same direction. The components ${\bf w}'_{i n}$, ${\bf w}'_{j m}$ can be arbitrary increasing collections, satisfying only the condition above. However, we can refine our calculation further by looking only at the coefficient of the monomial $\epsilon^{l_1 + l_2} (s, t)^{\overline{\beta}} e_{\beta}$ in $\log f_{\beta}$. By the specific form of the coefficients $a_{j p p}, a_{i q q}$ this coefficient is given by the sum over weight vectors ${\bf w}'$ for which the collections ${\bf w}'_{i n}$, ${\bf w}'_{j m}$ contain a single element. There is precisely one such ${\bf w}'$, given by ${\bf w}' = ((w_{11}), \dots, (w_{1 l_1}), (w_{21}), \dots, (w_{2 l_2}))$. Clearly, $|\Aut({\bf w}')| = 1$, and as ${\bf w}'$ is just a subdivision of $\bf w$ (the type of constraints is the same), we have $N^{\rm trop}({\bf w}') = N^{\rm trop}({\bf w})$. Thus the coefficient of the monomial $\epsilon^{l_1 + l_2} (s, t)^{\overline{\beta}} e_{\beta}$ equals $N^{\rm trop}({\bf w}) \prod_{k, l} \frac{1}{w^2_{k l}}$. On the other hand we know already that this coefficient is precisely $\dt(\overline{\beta}, Z^-)[\epsilon^{l_1 + l_2}]$. We have proved
\begin{equation}\label{dtTrop}
\dt(\overline{\beta}, Z^-)[\epsilon^{l_1 + l_2}] = N^{\rm trop}({\bf w}) \prod_{k, l} \frac{1}{w^2_{k l}}.
\end{equation}

\subsection{Comparison}\label{sec:tropical:sub:finish2} Comparing our formulae \eqref{dt3}, \eqref{dtTrop} gives the promised connection between the tropical types attached to flat sections and actual tropical counts,
\begin{equation}\label{GMNTrop}
\sum_{\deg(T) = {\bf w}} W_{T} \sum_i \eps(C_i(T)) = \kappa^{l_1 + l_2} \frac{N^{\rm trop}({\bf w})}{|\Aut({\bf w})|}  \prod_{k, l} \frac{1}{w^2_{k l}} ( |{\bf w}|_1\gamma +  |{\bf w}|_2 \eta),
\end{equation}
where we are summing over trees $T$ with $W_T(Z^+) \neq 0$, i.e. decorated by positive multiples of $\gamma$ or $\eta$. This equality holds in $\Gamma \otimes_{\Z} \Q$. We can make the relation a bit more explicit. Indeed one has
\begin{equation*}
W_{T} = \prod_{k, l} \frac{1}{w^2_{k l}}  \frac{1}{|\Aut(T)|} \prod_{v \to w} \bra \alpha(v), \alpha(w)\ket \gamma_{T}, 
\end{equation*}
and so
\begin{align*}
&\sum_{\deg(T) = {\bf w}}  \frac{1}{|\Aut(T)|} \prod_{\{v \to w\} \subset T} \bra \alpha(v), \alpha(w)\ket \sum_i \eps(C_i(T)) \gamma_{T} \\
&= \kappa^{l_1 + l_2}\frac{N^{\rm trop}({\bf w})}{|\Aut({\bf w})|} (|{\bf w}|_1\gamma +  |{\bf w}|_2 \eta).
\end{align*}


\begin{thebibliography}
\frenchspacing\smallbreak


\bibitem[BSt]{us} A. Barbieri and J. Stoppa, \emph{Frobenius type and CV-structures for Donaldson-Thomas theory and a converge property}, arXiv:1512.01176.

\bibitem[B1]{Boalch1}       P. Boalch, \textit{Symplectic manifolds and isomonodromic deformations}, Adv. Math. 163 (2001) 137--205.
        
\bibitem[B2]{Boalch2} P. Boalch, \textit{Stokes matrices, Poisson Lie groups and Frobenius manifolds}, Invent. Math. 146, (2001) 479-506.

\bibitem[B3]{Boalch3}  P. Boalch, \textit{G-bundles, Isomonodromy and quantum Weyl groups}, Int. Math. Res. Not. (2002) 1129--1166.

\bibitem[BT1]{bt_stab} T. Bridgeland and V. Toledano-Laredo, \textit{Stability conditions and Stokes factors}, Invent. Math. 187 (2012), no. 1, 61-98.

\bibitem[BT2]{bt_stokes} T. Bridgeland and V. Toledano-Laredo, \textit{Stokes factors and multilogarithms},  J. Reine Ang. Math. 682  89--128  (2012).

\bibitem[CV1]{vafa1} S. Cecotti and C. Vafa, \emph{Topological-antitopological fusion}, Nuclear Physics B 367 (1991) 359-461.


\bibitem[Du]{dubrovin} B. Dubrovin, \emph{Geometry and integrability of topological-antitopological fusion}, Comm. Math. Phys. 152 (1993), no. 3, 539-564.

\bibitem[FIKN]{FIKN} A. Fokas, A. Its, A. Kapaev and V. Novokshenov, Painlev\'e transcendents: the Riemann-Hilbert approach, Mathematical Surveys and Monographs of the AMS, 2006.

\bibitem[Ga]{gaiotto} D. Gaiotto, \emph{Opers and TBA}, arXiv:1403.6137 [hep-th]. 

\bibitem[GMN]{gmn} D. Gaiotto, G. Moore and A. Neitzke, \emph{Four dimensional wall-crossing via three-dimensional field theory}, Comm. Math. Phys. 299 (2010), no. 1, 163-224.

\bibitem[GM]{gama} A. Gathmann and H. Markwig, \emph{The numbers of tropical plane curves through points in general position}, Journal f\"ur die reine und angewandte Mathematik 602 (2007), 155-177.

\bibitem[Ge]{gessel} I. M. Gessel, \emph{A combinatorial proof of the multivariate Lagrange inversion formula}, Jour. Comb. Theory, Series A 45, 178--195 (1987).

\bibitem[Gr]{gross} M. Gross, \emph{Mirror Symmetry and the Strominger-Yau-Zaslow conjecture}, Current Developments in Mathematics, International Press (2012).

\bibitem[GPS]{gps} M. Gross, R. Pandharipande and B. Siebert, \emph{The tropical vertex}, Duke Math. J. {\bf 153}, no. 2, 297-362 (2010).

\bibitem[H]{hert} C. Hertling, \emph{$tt^*$ geometry, Frobenius manifolds, their connections, and the construction for singularities}, J. reine angew. Math \textbf{555} (2003), 77--161.

\bibitem[J]{joyHolo} D. Joyce, \emph{Holomorphic generating functions for invariants counting coherent sheaves on Calabi-Yau 3-folds}, Geom. Topol. 11 (2007), 667-725. 

\bibitem[KS1]{ks} M. Kontsevich and Y. Soibelman, \emph{Stability structures, motivic Donaldson-Thomas invariants and cluster transformations}, arXiv:0811.2435.

\bibitem[KS2]{ks2} M. Kontsevich and Y. Soibelman, \emph{Wall-crossing structures in Donaldson-Thomas invariants, integrable systems and Mirror Symmetry}, Homological Mirror Symmetry and Tropical Geometry, Lecture Notes of the Unione Matematica Italiana (2014) 197--308.

\bibitem[Mi]{mikhalkin} G. Mikhalkin, \emph{Enumerative tropical algebraic geometry in $\R^2$}. J. Amer. Math. Soc. \textbf{18}, 313-377 (2005).

\bibitem[Mu]{musk} N. I. Muskhelishvili, \emph{Singular integral equations: boundary problems of function theory and their application to mathematical physics}, Dover, 1992.

\bibitem[R]{rein} M. Reineke, \emph{The Harder-Narasimhan system in quantum groups and cohomology of quiver moduli}, Invent. Math. 152 (2003) 349--368.

\bibitem[St]{jacopo} J. Stoppa, \emph{Joyce-Song wall-crossing as an asymptotic expansion}, Kyoto Journal of Mathematics 54 , no. 1, 103-156 (2014).

\end{thebibliography}
\end{document}